\documentclass[preprint,11pt,3p]{elsarticle}
\usepackage{color}
\usepackage{amsmath,amssymb,amsthm,amscd,mathrsfs}
\usepackage{fancyhdr}
\usepackage{geometry}
\usepackage{xcolor}
\newcommand{\tn}{\tilde{\nu}}
\numberwithin{equation}{section}
\newtheorem{theorem}{Theorem}[section]
\newtheorem*{condition}{Condition}

\newtheorem{proposition}{Proposition}[section]
\newtheorem{lemma}{Lemma}[section]
\newtheorem{definition}{Definition}[section]
\newtheorem{corollary}{Corollary}[section]
\newtheorem{remark}{Remark}[section]
\newtheorem{ex}{Example}[section]
\biboptions{sort&compress}
\journal{Elsevier}
\usepackage{times}
\usepackage[colorlinks=true]{hyperref}

\begin{document}

\begin{frontmatter}
\title{Critical Sobolev inequalities and global extremals for \\ homogeneous H\"{o}rmander vector fields on $\mathbb{R}^n$}

\author[label1]{Hua Chen\corref{cor1}}
\ead{chenhua@whu.edu.cn}
\author[label2]{Hong-Ge Chen}
\ead{hongge\_chen@whu.edu.cn}
\author[label3]{Jin-Ning Li}
\ead{lijinning@whu.edu.cn}

\address[label1]{School of Mathematics and Statistics, Wuhan University, Wuhan 430072, China}
\address[label2]{School of Mathematics and Statistics, Key Laboratory of Nonlinear Analysis \& Applications \\ (Ministry of Education), Central China Normal University, Wuhan 430079, China}
\address[label3]{College of Mathematics and Statistics, Chongqing University, Chongqing 401331, China}

\cortext[cor1]{corresponding author}

\begin{abstract}

We study critical Sobolev inequalities and Sobolev extremal functions for homogeneous H\"{o}rmander vector fields on $\mathbb R^n$.  The focus is the non-equiregular case.  In this setting, the sub-Riemannian flag may have different growth at different points, the volume of subunit balls is not governed by a single local dimension, and the translation structure available on Carnot groups is absent in general.  These features make both the range of admissible Sobolev exponents and the attainment of the optimal Sobolev constant substantially more delicate, especially on unbounded domains.

We first give a domain-dependent description of the admissible Sobolev exponents in terms of the volume growth rates of subunit balls.  This description separates the roles of the pointwise homogeneous dimension, the non-isotropic dimension of the domain, and the singular strata that may be approached at infinity.  As a consequence, we prove a global Sobolev inequality for all homogeneous H\"{o}rmander vector fields on $\mathbb R^n$, without assuming any underlying group law.

We then prove that the optimal Sobolev constant is attained. The main new geometric ingredient is a smooth family of maps $T(w,x)$ along the maximal level set $H$ of the pointwise homogeneous dimension. These maps preserve Lebesgue measure and the horizontal gradient, and they play the role of left translations in a setting where no group multiplication is available. In this sense,  $T(w,x)$ compensates for the lack of an intrinsic translation structure and allows one to recenter minimizing sequences in the critical concentration-compactness argument. The compactness theorem obtained in this way applies to all homogeneous H\"{o}rmander vector fields.  We also prove that, on every open set meeting this maximal level set, the optimal Sobolev constant is independent of the domain.

\end{abstract}
\begin{keyword}
Homogeneous H\"{o}rmander's vector fields \sep Sobolev inequality \sep non-isotropic dimension \sep extremal function.
\MSC[2020] 35J70 \sep 35H20 \sep 35R03 \sep 46E35
\end{keyword}
\end{frontmatter}

%\tableofcontents

\section{Introduction}

\subsection{Overview and motivation} 
The Sobolev inequality is one of the central estimates in analysis.  In Euclidean space it says that, for $n\geq2$ and $1\leq p<n$,
\begin{equation}\label{1-1}
  C\left(\int_{\mathbb{R}^n}|u|^{p^{*}}dx\right)^{\frac{p}{p^{*}}}
  \leq \int_{\mathbb{R}^n}|\nabla u|^{p}dx,
  \qquad p^*=\frac{np}{n-p},
  \qquad u\in C_0^\infty(\mathbb R^n).
\end{equation}
The exponent $p^*=\frac{np}{n-p}$ is the critical Sobolev exponent determined by the Euclidean dimension.  For $1<p<n$, the Sobolev inequality at this exponent is invariant under translations and dilations, and these symmetries account for the loss of compactness.  With the optimal constant, the extremals in the homogeneous Sobolev space are the Aubin--Talenti bubbles, up to these symmetries and multiplication by constants.  This structure is closely related to the role of the optimal Sobolev inequality in the Yamabe problem, especially when $p=2$; see, for instance, \cite{Lee1987}.

This paper addresses the corresponding questions for sub-Riemannian structures on $\mathbb R^n$ generated by a family of homogeneous H\"{o}rmander vector fields.  Let $X=(X_1,\ldots,X_m)$ be a family of smooth real vector fields on $\mathbb R^n$.  We say that $X$ satisfies H\"{o}rmander's condition with index $r$ if the following holds; see \cite{hormander1967}.

\begin{condition}[H.0]
\label{H0}
There exists a smallest positive integer $r$ such that $X_{1},\ldots,X_{m}$, together with all their commutators of length at most $r$, span the tangent space $T_x\mathbb R^n$ at each point $x\in\mathbb R^n$.
\end{condition}

The associated H\"{o}rmander operator
\[
  \triangle_X:=-\sum_{j=1}^{m}X_j^*X_j
\]
is a natural subelliptic analogue of the Laplacian.  It appears in sub-Riemannian geometry, CR geometry, and control theory, and its analysis is closely tied to the geometry of Carnot--Carath\'eodory balls; see, for example, \cite{Folland-Stein1974,stein1976,NSW85,Bellaiche1996,Jean2014}.  Related H\"{o}rmander vector fields also arise naturally in several complex variables, for instance in problems on CR mappings between algebraic hypersurfaces \cite{Baouendi-Huang-Rothschild1996}.  In this degenerate setting, Sobolev inequalities are not merely formal analogues of their Euclidean counterparts.  They reflect the volume growth of subunit balls and determine the critical scale for nonlinear equations.

An important model class is provided by Carnot groups, which are the nilpotent models of equiregular sub-Riemannian manifolds.  Let $X_1,\ldots,X_m$ be left-invariant horizontal vector fields spanning the first layer of a Carnot group $\mathbb G$ of homogeneous dimension $Q$.  Then, for $1<p<Q$, the Sobolev inequality takes the form
\begin{equation}\label{1-2}
 C\left(\int_{\mathbb{G}} |u|^{p_Q^*}\,dx\right)^{\frac{p}{p_Q^*}}
 \leq \int_{\mathbb{G}} |Xu|^p\,dx,
 \qquad u\in C_0^\infty(\mathbb G),
\end{equation}
where $C=C(X,p)>0$, $p_Q^*=\frac{pQ}{Q-p}$ is the critical Sobolev exponent, $Xu=(X_1u,\ldots,X_m u)$ is the horizontal gradient, and $dx$ is a Haar measure on $\mathbb G$.  This inequality goes back to the work of Folland and Stein \cite{Folland-Stein1974,Folland1975} and has since been developed in many directions; see, for example, \cite{varopoulos1992book}.  In this setting, the group law provides translations and the intrinsic dilations determine the critical exponent. The existence of extremals for the corresponding optimal inequality is also known.   In the Heisenberg case with $p=2$, the problem is closely related to the CR Yamabe problem and was studied by Jerison--Lee \cite{Jerision-Lee1984,Jerision-Lee1987,Jerision-Lee1988,Jerision-Lee1989} and Gamara--Yacoub \cite{Gamara2001,Gamara-Yacoub-2001}.  For general Carnot groups and $1<p<Q$, the existence of extremals was established by Vassilev \cite{Vassilev2006}.

The passage from Carnot groups to non-equiregular sub-Riemannian structures is not a minor change in coefficients.  It removes two structural features that underlie the classical theory.  The first is the loss of a single dimensional scale.  The flag of the sub-Riemannian
structure may vary from point to point, and hence the volume of small subunit balls is governed by a pointwise homogeneous dimension rather than by one global homogeneous dimension.  On unbounded domains, this local variation is not the only issue: one must also take into account the singular strata that can be approached by sequences escaping to infinity.  Consequently, the relevant Sobolev exponents are no longer determined by a single fixed number; rather, they form a domain-dependent interval of admissible values, even when the vector fields themselves are fixed.

The second structural feature is the availability of translations preserving the relevant energy.  On a Carnot group, left translations preserve Haar measure, the horizontal gradient norm, the subunit distance, and hence the critical Sobolev quotient.  In a general non-equiregular sub-Riemannian structure, there is no compatible group multiplication, while ordinary Euclidean translations usually change the vector fields and hence the horizontal energy.  This loss of an intrinsic recentering operation is the main obstruction in the critical variational problem.

The first difficulty is already visible in Grushin-type vector fields.  Consider
\[
 X=(\partial_{x_1},\ldots,\partial_{x_m},(\alpha+1)|x|^\alpha\partial_{y_1},\ldots,
      (\alpha+1)|x|^\alpha\partial_{y_l}),
\]
where $z=(x,y)\in\mathbb R^m\times\mathbb R^l$, $n=m+l$, and
$Q=m+l(\alpha+1)$.  When $Q>2$, the global Sobolev inequality for $p=2$ takes the form
\begin{equation}\label{1-3}
  C\left(\int_{\mathbb{R}^n}|u|^{\frac{2Q}{Q-2}}dz\right)^{\frac{Q-2}{Q}}
  \leq \int_{\mathbb{R}^n}\left(|\nabla_xu|^2+(\alpha+1)^2|x|^{2\alpha}|\nabla_yu|^2\right)dz,
  \qquad u\in C_0^\infty(\mathbb R^n).
\end{equation}
This inequality appears in the work of Loiudice \cite{Loiudice2006} and Monti--Morbidelli \cite{Monti2006}, building on Franchi--Lanconelli \cite{Franchi-Lanconelli1984}.  If one restricts instead to an exterior region away from the degeneracy set, for instance
\[
  \Omega=\{(x,y)\in \mathbb{R}^n \mid |x|>1\},
\]
then, when $n>2$, the Euclidean critical exponent $\frac{2n}{n-2}$ is also admissible.  Since Lebesgue spaces on domains of infinite measure are not ordered by inclusion, this is not a consequence of a simple embedding between $L^q$ spaces.  The natural problem is therefore to determine the full range of Sobolev exponents for a given domain.

This is the first theme of the paper.  We determine the admissible range of Sobolev exponents while keeping the dependence on the domain explicit.  The answer is formulated in terms of two sets: a geometric set recording the volume growth exponents of subunit balls, and an analytic set recording the Sobolev exponents for which the corresponding inequality holds on the given domain.  Theorems \ref{thm1}--\ref{thm3} show that these two sets are closely linked, and identify their endpoints in terms of the pointwise homogeneous dimension, the non-isotropic dimension of the domain, and the singular geometry that can be approached at infinity.  As a consequence, we obtain the global critical Sobolev inequality on $\mathbb R^n$ for homogeneous H\"{o}rmander vector fields, without assuming a group law or equiregularity.

The second theme concerns the attainment of optimal Sobolev constants.  For non-equiregular homogeneous H\"{o}rmander vector fields, the existence of extremal functions for critical Sobolev inequalities has remained largely open.  Before the present work, even for $p=2$, such attainment had been established only for some special Grushin-type models.  For the Grushin-type vector fields above, let
\begin{equation}\label{1-4}
  c_{m,l,\alpha}:=\inf_{\substack{u\in C_0^\infty(\mathbb R^n)\\
  \|u\|_{L^{2_Q^*}(\mathbb R^n)}=1}}
  \int_{\mathbb{R}^n}\left(|\nabla_xu|^2+(\alpha+1)^2|x|^{2\alpha}|\nabla_yu|^2\right)dz,
\end{equation}
where $2_Q^*=\frac{2Q}{Q-2}$.  The attainment of $c_{m,l,\alpha}$ was known only in the following situations:
\begin{itemize}
  \item Beckner \cite{Beckner2001}: $(m,l,\alpha)=(2,1,1)$;

  \item Monti--Morbidelli \cite{Monti2006,Monti2006-cpde}: $m=1$, $l\geq1$, $\alpha>0$;

  \item Dou--Sun--Wang--Zhu \cite{Dou-Sun-Wang2022}: $(m,l)\neq (2,1)$, $\alpha>0$, and $\frac{2Q}{Q-2}\in \mathbb N$.
\end{itemize}
These proofs exploit model-specific features tied to the explicit form of the Grushin-type vector fields, such as Kelvin-type transforms, hyperbolic symmetries, or symmetry reductions.  Such tools do not extend to a general non-equiregular family of homogeneous H\"{o}rmander vector fields.  In this generality the vector fields are described by homogeneity and H\"{o}rmander's condition rather than by an explicit normal form; in particular, one usually has no compatible group law, no explicit Kelvin transform, and no comparable symmetry reduction.

A natural way to attack the general problem is to use the concentration-compactness method for the critical Sobolev quotient.  The preceding obstruction shows exactly what is missing: one needs a way to recenter a minimizing sequence without changing the measure, the horizontal gradient norm, or the critical quotient.  In the absence of group translations, this recentering step has to be replaced by a construction intrinsic to the given family of vector fields.

The main geometric contribution of this paper is to construct precisely such a replacement.  We build a smooth family of maps
\[
   T:H\times\mathbb R^n\to\mathbb R^n,
\]
where $H=\{x\in\mathbb R^n\mid\nu(x)=Q\}$ is the maximal level set of the pointwise homogeneous dimension.  For each $w\in H$, the map $T(w,\cdot)$ preserves Lebesgue measure and the horizontal gradient norm, and satisfies $T(w,0)=w$.  In this precise sense, it plays the role of left multiplication in Carnot groups.  The point is not that the vector fields carry a hidden group symmetry.  Rather, the construction supplies the missing recentering operation in a setting where group multiplication and left invariance are absent.

This construction, as shown in Proposition \ref{prop4-1}, is the key point behind Theorem \ref{thm4}.  It allows one to move concentration centers along $H$ without changing the critical Sobolev quotient, and then to combine this motion with homogeneous dilations.  The resulting compactness theorem is not model-specific.  It applies to all homogeneous H\"{o}rmander vector fields and, in particular, covers the smooth Grushin-type vector fields studied by Beckner \cite{Beckner2001}, Monti--Morbidelli \cite{Monti2006,Monti2006-cpde}, and Dou--Sun--Wang--Zhu \cite{Dou-Sun-Wang2022}, without relying on Kelvin transforms, hyperbolic symmetries, or symmetry reductions.  The same construction also yields the domain-independence of the optimal Sobolev constant on every open set meeting $H$.

\subsection{Main results}

For $n\geq 2$, let  $X=(X_{1},\ldots,X_{m})$ be a family of real smooth vector fields on $\mathbb{R}^n$ satisfying 
the following conditions:
\begin{condition}[H.1]
\label{H1}
   There exists a family of non-isotropic dilations $\{\delta_{t}\}_{t>0}$ of the form
  \[ \delta_{t}:\mathbb{R}^n\to \mathbb{R}^n,\qquad \delta_{t}(x)=(t^{\alpha_{1}}x_{1},t^{\alpha_{2}}x_{2},\ldots,t^{\alpha_{n}}x_{n}), \]
 where $1=\alpha_{1}\leq \alpha_{2}\leq\cdots\leq \alpha_{n}$ are positive integers, such that $X_{1},\ldots,X_{m}$ are $\delta_{t}$-homogeneous of degree $1$. Specifically, for all $t>0$, $f\in  C^{\infty}(\mathbb{R}^n)$, and $j = 1, \ldots, m$,
 \[ X_{j}(f\circ \delta_{t})=t(X_{j}f)\circ \delta_{t}. \]
\end{condition}

\begin{condition}[H.2]
\label{H2}
    The vector fields $X_{1},\ldots,X_{m}$ are linearly independent in $\mathcal{X}(\mathbb{R}^n)$ as linear differential operators and satisfy  H\"{o}rmander's condition at $0\in \mathbb{R}^{n}$, i.e.,
       \[ \dim\{Y(0)\mid Y\in \text{\rm Lie}(X)\}=n, \]
 where $\text{\rm Lie}(X)$ is the Lie algebra generated by $X_1,\ldots,X_m$, and  $Y(0)$ denotes the value of the vector field $Y$ at the origin.
\end{condition}

The vector fields $X=(X_{1},\ldots,X_{m})$ satisfying conditions (\hyperref[H1]{H.1}) and (\hyperref[H2]{H.2}) is usually referred to as homogeneous H\"{o}rmander vector fields.  Proposition \ref{prop2-3} indicates that $X$ satisfy H\"{o}rmander's condition (\hyperref[H0]{H.0}) on $\mathbb R^n$ with H\"{o}rmander index $\alpha_n$.  The number
\begin{equation}\label{1-5}
  Q:=\sum_{j=1}^{n}\alpha_j
\end{equation}
is called the homogeneous dimension associated with $X$. A result of \cite{Biagi-Bofiglioli-ccm2015} implies that if $\dim {\rm Lie}(X)>n$, then $X$ cannot be left-invariant under any group law.

Conditions (\hyperref[H1]{H.1}) and (\hyperref[H2]{H.2})  retain the anisotropic dilations familiar from Carnot groups but do not require a
compatible group law.  The resulting class of homogeneous H\"{o}rmander vector fields contains both the left-invariant horizontal vector fields on Carnot groups and standard non-equiregular models, such as the Grushin, Bony, and Martinet vector fields; see, for example, \cite{Bonfiglioli2007}.  Homogeneous H\"{o}rmander vector fields of this type also arise as nilpotent approximations of H\"{o}rmander vector fields in privileged coordinates; see \cite{Jean2014,Bellaiche1996,Folland1977}.  The associated sum-of-squares operators, known as the homogeneous H\"{o}rmander operators, constitute a broad class of degenerate operators  arising in sub-Riemannian geometry.

We now introduce the dimension parameters presented in the statements.  For each $x\in\mathbb R^n$, let $V_j(x)$ be the $j$-th term of the sub-Riemannian flag, namely the span at $x$ of all commutators of length at most $j$, and let $\nu_j(x)=\dim V_j(x)$. Let  $r(x)$ denote the degree of nonholonomy at $x$ and set $\nu_0(x)=0$, the pointwise homogeneous dimension is
\[
   \nu(x)=\sum_{j=1}^{r(x)} j\bigl(\nu_j(x)-\nu_{j-1}(x)\bigr).
\]
For an open set $\Omega$, we write
\[
   \tilde{\nu}=\max_{x\in\overline\Omega}\nu(x)
\]
for the non-isotropic dimension of $\overline\Omega$.  When $\Omega$ is unbounded, we also use the asymptotic projection
\[
  \Pi_\infty(\Omega)=\bigcap_{R>0}\overline{\left\{\delta_{1/d(x)}(x)\mid x\in\Omega,\ d(x)>R\right\}}
  \subset \partial B(0,1),
\]
where $d(x)=d(x,0)$ is the subunit distance to the origin, and the corresponding asymptotic singular dimension
\[
   \nu_{\rm sing}:=\max_{z\in\Pi_\infty(\Omega)}\nu(z).
\]
For bounded $\Omega$ we set $\nu_{\rm sing}=-\infty$.  These quantities are defined and discussed in Section \ref{Section2}.  They are included here because they are precisely the invariants which enter the lower endpoint of the range of Sobolev exponents.

We then define the two sets which connect geometry and Sobolev inequalities.

\begin{definition}
\label{def-1-1}
Let $\Omega$ be an open subset of $\mathbb{R}^n$. 
\begin{itemize}
  \item The volume growth rate set $\mathcal{I}(\Omega)$ is 
\begin{equation}\label{1-6}
  \mathcal{I}(\Omega):=\left\{\kappa>0~\bigg|~\inf_{r>0,~x\in \overline{\Omega}}\frac{|B(x,r)|}{r^{\kappa}}>0 \right\},
\end{equation}
where $B(x,r)$ is the subunit ball associated with the subunit metric $d$, and $|B(x,r)|$ denotes its Lebesgue measure. We also define the interior volume growth set
\begin{equation*}
  \mathcal{I}_{\mathrm{int}}(\Omega) := \left\{ \kappa > 0 \;\middle|\; \inf_{x\in\Omega,\, 0<r<d(x,\Omega^{c})} \frac{|B(x,r)|}{r^{\kappa}} > 0 \right\},
\end{equation*}
with the convention  $d(x,\varnothing)=+\infty$. 

  \item  The admissible endpoint set $\mathcal{S}(\Omega)$ is 
\begin{equation}\label{1-7}
  \mathcal{S}(\Omega):=\left\{1<\kappa\leq Q \mid \exists C>0~\text{such that}~\|u\|_{L^{\frac{\kappa}{\kappa-1}}(\Omega)}\leq C  \|Xu\|_{L^1(\Omega)}~~\forall u\in C_0^\infty(\Omega)\right\}.
\end{equation}
Thus $\kappa\in \mathcal{S}(\Omega)$, while $\frac{\kappa}{\kappa-1}$ is the corresponding endpoint Sobolev exponent.
\end{itemize}
\end{definition}

Our first result relates the two sets $\mathcal I(\Omega)$ and $\mathcal S(\Omega)$.
  \begin{theorem}
    \label{thm1} 
    Let $X=(X_1,\ldots,X_m)$ be homogeneous H\"{o}rmander vector fields satisfying conditions (\hyperref[H1]{H.1}) and (\hyperref[H2]{H.2}). If $\Omega$ is an open subset of $\mathbb R^n$, then
  \begin{equation}
    \mathcal{I}(\Omega)\subset  \mathcal{S}(\Omega)\subset \mathcal{I}_{\mathrm{int}}(\Omega).
  \end{equation}
  \end{theorem} 

Theorem \ref{thm1} should be read as a localized volume-growth and Sobolev principle.  For Carnot groups the corresponding dimension is a single number.  Here, on the contrary, the relevant dimension may depend on the domain and may range over an interval.

To state the structure of this interval, we use automorphisms of the vector fields.

\begin{definition}
\label{def-1-2}
Let  ${\rm Diff}(\mathbb{R}^n)$ be the group of all smooth diffeomorphisms of $\mathbb{R}^n$. A smooth diffeomorphism  $\mathscr{A}\in {\rm Diff}(\mathbb{R}^n)$ is an automorphism of  $X=(X_1,\ldots,X_m)$ if 
\begin{equation}\label{eq-1-15}
X_{j}(f\circ \mathscr{A})=(X_{j}f)\circ \mathscr{A}\qquad \forall~ 1\leq j\leq m,~~f\in C^{\infty}(\mathbb{R}^n). 
\end{equation}
We denote by $\mathcal G$ the group of volume-preserving automorphisms:
 \begin{equation}\label{group-G}
\mathcal{G}:=\{\mathscr{A}\in {\rm Diff}(\mathbb{R}^n)\mid \mathscr{A}~\text{is an automorphism of $X$ and}~~|\det(J_{\mathscr{A}}(x))|=1\quad\forall x\in \mathbb{R}^n\},
 \end{equation}
where  $J_{\mathscr{A}}(x)$ denotes the Jacobian matrix of $\mathscr{A}$ at $x$.
\end{definition}
\begin{remark}
The group $\mathcal G$ is nontrivial.  Indeed, by Proposition \ref{prop2-2}, every coordinate direction $e_j$ with
$\alpha_j=\alpha_n$ is translation-invariant.  Hence $x\mapsto x+w$ belongs to $\mathcal G$ for every
$  w\in\operatorname{span}\{e_j\mid\alpha_j=\alpha_n\}$.
\end{remark}

The next theorem describes the volume-growth set. 
\begin{theorem}
\label{thm2}
Under the assumptions of Theorem \ref{thm1},  $\mathcal{I}(\Omega)$ has the following properties:
  \begin{enumerate}[(\text{I}1)]
    \item For every $\mathscr{A}\in \mathcal{G}$ and $t>0$, 
   \[ \mathcal{I}(\Omega)=\mathcal{I}(\mathscr{A}(\Omega))=\mathcal{I}(\delta_{t}(\Omega)).\] 
        \item $\mathcal{I}(\Omega)$ is a non-empty interval of the form either $(\inf\mathcal{I}(\Omega),Q]$ or $[\inf\mathcal{I}(\Omega), Q]$, where the
infimum satisfies
            \[ \inf\mathcal{I}(\Omega)\geq \tilde{\nu}. \]
   Here, $\tilde{\nu}$ denotes the non-isotropic dimension of $\overline{\Omega}$ associated with  $X$, as recalled above and defined in \eqref{2-3}.       
             
    \item If $0\in \overline{\Omega}$, then \[ \mathcal{I}(\Omega)=\{Q\}.\] In particular, $\mathcal{I}(\mathbb{R}^n)=\{Q\}$.
    \item \label{I4} If the asymptotic singular
dimension (see \eqref{def-asd})  satisfies $\nu_{\rm sing}=\max_{z \in \Pi_\infty(\Omega)} \nu(z) \le \tilde{\nu}$, then 
 \[\mathcal{I}(\Omega)=[\tilde{\nu}, Q].\]
In particular, if $\Omega$ is bounded, then $\Pi_\infty(\Omega)=\varnothing$, and we adopt the notation $\nu_{\rm sing}=-\infty$.
  \end{enumerate} 
\end{theorem}

The lower endpoint of $\mathcal I(\Omega)$ can be delicate.  It need not be an integer, and it need not belong to $\mathcal I(\Omega)$; see Example \ref{example3-1}.  This is why the statement is naturally formulated in terms of sets rather than a single effective dimension.

The corresponding statement for the admissible endpoint set is as follows.

\begin{theorem}
\label{thm3}
Under the assumptions of Theorem \ref{thm1}, $\mathcal{S}(\Omega)$ satisfies the following properties:
  \begin{enumerate}[(S1)]
    \item For every $\mathscr{A}\in \mathcal{G}$ and $t>0$, 
    \[
    \mathcal{S}(\Omega) = \mathcal{S}(\mathscr{A}(\Omega)) = \mathcal{S}(\delta_t(\Omega)).   \]
    \item $\mathcal{S}(\Omega)$ is a non-empty interval of the form either $(\inf \mathcal{S}(\Omega), Q]$ or $[\inf \mathcal{S}(\Omega), Q]$, where the infimum satisfies
    \[
    \inf \mathcal{S}(\Omega) \geq \max_{x \in \Omega} \nu(x).
    \]
Here, $\nu(x)$ denotes the pointwise homogeneous dimension at $x$, as recalled above and defined in \eqref{2-1} below.
    \item \label{S3} 
Suppose there exists a point $x_0 \in \overline{\Omega}$ with $\nu(x_0) = \tilde{\nu} := \max_{x \in \overline{\Omega}} \nu(x)$ satisfying the following interior corkscrew condition with respect to the subunit metric: there are constants $c \in (0, 1)$ and $r_0 > 0$ such that for any $0 < r < r_0$, there exists a point $y_r\in \Omega$ satisfying
$$B(y_r, c r) \subset \Omega \cap B(x_0, r).$$
Then $\inf \mathcal{S}(\Omega) \geq \tilde{\nu}$.

\item\label{S4} If $\partial\Omega$ is locally $C^1$ near $x_0\in\partial\Omega$ and $x_0$ is non-characteristic, namely $X_j(x_0)\notin T_{x_0}(\partial\Omega)$ for some $1\leq j\leq m$, then $x_0$ satisfies the corkscrew condition in (\hyperref[S3]{S3}).  Consequently, if also $\nu(x_0)=\tilde{\nu}$, then $\inf\mathcal S(\Omega)\geq\tilde{\nu}$.
  \item If $0\in\overline\Omega$ and there is an open set $U\subset\Omega$ such that $\delta_t(U)\subset\Omega$ for every $0<t\leq1$, then the origin satisfies the corkscrew condition in (\hyperref[S3]{S3}), and
  \[
      \mathcal S(\Omega)=\{Q\}.
  \]
  In particular, $\mathcal S(\mathbb R^n)=\{Q\}$.
  \item\label{S6} If $\nu_{\rm sing}\leq\tilde{\nu}$ and the corkscrew condition in (\hyperref[S3]{S3}) holds at a point $x_0\in\overline\Omega$ with $\nu(x_0)=\tilde{\nu}$, then
  \[
      \mathcal S(\Omega)=[\tilde{\nu},Q].
  \]
 
  \end{enumerate}
\end{theorem}

The following immediate consequences illustrate the scope of the preceding results.

\begin{corollary}
\label{corollary1-1}
Under the assumptions of Theorem \ref{thm1}, the following statements hold:
\begin{enumerate}
  \item [(1)] If $\mathbb R^n$ is equiregular with respect to $X$, then
  \[   \mathcal{I}(\Omega)=\mathcal{S}(\Omega)=\{Q\}. \]
  \item [(2)] If $\Omega=\mathbb R^n\setminus K$ is an exterior domain, where $K$ is compact, then \[\mathcal{I}(\Omega)=[\tn,Q]\subset \mathcal{S}(\Omega).\]
       Moreover, $\mathcal{S}(\Omega)=[\tn,Q]$ if one of the following conditions is satisfied:
\begin{enumerate}[(2.1)]
  \item $\Omega=\mathbb{R}^{n}\setminus \{0\}$;
  \item the corkscrew condition in (\hyperref[S3]{S3}) holds;
  \item the interior $K^\circ$ contains the origin.
\end{enumerate}
 \item [(3)]
If $\Omega$ is an open star-shaped domain with respect to the origin, then 
\[\mathcal{I}(\Omega)=\mathcal{S}(\Omega)=\{Q\}. \]

\end{enumerate}
\end{corollary}

Theorems \ref{thm1}--\ref{thm3} give a precise and domain-dependent form of the endpoint Sobolev inequality.
\begin{remark}
Let $\Omega$ be a bounded open subset of $\mathbb{R}^n$. Theorem \ref{thm1} and Theorem \ref{thm2} yield that
\begin{equation}\label{1-10}
  \left(\int_{\Omega}|u|^{\frac{\tn}{\tn-1}}dx\right)^{\frac{\tn-1}{\tn}}\leq C\int_{\Omega}|Xu|dx\qquad \forall~ u\in C_{0}^{\infty}(\Omega).
\end{equation}
 Furthermore, Theorem \ref{thm3} shows that, if there exists a point $x_0 \in \overline{\Omega}$ with $\nu(x_0) = \tilde{\nu}$ satisfying the interior corkscrew condition described in (\hyperref[S3]{S3}), the Sobolev exponent $\frac{\tn}{\tn-1}$ is optimal  (i.e., $\frac{\tn}{\tn-1}$ is the largest exponent such that \eqref{1-10} holds).  Inequality \eqref{1-10} coincides with our recent result in \cite[Theorem 1.1]{chen-chen-li2024}.
\end{remark}

\begin{remark} 
Let $\Omega$ be an unbounded open subset with infinite measure. If the asymptotic singular
dimension satisfies $\nu_{\rm sing}=\max_{z \in \Pi_\infty(\Omega)} \nu(z) \le \tilde{\nu}$, then Theorem \ref{thm1} and Theorem \ref{thm2} imply that
\begin{equation}\label{1-11}
  \left(\int_{\Omega}|u|^{\frac{q}{q-1}}dx\right)^{\frac{q-1}{q}}\leq C\int_{\Omega}|Xu|dx\qquad \forall~ u\in C_{0}^{\infty}(\Omega)
\end{equation}
holds for any $\tn\leq q\leq Q$.  Moreover, Theorem \ref{thm3} shows that if there exists a point $x_0 \in \overline{\Omega}$ with $\nu(x_0) = \tilde{\nu}$ satisfying the interior corkscrew condition described in (\hyperref[S3]{S3}), then $\frac{\tn}{\tn-1}$ is the largest Sobolev exponent for \eqref{1-11} to hold. 
\end{remark}

The endpoint inequality for $p=1$ also yields the usual $L^p$ Sobolev inequalities.  Applying it to suitable powers of $u$, one obtains in particular the following global inequality on $\mathbb R^n$.

\begin{corollary}\label{corollary1-2}
Under the assumptions of Theorem \ref{thm1}, let $Q\geq 3$ and $1\leq p<Q$. Then, there exists a constant $C=C(X,p)>0$ such that 
\begin{equation}\label{1-12}
  C\left(\int_{\mathbb{R}^n}|u|^{p_{Q}^{*}}dx\right)^{\frac{p}{p_{Q}^{*}}}\leq \int_{\mathbb{R}^n}|Xu|^{p}dx \qquad \forall~ u\in \mathcal{M}_{X,0}^{p,Q}(\mathbb{R}^n),
\end{equation}
where $\mathcal{M}_{X,0}^{p,Q}(\mathbb{R}^n)$ denotes the completion of  $C_{0}^{\infty}(\mathbb{R}^n)$ with respect to the norm 
\[ \|u\|_{\mathcal{M}_{X,0}^{p,Q}(\mathbb{R}^n)}:=\|u\|_{L^{p_{Q}^{*}}(\mathbb{R}^n)}+\|Xu\|_{L^p(\mathbb{R}^n)},\] and $p_{Q}^{*}:=\frac{pQ}{Q-p}$ is the critical Sobolev exponent.
\end{corollary}

\begin{remark}
Inequality \eqref{1-12} extends the Carnot-group Sobolev inequality \eqref{1-2} to homogeneous H\"{o}rmander vector fields on $\mathbb R^n$, without assuming an underlying group law.
\end{remark}

We now turn to the optimal Sobolev constant. Let $C_{0}=C_{0}(X,p)>0$ be the optimal constant  in the Sobolev inequality \eqref{1-12}:
\begin{equation}\label{1-13}
  C_{0}:=\inf_{u\in \mathcal{M}_{X,0}^{p,Q}(\mathbb{R}^n),~\|u\|_{L^{p_{Q}^{*}}(\mathbb{R}^n)}=1}\int_{\mathbb{R}^n}|Xu|^{p}dx.
\end{equation}
The following theorem is the compactness theorem behind the existence of extremals.

\begin{theorem}
\label{thm4}
Let $X=(X_1,\ldots,X_m)$ be a family of homogeneous H\"{o}rmander vector fields on $\mathbb R^n$ satisfying conditions (\hyperref[H1]{H.1}) and (\hyperref[H2]{H.2}).  Assume that $Q\geq 3$ and $1<p<Q$.  Let
\[
   H:=\{x\in\mathbb R^n\mid \nu(x)=Q\}
\]
be the maximal level set of the pointwise homogeneous dimension $\nu(x)$.  Then $H$ is a smooth embedded submanifold of $\mathbb R^n$, and there exists a map
\[
   T:H\times\mathbb R^n\to\mathbb R^n,
\]
depending only on $X$, such that:
\begin{enumerate}[(1)]
 \item $T\in C^\infty(H\times\mathbb R^n)$;
 \item for each $w\in H$, the map $T(w,\cdot)$ belongs to the group $\mathcal G$ (see \eqref{group-G}  above) and satisfies $T(w,0)=w$.
\end{enumerate}
Moreover, for every minimizing sequence $\{v_k\}_{k=1}^{\infty}\subset \mathcal M_{X,0}^{p,Q}(\mathbb R^n)$ for the variational problem \eqref{1-13} satisfying
\[
  \|v_k\|_{L^{p_Q^*}(\mathbb R^n)}=1,
  \qquad
  \|Xv_k\|_{L^p(\mathbb R^n)}^p\to C_0,
\]
there exists a sequence $\{(w_k,\rho_k)\}_{k=1}^{\infty}\subset H\times(0,\infty)$ such that the sequence $\{v_k^{w_k,\rho_k}\}_{k=1}^{\infty}$ is relatively compact in $\mathcal M_{X,0}^{p,Q}(\mathbb R^n)$, where
\[
  v^{w,\rho}(x):=\rho^{\frac{Q-p}{p}}v(T(w,\delta_\rho(x))).
\]
\end{theorem}

Theorem \ref{thm4} is the main compactness result of the paper.  It gives compactness of normalized minimizing sequences modulo the two natural defects of compactness: homogeneous dilations and motion along $H$ through the maps $T(w,\cdot)$.  In the Carnot group case these maps are left translations; in the general homogeneous H\"{o}rmander setting they are constructed from the vector fields and may be nonlinear.

As a consequence of Theorem \ref{thm4}, we obtain the existence and qualitative properties of extremals, extending the results of Garofalo--Vassilev \cite[Theorem 6.1]{Garofalo-Vassilev-2000} and Vassilev \cite[Theorem 3.1]{Vassilev2006}.

\begin{corollary}
\label{corollary1-3}
Under the assumptions of Theorem \ref{thm4}, the critical quasilinear degenerate elliptic equation 
\begin{equation}\label{1-1-16}
\sum_{j=1}^{m}X_{j}(|Xu|^{p-2}X_{j}u)=-u^{p_{Q}^{*}-1}\qquad \mbox{in}~~\mathbb{R}^n,
\end{equation}
admits a positive least-energy weak solution (ground state) $u_0\in \mathcal{M}_{X,0}^{p,Q}(\mathbb{R}^n)\cap C(\mathbb{R}^n)$. More precisely, $u_0$ minimizes the associated energy among all nontrivial nonnegative weak solutions of \eqref{1-1-16}, and its $L^{p_Q^*}(\mathbb R^n)$-normalization is an extremal for the variational problem \eqref{1-13}. Furthermore, $u_0\in L^{q}(\mathbb{R}^n)$ for any $q\in (\frac{Q(p-1)}{Q-p},\infty]$ and 
\begin{equation}\label{decay}
  u_0(x)\approx d(x)^{\frac{p-Q}{p-1}}\qquad \mbox{as}~~~d(x)\to+\infty. 
\end{equation}
In particular, if $p=2$, then $u_0\in \mathcal{M}_{X,0}^{p,Q}(\mathbb{R}^n)\cap C^{\infty}(\mathbb{R}^n)$. 
\end{corollary}

\begin{remark}
By Proposition \ref{prop2-8}, the following statements are equivalent under conditions (\hyperref[H1]{H.1}) and (\hyperref[H2]{H.2}):
\begin{enumerate}
  \item [(1)]The level set $H=\{x\in \mathbb{R}^n\mid\nu(x)=Q\}$ coincides with the entire space $\mathbb{R}^n$.
  \item [(2)] $\mathbb R^n$ is equiregular with respect to $X$.
\end{enumerate}
\end{remark}

\begin{remark}
\label{remark1-7}
It follows from \cite{Biagi-Bofiglioli-ccm2015,Bonfiglioli2004} that, if
$\dim\operatorname{Lie}(X)=n$, then $\mathbb R^n$ can be endowed with a
Carnot group structure $\mathbb G=(\mathbb R^n,\circ)$ of step $\alpha_n$ for which $X_1,\ldots,X_m$ are left-invariant and whose Lie algebra is $\operatorname{Lie}(X_1,\ldots,X_m)$.  Since Carnot groups
are equiregular, one has $H=\mathbb R^n$.  In this case, the map $T$ in
Theorem \ref{thm4} can be chosen to be the group multiplication,
\[
  T(w,x):=w\circ x,
\]
so that $T(w,\cdot)$ is the left translation by $w$.  Consequently,
Theorem \ref{thm4} recovers Vassilev's compactness theorem for minimizing
sequences on Carnot groups \cite[Theorem 3.1]{Vassilev2006}.
\end{remark}

Finally, we record a domain-independence property of the optimal Sobolev constant.
\begin{theorem}
\label{thm5}
Assume that conditions (\hyperref[H1]{H.1}) and (\hyperref[H2]{H.2}) hold, and let $Q\geq 3$ and $1\leq p<Q$.  Let $\Omega\subset\mathbb R^n$ be an open set such that $\Omega\cap H\neq\varnothing$.  Define
\begin{equation}\label{1-15}
  \widetilde{C_{0}}(\Omega):=\inf_{u\in \mathcal{M}_{X,0}^{p,Q}(\Omega),~\|u\|_{L^{p_{Q}^{*}}(\Omega)}=1}\int_{\Omega}|Xu|^{p}dx.
\end{equation}
Then $\widetilde{C_{0}}(\Omega)=C_0$ (as defined by \eqref{1-13}).  In particular, the optimal Sobolev constant is independent of the choice of the open set $\Omega$, provided that $\Omega\cap H\neq\varnothing$.
\end{theorem}

\begin{remark}
For $p=2$, equation \eqref{1-1-16} has the Yamabe critical exponent $2_Q^*-1=\frac{Q+2}{Q-2}$. 
Accordingly, the variational problem \eqref{1-13} may be viewed as the homogeneous H\"{o}rmander model for the sub-Riemannian Yamabe theory; in the Heisenberg group, this is precisely the model
underlying the CR Yamabe problem
\cite{Jerision-Lee1987,Jerision-Lee1988}.

Theorem \ref{thm4} gives a geometric description of the loss of compactness for normalized minimizing sequences: homogeneous dilations determine the bubbling scale, whereas the volume-preserving automorphisms $T(w,\cdot)$ move concentration centers along the level set $H$. Moreover, Theorem \ref{thm5} shows that every open set intersecting $H$ has the same sharp Sobolev constant $C_0$. This identifies $H$ as the natural concentration locus for minimizing sequences.

As shown in the proof of Corollary \ref{corollary1-3}, the ground state $u_0$ has energy $\frac{1}{Q} C_0^{\frac{Q}{p}}$. 
Accordingly, this value represents the energy carried by a single ground-state bubble and is therefore a natural candidate for the local bubbling threshold in a non-equiregular sub-Riemannian Yamabe-type problem associated with the homogeneous H\"{o}rmander vector fields $X$.
\end{remark}

\subsection{Some examples}

We conclude the introduction with several examples illustrating the scope of
conditions (\hyperref[H1]{H.1})--(\hyperref[H2]{H.2}) and the possible
forms of the maps $T$.  The Heisenberg group gives the equiregular model.
The Grushin, Bony, and Martinet examples are non-equiregular but admit
ordinary translations along $H$.  The final two examples show that $T$
may instead be an affine shear or a genuinely nonlinear map.

\begin{ex}
Let $\mathbb H^N=(\mathbb R^{2N+1},\circ)$ be the Heisenberg group.  We
write points as $\xi=(x,y,z)$ and $\eta=(x',y',z')$, where
$x,y,x',y'\in\mathbb R^N$ and $z,z'\in\mathbb R$.  The group law is
\[
  \xi\circ\eta
  :=
  \bigl(
    x+x',\,y+y',\,
    z+z'-2(\langle x,y'\rangle-\langle x',y\rangle)
  \bigr),
\]
where $\langle\cdot,\cdot\rangle$ denotes the Euclidean inner product on
$\mathbb R^N$.  The standard left-invariant horizontal vector fields are
\[
  X_j=\partial_{x_j}+2y_j\partial_z,
  \qquad
  Y_j=\partial_{y_j}-2x_j\partial_z,
  \qquad j=1,\ldots,N.
\]
The associated dilations are $\delta_t(\xi)=\delta_t(x,y,z)=(tx,ty,t^2z)$, 
and the homogeneous dimension is $Q=2N+2$.  The vector fields $  X=(X_1,\ldots,X_N,Y_1,\ldots,Y_N)$ satisfy conditions (\hyperref[H1]{H.1}) and (\hyperref[H2]{H.2}).
Moreover, $H=\mathbb R^{2N+1}$, and the map in Theorem \ref{thm4} can be
chosen as $T(w,\xi)=w\circ\xi$. 
\end{ex}

\begin{ex}
\label{ex1}
Let $m,l\geq1$, let $\alpha$ be a positive even integer, and consider on
$\mathbb R^{n}$ the smooth Grushin-type vector fields
\[
  X=
  \bigl(
    \partial_{x_1},\ldots,\partial_{x_m},
    (\alpha+1)|x|^\alpha\partial_{y_1},\ldots,
    (\alpha+1)|x|^\alpha\partial_{y_l}
  \bigr),
\]
where $x\in\mathbb R^m$ and $y\in\mathbb R^l$ such that $n=m+l$. We
write points as $z=(x,y)\in \mathbb{R}^n$. Conditions
(\hyperref[H1]{H.1}) and (\hyperref[H2]{H.2}) hold for the dilations
\[
\delta_t(z)=\delta_t(x,y)=(tx,t^{\alpha+1}y).
\]
The homogeneous dimension is $Q=m+l(\alpha+1)$, and
\[
 H=\{z\in\mathbb R^{n}\mid\nu(z)=Q\}
   =\left\{z\in\mathbb{R}^{n}\mid z=(0,y),~y\in \mathbb{R}^{l}\right\}. \]
For $w\in H$, the map in Theorem \ref{thm4} can be chosen as $T(w,z)=z+w$. 
\end{ex}

\begin{remark}
By Theorem \ref{thm4}, the critical Grushin $p$-Sobolev constant is attained for all $m,l\geq 1$, all positive even integers $\alpha$, and $1<p<Q$, where $Q=m+l(\alpha+1)$. In particular, when $p=2$, the infimum $c_{m,l,\alpha}$ in \eqref{1-4} is attained.

The initial version of this paper (\href{https://arxiv.org/abs/2506.16125v1}{arXiv:2506.16125v1}, Theorem \ref{thm4}) first established the relative compactness of normalized minimizing sequences, modulo homogeneous dilations and translations by elements of $H$, for a class of homogeneous H\"{o}rmander vector fields including the smooth Grushin fields. Thus, in the Grushin case, where $H=\{z\in\mathbb{R}^{n}\mid z=(0,y),~y\in \mathbb{R}^{l}\}$, that theorem already yielded attainment for every positive even integer $\alpha$. This result was subsequently cited by Gandal--Loiudice--Tyagi \cite{Gandal-Loiudice-Tyagi2026}. Their work concerns only the Grushin setting, where they extended the attainment conclusion to every real number $\alpha>0$ and $1<p<Q$. Their proof follows the compactness argument developed in  \href{https://arxiv.org/abs/2506.16125v1}{arXiv:2506.16125v1}. In particular, the key localization step in their proof, which shows that any possible concentration point $z_0$ must lie in $H$, is the Grushin specialization of the argument given there. Once this is established, the half-mass normalization, implemented by homogeneous dilations and translations in the $y$-variables, rules out the remaining point-concentration alternative. 
\end{remark}

\begin{ex}
Let $n\geq2$ and consider the Bony-type vector fields on $\mathbb R^n$
\cite{Bony1969}:
\[
  X_1=\partial_{x_1},
  \qquad
  X_2=x_1\partial_{x_2}
      +x_1^2\partial_{x_3}
      +\cdots
      +x_1^{n-1}\partial_{x_n}.
\]
They satisfy conditions (\hyperref[H1]{H.1}) and
(\hyperref[H2]{H.2}) for the dilations
\[
  \delta_t(x)=\delta_t(x_1,\ldots,x_n)
  =(tx_1,t^2x_2,\ldots,t^nx_n).
\]
The homogeneous dimension is $ Q=\frac{n(n+1)}{2}$, 
and 
\[ H=\{x\in\mathbb R^n\mid\nu(x)=Q\}=\left\{x\in\mathbb{R}^{n}\mid x=(0,x_{2},\ldots,x_n)\right\}.\] 
For $w\in H$, one may choose
$T(w,x)=x+w$. 
  \end{ex}

\begin{ex}
Consider the Martinet-type vector fields on $\mathbb R^3$:
\[
  X_1=\partial_{x_1},
  \qquad
  X_2=\partial_{x_2}+x_1^2\partial_{x_3}.
\]
They satisfy conditions (\hyperref[H1]{H.1}) and
(\hyperref[H2]{H.2}) for the dilations
\[
  \delta_t(x)=\delta_t(x_1,x_2,x_3)=(tx_1,tx_2,t^3x_3).
\]
The homogeneous dimension is $Q=5$, and
\[  H=\{x\in \mathbb{R}^{3}\mid\nu(x)=Q\}=\left\{x\in\mathbb{R}^{3}\mid x=(0,x_{2},x_3)\right\}. \] 
For $w\in H$, one may again choose $T(w,x)=x+w$. 
\end{ex}

\begin{ex}
Consider the vector fields on $\mathbb R^4$:
\[
  X_1=\partial_{x_1},
  \qquad
  X_2=x_1\partial_{x_2}+x_1x_2\partial_{x_3},
  \qquad
  X_3=x_1^3\partial_{x_3},
  \qquad
  X_4=x_1^3\partial_{x_4}.
\]
They satisfy conditions (\hyperref[H1]{H.1}) and
(\hyperref[H2]{H.2}) for the dilations
\[
  \delta_t(x)=\delta_t(x_1,x_2,x_3,x_4)
  =(tx_1,t^2x_2,t^4x_3,t^4x_4).
\]
The homogeneous dimension is $Q=11$, and
\[  H=\{x\in\mathbb R^4\mid\nu(x)=Q\}= \{x\in \mathbb{R}^4\mid x=(0,x_2,x_3,x_4)\}.
\]
For $w=(0,w_2,w_3,w_4)\in H$, the map in Theorem \ref{thm4} may be
chosen as
\[
  T(w,x)
  =
  (x_1,x_2+w_2,x_3+w_3+w_2x_2,x_4+w_4).
\]
For fixed $w$, this map is an affine shear; when $w_2\neq0$, it is not a
translation.
\end{ex}

\begin{ex}
Consider the vector fields on $\mathbb R^3$:
\[
  X_1=\partial_{x_1}-x_2^2\partial_{x_3},
  \qquad
  X_2=\partial_{x_1}+\partial_{x_2}
      +(x_1-x_2)^2\partial_{x_3}.
\]
They satisfy conditions (\hyperref[H1]{H.1}) and
(\hyperref[H2]{H.2}) for the dilations
\[
 \delta_t(x)=\delta_t(x_1,x_2,x_3)=(tx_1,tx_2,t^3x_3).
\]
The homogeneous dimension is $Q=5$, and
\[
  H=\{x\in\mathbb R^3\mid\nu(x)=Q\}
   =\{x\in \mathbb{R}^3\mid x=(0,x_2,x_3)\}.
\]
For $w=(0,w_2,w_3)\in H$, the map in Theorem \ref{thm4} may be chosen as
\[ T(w,x)=(x_1,x_2+w_2,x_3+w_3-2x_{1}x_{2}w_{2}-x_{1}w_{2}^{2}+2x_{2}^{2}w_{2}+2x_{2}w_{2}^{2}). \]
For $w_2\neq0$, the map $T(w,\cdot)$ is genuinely nonlinear.
\end{ex}

\subsection{Outline of the proofs}
We briefly describe the main ideas of the proofs.

The proof of Theorem \ref{thm1} combines global heat-kernel estimates for homogeneous H\"{o}rmander operators due to Biagi--Bramanti \cite{Biagi2020} with the global ball--box theorem of Biagi--Bonfiglioli--Bramanti \cite{Biagi-Bonfiglioli-Bramanti2019}.  These results give uniform control of subunit balls through the Nagel--Stein--Wainger polynomial
\[
   \Lambda(x,r)=\sum_I |\lambda_I(x)|r^{d(I)}.
\]
The inclusion $\mathcal I(\Omega)\subset\mathcal S(\Omega)$ follows from heat-kernel estimates and semigroup arguments in the spirit of Baudoin \cite{Baudoin2017} and Varopoulos et al. \cite{varopoulos1992book}, adapted here to arbitrary open sets.  The reverse inclusion
$\mathcal S(\Omega)\subset\mathcal I_{\rm int}(\Omega)$ is obtained by testing the Sobolev inequality on cut-off functions adapted to interior subunit balls.

Theorems \ref{thm2} and \ref{thm3} require finer information on the polynomial structure of homogeneous H\"{o}rmander vector fields.  Homogeneity gives
\[
   |B(\delta_t(x),tr)|=t^Q|B(x,r)|,
\]
while volume-preserving automorphisms preserve both the subunit distance and Lebesgue measure.  The homogeneous components of the Nagel--Stein--Wainger polynomial then control the lower volume growth uniformly over a domain.  This yields the interval structure of $\mathcal I(\Omega)$ and $\mathcal S(\Omega)$.  The quantities $\nu(x)$, $\tilde\nu$, and $\nu_{\rm sing}$ describe, respectively, the pointwise homogeneous dimension, the largest such dimension on the domain, and the singular geometry that can be approached at infinity.  The corkscrew condition turns the presence of a point with dimension $\tilde\nu$ into an obstruction to any better Sobolev exponent.

The proof of Theorem \ref{thm4} has a geometric part and an analytic part.  The geometric part is the construction of the maps $T$ in Proposition \ref{prop4-1}.  The Lie algebra generated by the homogeneous H\"{o}rmander vector fields is finite-dimensional and nilpotent, and integrates to a simply connected nilpotent Lie group acting smoothly on $\mathbb R^n$.  By studying the stabilizer of the origin, its normalizer, and the associated quotient, we identify
$ H=\{x\in\mathbb R^n\mid\nu(x)=Q\}$ as an embedded smooth submanifold and construct, for every $w\in H$, a volume-preserving automorphism $T(w,\cdot)$ satisfying $T(w,0)=w$.  Consequently, the required replacement for left translations is derived from the algebraic structure of the vector fields.

The analytic part is a concentration-compactness argument in
$\mathcal M_{X,0}^{p,Q}(\mathbb R^n)$.  Given a minimizing sequence, we choose $w_k\in H$ and $\rho_k>0$ so that a fixed amount of its $L^{p_Q^*}$ mass lies in $B(w_k,\rho_k)$, and then recenter and rescale it by
\[
   v_k(x)\mapsto
   \rho_k^{\frac{Q-p}{p}}
   v_k\bigl(T(w_k,\delta_{\rho_k}(x))\bigr).
\]
Since $T(w_k,\cdot)$ and $\delta_{\rho_k}$ preserve the critical Sobolev quotient, the normalized sequence remains minimizing.  A refined concentration-compactness argument, together with the local Sobolev inequalities established in \cite{chen-chen-li2024}, rules out loss of mass at infinity and point concentration.  Local higher-integrability excludes point concentration away from $H$, whereas the half-mass normalization rules out concentration at a single 
point of $H$. Consequently, no point mass can occur, and the normalized sequence converges strongly to an extremal.

The same construction is used in the proof of Theorem \ref{thm5}.  A compactly supported near-minimizer for the whole-space constant is moved by $T(w,\cdot)$ to a point $w\in\Omega\cap H$ and then dilated into a subunit ball compactly contained in $\Omega$.  This gives
$\widetilde C_0(\Omega)\leq C_0$, while the reverse inequality follows immediately from the inclusion of the corresponding test spaces.

\subsection{Organization of the paper}

The remainder of the paper is organized as follows.  Section \ref{Section2} collects the necessary preliminaries on homogeneous H\"{o}rmander vector fields, Carnot--Carath\'eodory geometry, and the function spaces used below.  Section \ref{Section3} analyzes the sets $\mathcal I(\Omega)$ and $\mathcal S(\Omega)$ and establishes the results relating volume growth to Sobolev inequalities.  Section \ref{Section4} treats the critical variational problem, including the construction of the maps $T$, the compactness of minimizing sequences, and the domain-independence of the optimal Sobolev constant.

\subsection{Notation and terminology}

Throughout this paper,  $f(x)\approx g(x)$ indicates that
$C^{-1}g(x)\leq f(x)\leq Cg(x)$, where $C>0$ is a constant independent of the relevant variables in $f(x)$ and $g(x)$.  Moreover, for simplicity, different positive constants are usually denoted by $C$ sometimes without indices.

\section{Preliminaries}
\label{Section2}

\subsection{Basic settings for H\"{o}rmander vector fields}
Let $X=(X_1,\ldots,X_m)$ be a family of smooth vector fields on $\mathbb R^n$ satisfying H\"{o}rmander's condition (\hyperref[H0]{H.0}) with index $r$.   For each $x\in \mathbb{R}^n$ and $1\leq j\leq r$, let $V_{j}(x)$ be the subspace of $T_{x}(\mathbb{R}^n)$ spanned by all commutators of $X_{1},\ldots,X_{m}$ with length at most $j$ and set $V_0(x):=\{0\}$. Then, these subspaces form a flag of sub-Riemannian structure at $x$, i.e., 
\[ V_{0}(x)\subset V_{1}(x)\subset \cdots \subset V_{r(x)-1}(x)\subsetneq V_{r(x)}(x)=T_{x}(\mathbb{R}^n), \] 
where $r(x):=\min\{j\mid V_j(x)=T_x\mathbb R^n\}\leq r$  is called the degree of nonholonomy at $x$. Define $\nu_{j}(x):=\dim V_{j}(x)$ for $1\leq j\leq r$, with $\nu_{0}(x):=0$. The pointwise homogeneous dimension at $x$ is defined as
\begin{equation}\label{2-1}
\nu(x):=\sum_{j=1}^{r(x)}j(\nu_{j}(x)-\nu_{j-1}(x)).
\end{equation}
Note that $\nu(\cdot):\mathbb{R}^n\to \mathbb{N}^{+}$ is upper semi-continuous, whereas for each $1\leq j\leq r$,  the function $\nu_{j}(\cdot):\mathbb{R}^n\to \mathbb{N}^{+}$ is  lower semi-continuous (see \cite[p. 20]{Jean2014}).

A point $x\in\mathbb R^n$ is called regular if, for every $1\leq j\leq r$, the function $\nu_j$ is constant in a neighborhood of
$x$; otherwise, $x$ is called singular.  An open set $U\subset\mathbb R^n$ is called equiregular if every point of $U$ is
regular; otherwise, it is called non-equiregular. In the PDE literature, equiregularity is also known as the M\'etivier condition, introduced by M\'etivier in \cite{Metivier1976}. If $U$ is connected and equiregular, then $\nu(x)$ is constant on $U$; its common value, denoted by $\nu$, equals the Hausdorff dimension of $U$ with respect to the Carnot--Carath\'eodory metric and is often called the M\'etivier index in PDE.

Let $\Omega\subset\mathbb R^n$ be a nonempty open set.  We define its
non-isotropic dimension by
\begin{equation}\label{2-3}
  \tilde{\nu}:=\max_{x\in \overline{\Omega}}\nu(x).
\end{equation}
This quantity records the largest pointwise homogeneous dimension attained on $\overline\Omega$ and is one of the dimensional parameters governing the volume growth of subunit balls and the associated Sobolev inequalities. It has also been referred to as the generalized M\'etivier index in the PDE literature; see \cite{Yung2015,chen-chen2019,chen-chen2020}.

On the other hand, the H\"{o}rmander vector fields $X$ induce the following intrinsic metric:
\begin{definition}[Subunit metric]
\label{def2-1}
For any $x,y\in \mathbb{R}^n$ and $\delta>0$, let $C(x,y,\delta)$ be the set
of absolutely continuous curves $\varphi:[0,1]\to \mathbb{R}^n$ such that
$\varphi(0)=x$, $\varphi(1)=y$, and 
\[ \varphi'(t)=\sum_{i=1}^{m}a_{i}(t)X_{i}(\varphi(t)),\qquad\sum_{k=1}^{m}|a_{k}(t)|^2\leq \delta^2\]
 for a.e. $t\in [0,1]$. The subunit metric is 
    \[ d(x,y):=\inf\{\delta>0 \mid \exists \varphi\in C(x,y,\delta)\}. \]
\end{definition}
The Chow--Rashevskii theorem (see \cite[Theorem 57]{Bramanti2014}) ensures the well-definedness of the subunit metric.  Given any $x\in \mathbb{R}^n$ and $r>0$, we denote by
    \[ B(x,r):=\{y\in \mathbb{R}^n~|~d(x,y)<r\} \]
 the subunit ball associated with subunit metric $d$.

\subsection{Related results on homogeneous H\"{o}rmander vector fields}

We  begin with a brief review of the definitions and properties of $\delta_t$-homogeneous functions and vector fields. One can refer to \cite{Bonfiglioli2007} for a more detailed discussion.

\begin{definition}[$\delta_t$-homogeneous function]
\label{def2-4}
A real-valued function $f$ defined on $\mathbb{R}^n$ is said to be
 $\delta_{t}$-homogeneous of degree $\sigma\in \mathbb{R}$ if $f\not\equiv 0$ and satisfies 
\[ f(\delta_{t}(x))=t^{\sigma}f(x)\qquad\forall~ x\in \mathbb{R}^{n},~t>0. \]  
\end{definition}

Definition \ref{def2-4} implies that any non-zero function $f\in C(\mathbb{R}^n)$ which is $\delta_{t}$-homogeneous of degree $\sigma$ must have a non-negative degree $\sigma\geq 0$. Continuous functions that are $\delta_{t}$-homogeneous of degree $0$ are exactly the non-zero constants (see \cite[p. 33]{Bonfiglioli2007}). Additionally, as shown in \cite[Proposition 1.3.4]{Bonfiglioli2007}, a function $f\in C^{\infty}(\mathbb{R}^{n})$ is $\delta_{t}$-homogeneous of degree $\sigma\in \mathbb{N}$ if and only if it is a polynomial function of the form
\begin{equation}\label{2-4}
  f(x)=\sum_{\sum_{i=1}^{n}\alpha_{i}\beta_{i}=\sigma}c_{\beta_{1},\ldots,\beta_{n}}x_{1}^{\beta_{1}}x_{2}^{\beta_{2}}\cdots x_{n}^{\beta_{n}},
\end{equation}
where  $\beta_{1},\beta_{2},\ldots,\beta_{n}$ are non-negative integers, and $c_{\beta_{1},\ldots,\beta_{n}}\neq 0$ for some $\beta_{1},\ldots,\beta_{n}$.

\begin{definition}[$\delta_t$-homogeneous vector field]
\label{def2-5}
A non-zero smooth vector field $Y$ on $\mathbb R^n$ is said to be
$\delta_t$-homogeneous of degree $\sigma\in\mathbb R$ if
\[ Y(\varphi\circ\delta_t)(x)=t^{\sigma}(Y\varphi)(\delta_{t}(x))\qquad \forall~ \varphi\in C^{\infty}(\mathbb{R}^{n}),~ x\in \mathbb{R}^{n},~ t>0. \]
\end{definition}

Smooth $\delta_{t}$-homogeneous vector fields have the following properties.

\begin{proposition}[{\cite[Proposition 1.3.5, Remark 1.3.7]{Bonfiglioli2007}}]
\label{prop2-2}
Let $Y$ be a smooth non-zero vector field on $\mathbb{R}^{n}$ given by
\[ Y=\sum_{j=1}^{n}\mu_{j}(x)\partial_{x_{j}}. \]
Then $Y$ is $\delta_{t}$-homogeneous of degree $\sigma\in \mathbb{N}$ if and only if each $\mu_{j}\not\equiv 0$ is a $\delta_{t}$-homogeneous polynomial of degree $\alpha_{j}-\sigma$, as in \eqref{2-4}. Moreover, for $\mu_{j}\not\equiv 0$, we have $\alpha_{j}\geq\sigma$, and $Y$ can be written as
\[ Y=\sum_{j\leq n,~ \alpha_{j}\geq \sigma}\mu_{j}(x)\partial_{x_{j}}. \]
In particular, if $\sigma\geq 1$, since each $\mu_{j}\not\equiv 0$ is a $\delta_{t}$-homogeneous polynomial function of degree $\alpha_{j}-\sigma$, \eqref{2-4} implies that $\mu_{j}(x)=\mu_{j}(x_{1},\ldots,x_{j-1})$.
\end{proposition}

For the homogeneous H\"{o}rmander vector fields, we have the following propositions.

\begin{proposition}[{\cite[Proposition 2.5]{chen-chen-li-2022}}]
\label{prop2-3}
Let $X=(X_{1},\ldots,X_{m})$ be homogeneous H\"{o}rmander vector fields satisfying
conditions (\hyperref[H1]{H.1}) and (\hyperref[H2]{H.2}). Then, $X$ satisfy the H\"{o}rmander condition (\hyperref[H0]{H.0}) on $\mathbb{R}^n$ with H\"{o}rmander index $\alpha_n$.
\end{proposition}

Moreover, the subunit metric $d$ is invariant under the volume-preserving automorphisms of $X$.

\begin{proposition}
\label{prop2-4}
For any $\mathscr{A}\in \mathcal{G}$ and $x,y\in \mathbb{R}^n$, we have
 $d(\mathscr{A}(x),\mathscr{A}(y))=d(x,y)$.  
\end{proposition}

\begin{proof}
For $x,y\in \mathbb{R}^n$ and $\delta>0$, let $C(x,y,\delta)$ the set of absolutely continuous curves $\varphi:[0,1]\to \mathbb{R}^n$, which satisfies $\varphi(0)=x,\varphi(1)=y$ and
$ \varphi'(t)=\sum_{i=1}^{m}a_{i}(t)X_{i}(\varphi(t)) $
with $\sum_{k=1}^{m}|a_{k}(t)|^2\leq \delta^2$ for a.e. $t\in [0,1]$. Now, for any $\mathscr{A}\in \mathcal{G}$, it follows from \eqref{eq-1-15} that 
\begin{equation}\label{eq-2-5}
 J_{\mathscr{A}}(x)\cdot X_{i}(x)=X_{i}(\mathscr{A}(x))\qquad \forall~ x\in \mathbb{R}^n,
\end{equation}
where $J_{\mathscr{A}}(x)$ denotes the Jacobian matrix of $\mathscr{A}$ at $x$. Let $\psi(t):=\mathscr{A}(\varphi(t))$.
 Then, $\psi(0)=\mathscr{A}(x),\psi(1)=\mathscr{A}(y)$, and \eqref{eq-2-5} gives that
\[\begin{aligned}
 \psi'(t)&=J_{\mathscr{A}}(\varphi(t))\cdot\varphi'(t)=\sum_{i=1}^{m}a_{i}(t)J_{\mathscr{A}}(\varphi(t))\cdot X_{i}(\varphi(t))\\
 &=\sum_{i=1}^{m}a_{i}(t)X_{i}(\mathscr{A}(\varphi(t)))=\sum_{i=1}^{m}a_{i}(t)X_{i}(\psi(t)),
 \end{aligned} \]
with $\sum_{k=1}^{m}|a_{k}(t)|^2\leq \delta^2$ for a.e. $t\in [0,1]$. Thus, $\mathscr{A}(\varphi)\in C(\mathscr{A}(x),\mathscr{A}(y),\delta)$, which implies that  $d(\mathscr{A}(x),\mathscr{A}(y))\leq d(x,y)$. Since $\mathscr A^{-1}\in\mathcal G$, applying the same argument to $\mathscr A^{-1}$ gives $d(x,y)
  = d(\mathscr A^{-1}(\mathscr A(x)), \mathscr A^{-1}(\mathscr A(y)))\leq d(\mathscr A(x),\mathscr A(y))$. Hence, $d(\mathscr A(x),\mathscr A(y))=d(x,y)$.
\end{proof}

\begin{proposition}[{\cite[(2.2)]{Biagi-Bonfiglioli-Bramanti2019}}]
 \label{prop2-5} 
 For the homogeneous H\"{o}rmander vector fields $X$, the subunit metric  and subunit ball satisfy the following properties:
   \begin{enumerate}[(1)]
   \item For any $x,y\in \mathbb{R}^n$ and $t>0$, $d(\delta_{t}(x),\delta_{t}(y))=td(x,y)$.
\item For any $x,y\in \mathbb{R}^n$ and $t,r>0$, $y\in B(x,r)$ if and only if $\delta_{t}(y)\in B(\delta_{t}(x),tr)$.
\item For any $t,r>0$ and $x\in \mathbb{R}^n$, $|B(\delta_{t}(x),tr)|=t^{Q}|B(x,r)|$. In particular, $|B(0,t)|=|B(0,1)|t^Q$ for $t>0$.
\end{enumerate}
\end{proposition}

We then recall some notations from \cite{NSW85} to describe the volume of the subunit ball. For $1\leq j_{i}\leq m$, let $J=(j_{1},\ldots,j_{k})$ denote a multi-index of length
$|J|=k$. The associated commutator $X_{J}$ is defined as
\[ X_{J}:=[X_{j_{1}},[X_{j_{2}},\ldots[X_{j_{k-1}},X_{j_{k}}]\ldots]]. \]
For $k\geq 1$, let $X^{(k)}:=\{X_{J}|J=(j_{1},\ldots,j_{k}),~1\leq j_{i}\leq m, |J|=k \}$ denote the set of all commutators of length $k$. Enumerating the elements of $X^{(1)},\ldots,X^{(\alpha_{n})}$, we denote them by $Y_{1},\ldots,Y_{q}$, and assign each  $Y_{i}$ a formal degree $d(Y_{i})=k$ if $Y_{i}\in X^{(k)}$. 

For each $n$-tuple of integers $I=(i_{1},\ldots,i_{n})$ with $1\leq i_{j}\leq q$, we define
\begin{equation}\label{2-5}
\lambda_{I}(x):=\det(Y_{i_{1}},\ldots,Y_{i_{n}})(x),
\end{equation}
where $\det(Y_{i_{1}},\ldots,Y_{i_{n}})(x)=\det(b_{jk}(x))$ with $Y_{i_{j}}=\sum_{k=1}^{n}b_{jk}(x)\partial_{x_{k}}$. We also set
\begin{equation*}
d(I):=d(Y_{i_{1}})+\cdots+d(Y_{i_{n}}),
\end{equation*}
and
\begin{equation}\label{2-6}
  \Lambda(x,r):=\sum_{I}|\lambda_{I}(x)|r^{d(I)},
\end{equation}
where the sum is taken over all $n$-tuples. The function $\Lambda(x,r)$, known as  the Nagel--Stein--Wainger polynomial, describes the volume of subunit balls.

\begin{proposition}[Global Ball--Box theorem, see {\cite[Theorem B]{Biagi-Bonfiglioli-Bramanti2019}}]
\label{prop2-6}
Assume conditions (\hyperref[H1]{H.1}) and (\hyperref[H2]{H.2}) hold. Then there exist constants $C_1,C_2>0$,
depending only on $X$, such that
\begin{equation}\label{2-7}
  C_{1}\Lambda(x,r)\leq |B(x,r)|\leq C_{2}\Lambda(x,r),
\end{equation}
for every $ x\in \mathbb{R}^n$ and  $r>0$, where $\Lambda(x,r)$ is the Nagel--Stein--Wainger polynomial defined in \eqref{2-6}. Furthermore, 
\begin{equation}\label{2-8}
    \Lambda(x,r)=\sum_{k=n}^{Q}f_k(x)r^k,
  \end{equation}
where  $f_{k}(x)=\sum_{d(I)=k}|\lambda_I(x)|$ is a non-negative, continuous function satisfying $f_{k}(\delta_t(x))=t^{Q-k}f_{k}(x)$ for all $x\in \mathbb{R}^n$ and  $t>0$. Additionally,  $f_Q(x)\equiv f_{Q}(0)>0$. 
\end{proposition}

Proposition \ref{prop2-6} yields the following global doubling property.
\begin{proposition}[Global doubling property]
\label{prop2-7}
Assume conditions (\hyperref[H1]{H.1}) and (\hyperref[H2]{H.2}) hold. Then there exists a constant $C_3>0$, depending only on $X$, such that
\begin{equation}\label{2-9}
 |B(x,r_{2})|\leq C_{3}\left(\frac{r_{2}}{r_{1}} \right)^{Q}|B(x,r_{1})|
\end{equation}
for every $x\in\mathbb R^n$ and $0<r_1<r_2$.
\end{proposition}

We next present a key proposition related to the pointwise homogeneous  dimension and the Nagel--Stein--Wainger polynomial.

\begin{proposition}
\label{prop2-8}
Assume conditions (\hyperref[H1]{H.1}) and (\hyperref[H2]{H.2}) hold. Then for each $x\in \mathbb{R}^n$, the pointwise homogeneous dimension $\nu(x)$ satisfies
\begin{equation}\label{2-10}
  \nu(x)=\sum_{j=1}^{r(x)}j(\nu_{j}(x)-\nu_{j-1}(x))=\lim_{r\to 0^{+}}\frac{\log\Lambda(x,r)}{\log r}=\min\{d(I)|\lambda_{I}(x)\neq 0\},
\end{equation}
and $\nu(x)\leq \nu(0)=Q$ for all $x\in \mathbb{R}^n$. Moreover, for every $\mathscr A\in\mathcal G$, $x\in\mathbb R^n$, and $t>0$,
\begin{equation}\label{2-11}
\Lambda(\mathscr{A}(x),t)=\Lambda(x,t),
\end{equation}
and
\begin{equation}\label{2-12}
\begin{aligned}
  \nu_{j}(x)&=\nu_{j}(\delta_{t}(x))=\nu_{j}(\mathscr{A}(x))~~~\forall~ 1\leq j\leq \alpha_n,\\
 \nu(x)&=\nu(\delta_{t}(x))=\nu(\mathscr{A}(x)).
\end{aligned}
\end{equation}
Moreover, if $ H=\{x\in\mathbb R^n\mid\nu(x)=Q\}=\mathbb R^n$, 
then the sub-Riemannian structure generated by $X$ is equiregular on
$\mathbb R^n$, and its pointwise homogeneous dimension is identically
equal to $Q$.
\end{proposition}
\begin{proof}
Identity \eqref{2-10} follows from \cite[Proposition 2.2]{chen-chen2019}. Combining  \eqref{2-8} and \eqref{2-10} shows that $\nu(x)\leq \nu(0)=Q$ for all $x\in \mathbb{R}^n$.  

Let $\mathscr{A}\in \mathcal{G}$. By \eqref{eq-1-15} we have 
\[   Y(f(\mathscr{A}(\cdot)))=(Yf)(\mathscr{A}(\cdot))\]
 for all $Y\in X^{(k)}$, $f\in C^{\infty}(\mathbb{R}^n)$ and $1\leq k\leq \alpha_n$. This means
\begin{equation}\label{eq-2-14}
   J_{\mathscr{A}}(x)\cdot Y(x)=Y(\mathscr{A}(x))\qquad \forall~ x\in \mathbb{R}^n,~~Y\in X^{(k)},~~ 1\leq k\leq \alpha_n.
\end{equation}
Here, $J_{\mathscr{A}}(x)$ denotes the Jacobian matrix of $\mathscr{A}$ at $x$. It then follows from \eqref{2-5} and \eqref{2-6} that $|\lambda_I(\mathscr{A}(x))|=|\lambda_I(x)|$ and  consequently  $\Lambda(\mathscr{A}(x),t)=\Lambda(x,t)$ for any $t>0$. Furthermore,  \eqref{2-10}  derives that $\nu(x)=\nu(\mathscr{A}(x))$.

Let $V_{j}(x)={\rm span}~\{Y(x)\mid Y\in X^{(k)},~1\leq k\leq j\}\subset T_{x}(\mathbb{R}^n)$ and define $\nu_j(x)=\dim V_j(x)$. The pushforward map $(d\delta_{t})_{x}$, induced by the dilation $\delta_{t}(x)$, is a  linear isomorphism from $T_{x}(\mathbb{R}^n)$ to $T_{\delta_{t}(x)}(\mathbb{R}^n)$. By \cite[Proposition 2.3]{chen-chen-li-2022} every $Y\in X^{(k)}$ with $Y\not\equiv 0$ is  $\delta_{t}$-homogeneous of degree $k$, and $(d\delta_{t})_{x}(Y(x))=t^{k}Y(\delta_{t}(x))$. This implies that  $(d\delta_{t})_{x}V_{j}(x)\subset V_{j}(\delta_{t}(x))$ and $\nu_{j}(x)\leq \nu_{j}(\delta_{t}(x))$. Replacing $\delta_{t}$ by $\delta_{\frac{1}{t}}$ yields $\nu_{j}(\delta_{t}(x))\leq \nu_{j}(\delta_{\frac{1}{t}}(\delta_{t}(x)))=\nu_{j}(x)$, thus establishing that $\nu_{j}(x)=\nu_{j}(\delta_{t}(x))$ for $1\leq j\leq \alpha_n$. As a result of \eqref{2-1}, we conclude that $\nu(x)=\nu(\delta_t(x))$.

Similarly, from \eqref{eq-2-14} we see that the pushforward map $(d\mathscr{A})_{x}$, induced by the map $\mathscr{A}$, is a  linear isomorphism from $T_{x}(\mathbb{R}^n)$ to $T_{\mathscr{A}(x)}(\mathbb{R}^n)$ such that $ (d\mathscr{A})_{x}(Y(x))= J_{\mathscr{A}}(x)\cdot Y(x)=Y(\mathscr{A}(x))$
 for all $x\in \mathbb{R}^n$, $Y\in X^{(k)}$ and $1\leq k\leq \alpha_n$. Consequently, $\nu_{j}(\mathscr{A}(x))=\nu_{j}(x)$ and $\nu(\mathscr{A}(x))=\nu(x)$.

To complete the proof, we show that  if $H=\{x\in \mathbb{R}^n|\nu(x)=Q\}=\mathbb{R}^n$, then $\nu_{j}(x)=\nu_{j}(0)$ for all $x\in \mathbb{R}^n$ and $1\leq j\leq \alpha_n$. To this end,  define the set
\[ S:=\{x\in \mathbb{R}^n\mid \nu_{j}(x)=\nu_{j}(0),~~1\leq j\leq \alpha_n\}. \]
Note that $\nu_{j}(x)=n$ for $r(x)\leq j\leq \alpha_n$, and $0\in S$. We now employ a standard topological argument to show that $S=\mathbb{R}^n$, which indicates that the sub-Riemannian structure generated by $X$ is equiregular on
$\mathbb R^n$, and $\nu(x)\equiv Q$.

 For any $x_0\in S$, $\nu_{j}(x_0)=\nu_{j}(0)$ for $1\leq j\leq \alpha_n$. Since $\nu_j(x)$ is an integer-valued  lower semi-continuous function, there exists an open neighborhood $U_j$ of $x_0$ such that $\nu_j(x)\geq \nu_j(x_0)$ for all $x\in U_j$. Let $U=\bigcap_{j=1}^{\alpha_n}U_j$ be an open neighborhood of $x_0$. It follows that 
\begin{equation}\label{eq-2-15}
  \nu_j(x)\geq \nu_j(x_0)\qquad \forall~ x\in U,~1\leq j\leq \alpha_n.
\end{equation}
Combining the fact $\nu(x)\equiv\nu(0)$ and \eqref{2-1}, we deduce that
\[ \alpha_n (\nu_{\alpha_n}(x)-\nu_{\alpha_n}(0))-\sum_{j=1}^{\alpha_n-1}(\nu_{j}(x)-\nu_{j}(0))=0\qquad \forall~ x\in \mathbb{R}^n. \]
Observing that $\nu_{\alpha_n}(x)=\nu_{\alpha_n}(0)=n$ for all $x\in \mathbb{R}^n$, we have
\begin{equation}\label{eq-2-16}
  \sum_{j=1}^{\alpha_n-1}(\nu_{j}(x)-\nu_{j}(0))=0\qquad \forall~ x\in \mathbb{R}^n.
\end{equation}
Thus, \eqref{eq-2-15} and \eqref{eq-2-16} give that $\nu_j(x)= \nu_j(x_0)$ for all $x\in U$ and $1\leq j\leq \alpha_n$. This means, $S$ is an open subset of $\mathbb{R}^n$.

Let $\{x_k\}_{k=1}^{\infty}\subset S$ and $x_k\to x_0$ as $k\to \infty$. It follows that $\nu_j(x_k)=\nu_j(0)$ for $k\geq 1$ and $j=1,\ldots,\alpha_n$. The lower semi-continuity of $\nu_j(x)$ implies that 
\begin{equation}\label{eq-2-17}
\nu_j(x_0)\leq \liminf_{k\to\infty}\nu_j(x_k)=\nu_j(0)\qquad\forall~ j=1,\ldots,\alpha_n.
\end{equation}
According to \eqref{eq-2-16} and \eqref{eq-2-17}, we conclude that $\nu_j(x_0)=\nu_j(0)$ for $1\leq j\leq \alpha_n$. Therefore, $x_0\in S$, and $S$ is a closed subset of $\mathbb{R}^n$. Consequently, $S=\mathbb{R}^n$.
\end{proof}

Let $\Omega\subset\mathbb R^n$ be an open set and write $d(x):=d(x,0)$. If $\Omega$ is unbounded, we define its
asymptotic projection onto the unit sphere associated with the subunit metric $\partial B(0,1) := \{x \in \mathbb{R}^n \mid d(x)=1\}$ by
\begin{equation}
  \Pi_\infty(\Omega) := \bigcap_{R>0} \overline{ \left\{ \delta_{\frac{1}{d(x)}}(x) \;\middle|\; x \in \Omega, \ d(x) > R \right\} },
\end{equation}
where the closure is taken in $\partial B(0,1)$. The associated asymptotic singular dimension is defined by
\begin{equation}\label{def-asd}
  \nu_{\rm sing} := \max_{z \in \Pi_\infty(\Omega)} \nu(z).
\end{equation}
If $\Omega$ is bounded, we set
\[
  \Pi_\infty(\Omega):=\varnothing,
  \qquad
  \nu_{\rm sing}:=-\infty.
\]
By Proposition \ref{prop2-9}, the sphere $\partial B(0,1)$ is compact, which indicates that $\Pi_\infty(\Omega)$ is compact. Since $\nu(\cdot):\mathbb{R}^n\to \mathbb{N}^{+}$ is upper semi-continuous, the maximum above is attained. 
 
\subsection{Function spaces associated with vector fields}

We now introduce several function spaces associated with homogeneous H\"{o}rmander vector fields that will be used throughout the paper.

Let $U\subset\mathbb R^n$ be open and let $1\leq p<\kappa$. We define
\[
 \mathcal M_X^{p,\kappa}(U)
 :=
 \left\{
 u\in L^{\frac{\kappa p}{\kappa-p}}(U)\mid
 X_j u\in L^p(U),\ 1\leq j\leq m
 \right\},
\]
where the derivatives are understood in the distributional sense.  We
endow this space with the norm
\[
  \|u\|_{\mathcal M_X^{p,\kappa}(U)}
  :=
  \|u\|_{L^{\frac{\kappa p}{\kappa-p}}(U)}
  +
  \|Xu\|_{L^p(U)}.
\]
Then $\mathcal M_X^{p,\kappa}(U)$ is a Banach space, and it is reflexive
if $1<p<\kappa$.  We denote by $\mathcal M_{X,0}^{p,\kappa}(U)$ the
closure of $C_0^\infty(U)$ in this norm.

For $p\geq 1$, we define the generalized Sobolev space by setting
\begin{equation}
\mathcal{W}_{X}^{1,p}(U):=\{u\in L^{p}(U)\mid X_{j}u\in L^{p}(U),~~ 1\leq j\leq m\}. 
\end{equation}
We equip $\mathcal W_X^{1,p}(U)$ with the norm
\[ \|u\|_{\mathcal W_X^{1,p}(U)}:=\|u\|_{L^p(U)}+\|Xu\|_{L^p(U)}.\]
The space $\mathcal W_{X,0}^{1,p}(U)$ is the closure of
$C_0^\infty(U)$ with respect to this norm. In particular,  we adopt the notations 
\begin{equation}
\mathcal{H}_{X}^{1}(U):=\mathcal{W}_{X}^{1,2}(U)\quad \mbox{and}\quad \mathcal{H}_{X,0}^{1}(U):=\mathcal{W}_{X,0}^{1,2}(U). 
\end{equation}
When  $U=\mathbb{R}^n$,  it is known (see \cite[Proposition 2.11]{chen-chen-li-2022}) that  $\mathcal{H}_{X,0}^{1}(\mathbb{R}^n)=\mathcal{H}_{X}^{1}(\mathbb{R}^n)$.

On the other hand, condition (\hyperref[H1]{H.1}) implies the compactness of closed subunit ball and the following global Poincar\'{e}--Wirtinger inequality.
\begin{proposition}
\label{prop2-9}
Assume conditions (\hyperref[H1]{H.1}) and (\hyperref[H2]{H.2}).  Then
$\overline{B(x_0,r)}$ is compact for every $x_0\in\mathbb R^n$ and
$r>0$.  Moreover, there exists a constant $C>0$, depending only on $X$,
such that
\begin{equation}\label{2-13}
  \int_{B(x_0,r)}
  \left|f(x)-\frac1{|B(x_0,r)|}\int_{B(x_0,r)}f(y)dy\right|\,dx
  \leq
  C r\int_{B(x_0,r)}|Xf|\,dx
\end{equation}
for every $f\in{\rm Lip}(\overline{B(x_0,r)})$.
\end{proposition}
\begin{proof}
Since $X$ satisfy the H\"{o}rmander's condition (\hyperref[H0]{H.0}) on $\mathbb{R}^n$,  the identity map $\iota:(\mathbb{R}^n,|\cdot|)\to (\mathbb{R}^n,d)$ is continuous (see \cite[p. 69]{Garofalo1998}). Furthermore, by \cite[Proposition 2.3]{Danielli1998}, $(\mathbb{R}^n,d)$ is locally compact. In particular, there exists $r_0>0$ such that $\overline{B(0,r_0)}$ is compact. 

Using \cite[Theorem 1]{Franchi1995}, one can find $\widetilde{r}_{0}>0$ (depending on $\overline{B(0,r_0)}$ and $X$) such that for any $z\in \overline{B(0,r_0)}$ and any $0<r\leq \widetilde{r}_{0}$  the inequality
\begin{equation}\label{2-14}
  \int_{B(z,r)}\left|g(x)-\frac{1}{|B(z,r)|}\int_{B(z,r)}g(y)dy\right|dx\leq Cr\int_{B(z,r)}|Xg(x)|dx
\end{equation}
holds for all $g\in {\rm Lip}(\overline{B(z,r)})$. Now, for any $x_0\in \mathbb{R}^n$ and $r>0$, choosing
\[
  0<t_0\leq
  \min\left\{
    \frac{\widetilde r_0}{r},
    \frac{r_0}{2(r+d(x_0,0))}
  \right\},
\]
then $\delta_{t_0}(x_0)\in B(0,r_0)$, $t_0r\leq\widetilde r_0$, and 
$ \delta_{t_{0}}(B(x_{0},r))=B(\delta_{t_{0}}(x_{0}),t_{0}r)$. 
 
For any $f\in {\rm Lip}(\overline{B(x_{0},r)})$, it follows that $f\circ \delta_{\frac{1}{t_{0}}}\in {\rm Lip}(\overline{B(\delta_{t_{0}}(x_{0}),t_{0}r)})$. Then, \eqref{2-14} gives that
\begin{equation*}
\begin{aligned}
 &\int_{B(\delta_{t_{0}}(x_{0}),t_{0}r)}\left|f(\delta_{\frac{1}{t_{0}}}(x))-\frac{1}{|B(\delta_{t_{0}}(x_{0}),t_{0}r)|}\int_{B(\delta_{t_{0}}(x_{0}),t_{0}r)}f(\delta_{\frac{1}{t_{0}}}(y))dy\right|dx\\
&\leq Ct_{0}r\int_{B(\delta_{t_{0}}(x_{0}),t_{0}r)}|X(f(\delta_{\frac{1}{t_{0}}}(x)))|dx,
\end{aligned}
\end{equation*}
which yields \eqref{2-13}. Finally, from 
$\delta_{t_{0}}(\overline{B(x_{0},r)})=\overline{B(\delta_{t_{0}}(x_{0}),t_{0}r)}\subset \overline{B(0,r_0)}$, 
we obtain that $\overline{B(x_{0},r)}$ is compact for any $x_{0}\in \mathbb{R}^n$ and $r>0$.
\end{proof}

\begin{proposition}
\label{prop2-10}
Assuming conditions (\hyperref[H1]{H.1}) and (\hyperref[H2]{H.2}), and let  $1\leq p<Q$. If $u_k\rightharpoonup 0$ in $\mathcal{M}_{X,0}^{p,Q}(\mathbb{R}^n)$, then $u_k\to 0 $ in $L_{\rm loc}^{p}(\mathbb{R}^n)$.  
\end{proposition}
\begin{proof}
It follows from  Proposition \ref{prop2-9} and \cite[Theorem 1.28]{Garofalo1996} that for any bounded $X$-Poincar\'e--Sobolev domain $U\subset \mathbb{R}^n$, the embedding $\mathcal{W}_{X}^{1,p}(U)\hookrightarrow L^{p}(U)$ is compact provided $1\leq p<Q$. 
For any $r_{0}>0$, Proposition \ref{prop2-9} and \cite[Proposition 2.5]{Danielli1998} imply that $B(0,r_0)$ is a bounded  $X$-Poincar\'e--Sobolev domain. 
If $u_k\rightharpoonup 0$ in $\mathcal{M}_{X,0}^{p,Q}(\mathbb{R}^n)$, we have $u_k\rightharpoonup 0$ in $\mathcal{W}_{X}^{1,p}(B(0,r_0))$. Since for $1\leq p<Q$, the embedding $\mathcal{W}_{X}^{1,p}(B(0,r_0))\hookrightarrow L^{p}(B(0,r_0))$ is compact, it follows that $u_k \to 0$ in $L^{p}(B(0,r_0))$. Thus,  $u_k\to 0 $ in $L_{\rm loc}^{p}(\mathbb{R}^n)$.
\end{proof}

\section{The Sobolev inequality and growth estimate of the volume of subunit ball}
\label{Section3}

In this section, we investigate the properties of the sets $\mathcal{I}(\Omega)$ and $\mathcal{S}(\Omega)$, which form the proofs of Theorems \ref{thm1}-\ref{thm3} and Corollary \ref{corollary1-1}. Our main tool is the global heat kernel of homogeneous H\"{o}rmander operator 
$\triangle_{X}=\sum_{j=1}^{m}X_{j}^2$.

\subsection{The global heat kernel of homogeneous H\"{o}rmander operator}
By Proposition \ref{prop2-2}, we see that $\operatorname{div}X_j=0$, which yields that $X_j^*=-X_j$ and $\triangle_{X}=-\sum_{j=1}^{m}X_{j}^{*}X_{j}$. Then, we have

 \begin{proposition}
\label{prop3-1}
The Friedrichs realization of the non-negative symmetric operator $-\triangle_X$ generates a heat semigroup $\{P_t\}_{t\geq0}$ admitting a unique heat kernel $h(x,y,t)$ such that, for every $f\in L^2(\mathbb R^n)$, 
\begin{equation}\label{3-1}
  P_{t}f(x)=e^{t\triangle_X}f(x)=\int_{\mathbb{R}^n}h(x,y,t)f(y)dy
\end{equation}
holds for all $x\in \mathbb{R}^n$ and $t>0$. Moreover, $h(x,y,t)\in C^{\infty}(\mathbb{R}^n\times\mathbb{R}^n\times\mathbb{R}^{+})$ and satisfies the following properties:
\begin{enumerate}[(1)]
\item For any $x,y\in\mathbb{R}^n$ and $t>0$, $h(x,y,t)=h(y,x,t)$.
\item For any $x\in\mathbb{R}^n$ and $t>0$, 
 $\int_{\mathbb{R}^n}h(x,y,t)dy=1$.
 \item For any fixed point $y\in \mathbb{R}^n$, $h(x,y,t)$ is the solution of
\begin{equation}
\label{3-2}
  \left(\frac{\partial}{\partial t}-\triangle_{X}\right)h(x,y,t)=0 \qquad \forall~~ (x,t)\in\mathbb{R}^n\times \mathbb{R}^{+}.
\end{equation}
  \item There exists $A_1\geq 1$ such that for any $x,y\in \mathbb{R}^n$ and $t>0$, we have
\begin{equation}
 \label{3-3}
    \frac{1}{A_{1}|B (x,\sqrt{t})|}e^{-\frac{A_{1} d^{2}(x,y)}{t}}\leq h(x,y,t)\leq
 \frac{A_{1}}{|B (x,\sqrt{t})|}e^{-\frac{ d^{2}(x,y)}{A_{1}t}}.
 \end{equation}
  \item There exists $A_2\geq 1$ such that for any $x,y\in \mathbb{R}^n$ and $t>0$, we have
     \begin{equation}\label{3-4}
 |X_{j}^{y}h(x,y,t)|\leq
 \frac{A_{2}}{t^{\frac{1}{2}}|B (x,\sqrt{t})|}e^{-\frac{ d^{2}(x,y)}{A_{2}t}}
 \end{equation}
 for $1\leq j\leq m$.
\item For any $f\in L^2(\mathbb{R}^n)$, the function $u(x,t)=P_{t}f(x)\in C^{\infty}(\mathbb{R}^n\times \mathbb{R}^+)$ and satisfies the
degenerate heat equation $\partial_t u=\triangle_{X}u$ on $\mathbb{R}^n\times \mathbb{R}^+$.
 \item For $f\in C_0^\infty(\mathbb{R}^n)$, we have 
 \begin{equation}\label{3-5}
\lim\limits_{t\to 0^+}P_tf(x)=f(x)\qquad \forall~ x\in \mathbb{R}^n,
 \end{equation}
 and
\begin{equation}\label{3-6}
\lim\limits_{t\to +\infty}P_tf(x)=0\qquad \forall~ x\in \mathbb{R}^n.
\end{equation} 
\end{enumerate}
\end{proposition}
\begin{proof}
The asserted  properties (1)-(6), \eqref{3-1} and \eqref{3-5}  follow from \cite[Section 3]{chen-chen-li-2022}, \cite[Theorem 1.4]{Biagi2019} and \cite[Theorem 2.4]{Biagi2020}. Moreover, Proposition \ref{prop2-6} gives $|B(x,\sqrt{t})|\geq C_1 f_Q(0)t^{\frac{Q}{2}}$. 
Hence, for every $f\in C_0^\infty(\mathbb R^n)$,
\[
 |P_tf(x)|
 \leq \int_{\mathbb R^n}h(x,y,t)|f(y)|dy
 \leq \frac{A_1}{C_1f_Q(0)}t^{-\frac{Q}{2}}\|f\|_{L^1(\mathbb R^n)},
\]
which implies \eqref{3-6}.
\end{proof}

\subsection{Proof of Theorem \ref{thm1}} 
The proof of Theorem \ref{thm1} follows from the following two lemmas.

\begin{lemma}
\label{lemma3-1}
$\mathcal{I}(\Omega)\subset  \mathcal{S}(\Omega)$.
\end{lemma}
\begin{proof}
For any $\kappa\in  \mathcal{I}(\Omega)$, there exists a positive constant $\widetilde{C}=\widetilde{C}(X,\Omega)>0$ such that 
\begin{equation}\label{3-7}
  |B(y,s)|\geq \widetilde{C}s^{\kappa}\qquad \forall~ y\in \overline{\Omega},~s>0.
\end{equation} 
We first verify that $1<\kappa\leq Q$, as required in the definition of $\mathcal S(\Omega)$. Fix $y_0\in\overline\Omega$ and set
$ M_0:=\sum_{\ell=\nu(y_0)}^Q f_\ell(y_0)>0$. By Propositions \ref{prop2-6} and \ref{prop2-8}, for $0<s\leq1$,
\[
 \widetilde C s^\kappa\leq |B(y_0,s)|
 \leq C_2\sum_{\ell=\nu(y_0)}^Q f_\ell(y_0)s^\ell
 \leq C_2M_0s^{\nu(y_0)},
\]
whereas for $s\geq1$,
\[
 \widetilde C s^\kappa\leq |B(y_0,s)|
 \leq C_2M_0s^Q.
\]
Letting $s\to 0$ in the first inequality and $s\to +\infty$ in the second gives
$
 n\leq\nu(y_0)\leq\kappa\leq Q$. Since $n\geq2$, it follows that $1<\kappa\leq Q$.

We then show that for any $u\in \mathcal{W}_{X,0}^{1,1}(\Omega)$, the following weak Sobolev inequality holds:
\begin{equation}\label{3-8}
\sup_{\lambda>0} \lambda^{\frac{\kappa}{\kappa-1}}|\{x\in \Omega| |u(x)|>\lambda\}|\leq c_{0}\|Xu\|_{L^1(\Omega)}^{\frac{\kappa}{\kappa-1}},
\end{equation}  
where
\[ c_{0}:=2^{\frac{3\kappa+1}{\kappa-1}}m^{\frac{2\kappa-1}{\kappa-1}}(\widetilde{C}(\kappa-1))^{-\frac{1}{\kappa-1}}A_{1}^{\frac{Q\kappa+\kappa+2}{2(\kappa-1)}}A_{2}^{\frac{(Q+1)\kappa}{2(\kappa-1)}} C_{3}^{\frac{\kappa}{\kappa-1}}>0,\]
 $A_{1},A_{2}$ are the positive constants appeared in Proposition \ref{prop3-1}, and $C_{3}$ is the positive constant defined in Proposition \ref{prop2-7}.

For $f\in C_0^\infty(\mathbb R^n)$ and $0<\varepsilon<R$, we have
\[
\begin{aligned}
-\int_\varepsilon^R P_t\triangle_X f(x)\,dt
&=
-\int_\varepsilon^R
 \int_{\mathbb R^n}
 h(x,y,t)(\triangle_Xf)(y)\,dy\,dt  \\
&=
-\int_\varepsilon^R
 \int_{\mathbb R^n}
 f(y)\triangle_X^yh(x,y,t)\,dy\,dt  \\
&=
-\int_\varepsilon^R
 \int_{\mathbb R^n}
 f(y)\partial_th(x,y,t)dydt \\
&=-\int_\varepsilon^R
 \partial_tP_tf(x)dt =P_\varepsilon f(x)-P_Rf(x).
\end{aligned}
\]
Letting $\varepsilon\to 0^{+}$ and $R\to+\infty$, and using
\eqref{3-5} and \eqref{3-6}, we obtain
\begin{equation}\label{3-9}
f(x)
=
-\int_0^\infty P_t\triangle_Xf(x) dt.
\end{equation}
For any open subset $\Omega\subset \mathbb{R}^n$, we identify  $C_0^\infty(\Omega)$ as a subspaces of $C_{0}^{\infty}(\mathbb{R}^n)$ by extending every function $u\in C_0^\infty(\Omega)$ to $\mathbb{R}^n$ by 
 setting $u=0$ outside $\Omega$. Then, for any $u\in C_0^\infty(\Omega)$ and $\lambda>0$, we derive from \eqref{3-9} that
 \begin{equation}\label{3-10}
 |\{x\in \Omega| |u(x)|>\lambda\}|\leq \sum_{i=1}^m\left|\left\{x\in \Omega\bigg| \left|\int_0^{+\infty}(P_tX_{i})(X_iu)(x)dt\right|\geq\frac{\lambda}{m}\right\}\right|.
 \end{equation}
 
Fix $1\leq i\leq m$ and let $f:=X_iu\in C_0^\infty(\Omega)$. By Propositions \ref{prop2-7} and \ref{prop3-1}, as well as \eqref{3-7}, we deduce that for any $x, y\in \overline{\Omega}$ and $t>0$, 
\begin{equation*}
  |X_i^yh(x,y,t)|\leq C_{3}(A_{1}A_{2})^{\frac{Q}{2}+1}t^{-\frac{1}{2}}h(y,x,A_{1}A_{2}t)\leq  C_{3}A_{1}^{\frac{Q-\kappa}{2}+2}A_{2}^{\frac{Q-\kappa}{2}+1}\widetilde{C}^{-1}t^{-\frac{\kappa+1}{2}}.
\end{equation*}
This means
\begin{equation}\label{3-11}
 \begin{aligned}
  \|P_tX_if\|_{L^\infty(\Omega)}&=\sup_{x\in\overline{\Omega}}\left|\int_{\Omega}h(x,y,t)X_if(y)dy \right|=\sup_{x\in\overline{\Omega}}\left|\int_{\Omega}f(y)X_i^yh(x,y,t)dy\right|\\
  &\leq \sup_{x,y\in \overline{\Omega}}|X_i^yh(x,y,t)|\|f\|_{L^1(\Omega)}\leq C_{3}A_{1}^{\frac{Q-\kappa}{2}+2}A_{2}^{\frac{Q-\kappa}{2}+1}\widetilde{C}^{-1}t^{-\frac{\kappa+1}{2}}\|f\|_{L^1(\Omega)}.
 \end{aligned} 
  \end{equation}
 According to \eqref{3-11}, for any $s>0$ and $x\in \Omega$, we have
 \begin{equation}\label{3-12}
\begin{aligned}
  \bigg|\int_{s}^{+\infty}P_tX_{i}f(x) dt\bigg|&\leq C_{3}A_{1}^{\frac{Q-\kappa}{2}+2}A_{2}^{\frac{Q-\kappa}{2}+1}\widetilde{C}^{-1}\|f\|_{L^1(\Omega)}\int_{s}^{+\infty}t^{-\frac{1+\kappa}{2}}dt\\
  &=\frac{2 C_{3}A_{1}^{\frac{Q-\kappa}{2}+2}A_{2}^{\frac{Q-\kappa}{2}+1}\widetilde{C}^{-1}}{\kappa-1}s^{\frac{1-\kappa}{2}}\|f\|_{L^1(\Omega)}:=c_{1}s^{\frac{1-\kappa}{2}}\|f\|_{L^1(\Omega)},
\end{aligned}
  \end{equation}
where we substitute $c_1:=\frac{2 C_{3}A_{1}^{\frac{Q-\kappa}{2}+2}A_{2}^{\frac{Q-\kappa}{2}+1}\widetilde{C}^{-1}}{\kappa-1}$ for brevity. 
  
For any $\lambda>0$,  let $s_{\lambda}:=\left(\frac{4mc_{1}}{\lambda}\right)^{\frac{2}{\kappa-1}}\|f\|_{L^1(\Omega)}^{\frac{2}{\kappa-1}} $ such that
$ c_{1}s_{\lambda}^{\frac{1-\kappa}{2}}\|f\|_{L^1(\Omega)}=\frac{\lambda}{4m}$. Then, 
 \begin{equation*}
  \bigg|\int_{s_{\lambda}}^{+\infty}P_tX_{i}f(x) dt\bigg|\leq \frac{\lambda}{4m} \qquad \forall~ x\in\Omega,
  \end{equation*}
which implies that
 \begin{equation}
\begin{aligned}\label{3-13}
\left|\left\{x\in \Omega\bigg| \bigg|\int_0^{+\infty}P_tX_{i}f(x) dt\bigg|\geq\frac{\lambda}{m}\right\}\right|\leq \left|\left\{x\in \Omega\bigg| \bigg|\int_0^{s_{\lambda}}(P_tX_{i})f(x) dt\bigg|\geq\frac{\lambda}{2m}\right\}\right|.
\end{aligned} 
  \end{equation}
Observing that $\int_{\mathbb{R}^n}h(x,y,t)dx=1$ for any $y\in \mathbb{R}^n$ and $t>0$, we obtain
\begin{equation}\label{3-14}
\begin{aligned}
 & \|P_tX_if\|_{L^1(\Omega)}=\int_{\Omega}\left|\int_{\Omega}h(x,y,t)X_if(y)dy\right|dx=\int_{\Omega}\left|\int_{\Omega}f(y)X_i^yh(x,y,t)dy\right|dx\\
  &\leq C_{3}(A_{1}A_{2})^{\frac{Q}{2}+1}t^{-\frac{1}{2}}\int_{\mathbb{R}^n}\int_{\Omega}|f(y)|h(x,y,A_{1}A_{2}t)dydx=C_{3}(A_{1}A_{2})^{\frac{Q}{2}+1} t^{-\frac{1}{2}}\|f\|_{L^1(\Omega)}.
 \end{aligned} 
\end{equation}
It follows from  \eqref{3-14} that
\[\begin{aligned}
&\left|\left\{x\in \Omega\bigg| \bigg|\int_0^{s_{\lambda}}(P_tX_{i})f(x) dt\bigg|\geq\frac{\lambda}{2m}\right\}\right|\leq \frac{2m}{\lambda}\int_0^{s_{\lambda}}\int_{\Omega}\left|(P_tX_{i})f(x)\right| dxdt\\
&\leq\frac{4m}{\lambda}C_{3}(A_{1}A_{2})^{\frac{Q}{2}+1} s_{\lambda}^{\frac{1}{2}}\|f\|_{L^1(\Omega)}=\left(\frac{2}{\widetilde{C}(\kappa-1)}\right)^{\frac{1}{\kappa-1}}A_{1}^{\frac{Q\kappa+\kappa+2}{2(\kappa-1)}}A_{2}^{\frac{(Q+1)\kappa}{2(\kappa-1)}}\left( \frac{4mC_{3}}{\lambda}\right)^{\frac{\kappa}{\kappa-1}}\|f\|_{L^1(\Omega)}^{\frac{\kappa}{\kappa-1}}.
\end{aligned} \]
Thus, we conclude that for any $1\leq i\leq m$ and $u\in C_{0}^{\infty}(\Omega)$, 
\[ \left|\left\{x\in \Omega\bigg| \bigg|\int_0^{+\infty}(P_tX_{i})(X_iu)(x) dt\bigg|\geq\frac{\lambda}{m}\right\}\right|\leq c_{2}\lambda^{-\frac{\kappa}{\kappa-1}}\|X_iu\|_{L^1(\Omega)}^{\frac{\kappa}{\kappa-1}},\]
where we set $c_{2}:=\left(\frac{2}{\widetilde{C}(\kappa-1)}\right)^{\frac{1}{\kappa-1}}A_{1}^{\frac{Q\kappa+\kappa+2}{2(\kappa-1)}}A_{2}^{\frac{(Q+1)\kappa}{2(\kappa-1)}}\left( 4mC_{3}\right)^{\frac{\kappa}{\kappa-1}}$ for brevity.  As a result of \eqref{3-10}, 
 \begin{equation}
\begin{aligned}\label{3-15}
 |\{x\in \Omega| |u(x)|>\lambda\}| &\leq \sum_{i=1}^m\left|\left\{x\in \Omega\bigg| \left|\int_0^{+\infty}(P_tX_{i})(X_iu)(x)dt\right|\geq\frac{\lambda}{m}\right\}\right|\leq mc_{2}\lambda^{-\frac{\kappa}{\kappa-1}}\|Xu\|_{L^1(\Omega)}^{\frac{\kappa}{\kappa-1}}\\
&= m^{\frac{2\kappa-1}{\kappa-1}}\left(\frac{2}{\widetilde{C}(\kappa-1)}\right)^{\frac{1}{\kappa-1}}A_{1}^{\frac{Q\kappa+\kappa+2}{2(\kappa-1)}}A_{2}^{\frac{(Q+1)\kappa}{2(\kappa-1)}}\left( 4C_{3}\right)^{\frac{\kappa}{\kappa-1}}\lambda^{-\frac{\kappa}{\kappa-1}}\|Xu\|_{L^1(\Omega)}^{\frac{\kappa}{\kappa-1}}
 \end{aligned}
 \end{equation}
holds for any $\lambda>0$ and $u\in C_{0}^{\infty}(\Omega)$.

For any $u\in \mathcal{W}_{X,0}^{1,1}(\Omega)$, there exists a sequence $\{u_k\}_{k=1}^{\infty}\subset C_0^\infty(\Omega)$ such that $u_k \to u$ in $\mathcal{W}_{X,0}^{1,1}(\Omega)$. Then, we can select a subsequence $\{u_{k_j}\}_{j=1}^{\infty}\subset \{u_k\}_{k=1}^{\infty}$ such that for any $\varepsilon>0$,
 \begin{equation}\label{3-16}
\lim_{j\to\infty}|\{x\in \Omega||u_{k_j}(x)-u(x)|\geq \varepsilon \}|=0.
 \end{equation}
  Now, for any $\lambda>0$, by \eqref{3-15} we have
 \begin{equation}\label{3-17}
  \begin{aligned}
|\{x\in \Omega| |u(x)|>\lambda\}|&\leq \left|\left\{x\in \Omega\Big||u(x)-u_{k_j}(x)|>\frac{\lambda}{2}\right\}\right|+\left|\left\{x\in \Omega\Big||u_{k_j}(x)|>\frac{\lambda}{2}\right\}\right|\\
  &\leq \left|\left\{x\in \Omega\Big||u(x)-u_{k_j}(x)|>\frac{\lambda}{2}\right\}\right|+c_{0}\lambda^{-\frac{\kappa}{\kappa-1}}\|Xu_{k_{j}}\|_{L^1(\Omega)}^{\frac{\kappa}{\kappa-1}},
  \end{aligned}
  \end{equation}
where
\begin{equation}\label{3-18}
c_{0}:=2^{\frac{3\kappa+1}{\kappa-1}}m^{\frac{2\kappa-1}{\kappa-1}}\left(\frac{1}{\widetilde{C}(\kappa-1)}\right)^{\frac{1}{\kappa-1}}A_{1}^{\frac{Q\kappa+\kappa+2}{2(\kappa-1)}}A_{2}^{\frac{(Q+1)\kappa}{2(\kappa-1)}} C_{3}^{\frac{\kappa}{\kappa-1}}>0.
\end{equation}
Letting $j\to \infty$ in \eqref{3-17} and using \eqref{3-16}, the weak Sobolev inequality \eqref{3-8} follows.

Let $\phi(t)=\max\{0,\min\{t,1\}\}=t_+-(t-1)_+\geq 0$ be the auxiliary function on $\mathbb{R}$. For any $u\in C^\infty_0(\Omega)$ and $i\in \mathbb{Z}$, we define
  \begin{equation}\label{3-19}
u_i(x):=\phi(2^{1-i}|u(x)|-1)=\left\{
                            \begin{array}{ll}
                              0, & \hbox{if $|u(x)|\leq 2^{i-1}$;} \\[1.5mm]
                              2^{1-i}|u(x)|-1, & \hbox{if $2^{i-1}<|u(x)|\leq 2^{i}$;} \\[1.5mm]
                              1, & \hbox{if $|u(x)|>2^{i}$.}
                            \end{array}
                          \right.
\end{equation}
From \cite[Proposition 2.9]{chen-chen-li2024}, we see that $u_i=\left((2^{1-i}|u|-1)_{+}-(2^{1-i}|u|-2)_{+}\right)\in \mathcal{W}^{1,1}_{X,0}(\Omega)$,  and
   \begin{equation*}
Xu_i= \begin{cases}2^{1-i}{\rm sgn}(u)Xu & \text { if }2^{i-1} < |u(x)| \leq 2^i,\\ 0 & \text { otherwise. }\end{cases}
\end{equation*}
Applying \eqref{3-8} to $u_i$,  we have for any $t>0$,
\begin{equation}\label{3-20}
\begin{aligned}
|\{x\in \Omega|u_i(x)>t\}|&\leq c_{0}t^{-\frac{\kappa}{\kappa-1}}\|Xu_i\|_{L^1(\Omega)}^{\frac{\kappa}{\kappa-1}}\leq  c_{0}2^{(1-i)\frac{\kappa}{\kappa-1}}t^{-\frac{\kappa}{\kappa-1}}\left(\int_{\{x\in\Omega|2^{i-1} < |u(x)| \leq 2^i\}}|Xu|dx\right)^{\frac{\kappa}{\kappa-1}}.
\end{aligned}
 \end{equation}
According to \eqref{3-19} and \eqref{3-20}, we deduce that
  \begin{align*}
    &\int_{\Omega}|u(x)|^{\frac{\kappa}{\kappa-1}}dx=\sum_{i=-\infty}^{\infty}\int_{\{x\in\Omega|2^{i}<|u(x)|\leq 2^{i+1}\}}|u(x)|^{\frac{\kappa}{\kappa-1}}dx\\
&\leq\sum_{i=-\infty}^\infty  2^{(i+1)\frac{\kappa}{\kappa-1}}|\{x\in \Omega|2^i<|u(x)|\leq 2^{i+1}\}|\leq \sum_{i=-\infty}^\infty  2^{(i+1)\frac{\kappa}{\kappa-1}}|\{x\in \Omega|u_i(x)= 1\}|\\
&\leq \sum_{i=-\infty}^\infty  2^{(i+1)\frac{\kappa}{\kappa-1}}|\{x\in \Omega|u_i(x)>2^{-1}\}|\\
     &\leq c_{0}\sum_{i=-\infty}^\infty  2^{(i+2)\frac{\kappa}{\kappa-1}}2^{(1-i)\frac{\kappa}{\kappa-1}}\left(\int_{\{x\in\Omega|2^{i-1} < |u(x)| \leq 2^i\}}|Xu|dx\right)^{\frac{\kappa}{\kappa-1}}\leq 2^{\frac{3\kappa}{\kappa-1}}c_{0}\|Xu\|_{L^1(\Omega)}^{\frac{\kappa}{\kappa-1}}
  \end{align*}
for any $u\in C_{0}^{\infty}(\Omega)$. This indicates that  $\kappa\in \mathcal{S}(\Omega)$ and $\mathcal{I}(\Omega)\subset \mathcal{S}(\Omega)$.
\end{proof}

\begin{lemma}
\label{lemma3-2}
$\mathcal{S}(\Omega)\subset\mathcal{I}_{\mathrm{int}}(\Omega)$.
\end{lemma}
\begin{proof}
Letting $\kappa\in\mathcal{S}(\Omega)$, it follows that
\begin{equation}\label{3-A-21}
 \|u\|_{L^{\frac{\kappa}{\kappa-1}}(\Omega)}\leq C\|Xu\|_{L^1(\Omega)}
 \qquad\forall u\in\mathcal W_{X,0}^{1,1}(\Omega).
\end{equation}

Fix $x\in\Omega$ and $0<r<d(x,\Omega^c)$, with the convention $d(x,\varnothing)=+\infty$, and define
\[
 f(y):=(r-d(x,y))_+.
\]
Then $0\leq f\leq r$, and $|f(y_1)-f(y_2)|\leq d(y_1,y_2)$ for all  $y_{1},y_{2}\in \mathbb{R}^n$. From Proposition \ref{prop2-9},  ${\rm supp}~f=\overline{B(x,r)}$ is a compact subset of $\mathbb{R}^n$. 
According to \cite[Theorem 2.5]{Monti2001},  $|Xf|\leq 1$ a.e. on $\mathbb{R}^n$, which implies that $f\in \mathcal{W}_{X,0}^{1,1}(\mathbb{R}^n)$. Since $0<r<d(x,\Omega^{c})$, then  ${\rm supp}~f=\overline{B(x,r)}\subset\subset \Omega$, and therefore $f\in \mathcal{W}_{X,0}^{1,1}(\Omega)$.

Using \eqref{3-A-21}, we have
\[\frac{r}{2}\left|B\left(x,\frac{r}{2}\right)\right|^{\frac{\kappa-1}{\kappa}}\leq \left(\int_{B\left(x,\frac{r}{2}\right)}|f(y)|^{\frac{\kappa}{\kappa-1}}dy\right)^{\frac{\kappa-1}{\kappa}}\leq C  \|Xf\|_{L^1(\Omega)}\leq C|B(x,r)|. \]
Using Proposition \ref{prop2-7}, we have $|B(x,\frac{r}{2})|\geq C_{3}^{-1}2^{-Q}|B(x,r)|$. Therefore,  for any $x\in\Omega$ and $0<r<d(x,\Omega^{c})$, 
\[ \frac{|B(x,r)|}{r^{\kappa}}\geq 2^{Q(1-\kappa)-\kappa}C^{-\kappa}C_{3}^{1-\kappa}, \]
which implies $\kappa\in \mathcal{I}_{\mathrm{int}}(\Omega)$. Consequently, $\mathcal{S}(\Omega)\subset\mathcal{I}_{\mathrm{int}}(\Omega)$.
\end{proof}

\subsection{Proof of Theorem \ref{thm2}}
The proof of Theorem \ref{thm2} is completed by the following Lemmas \ref{lemma3-3}-\ref{lemma3-6}.
\begin{lemma}
\label{lemma3-3}
 $\mathcal{I}(\Omega)$ is a non-empty interval of the form either $(\inf\mathcal{I}(\Omega),Q]$ or $[\inf\mathcal{I}(\Omega), Q]$, with $\inf\mathcal{I}(\Omega)\geq \tilde{\nu}$. Moreover, if $0\in \overline{\Omega}$, then $\mathcal{I}(\Omega)=\{Q\}$. In particular, $\mathcal{I}(\mathbb{R}^n)=\{Q\}$. 
\end{lemma}
\begin{proof}
We first show that $\sup \mathcal{I}(\Omega)=Q\in \mathcal{I}(\Omega)$.
By Proposition \ref{prop2-6}, we have $Q\in \mathcal{I}(\Omega)$.
  Assume by contradiction that $\kappa\in \mathcal{I}(\Omega)$ and $\kappa>Q$. Combining \eqref{1-6} and \eqref{2-7}, we obtain 
 \begin{equation}\label{3-23}
 \sum_{k=n}^{Q}f_k(x)r^k\geq    Cr^{\kappa}\qquad \forall~ x\in \overline{\Omega},~r>0.
 \end{equation}
Taking $r\to \infty$ in \eqref{3-23}, we obtain a contradiction because $\kappa>Q$ and $f_Q(x)\equiv f_{Q}(0)>0$. Therefore, $\sup \mathcal{I}(\Omega)=Q\in \mathcal{I}(\Omega)$.

If $\kappa\in \mathcal{I}(\Omega)$ with $\kappa<Q$, for any $\alpha\in [\kappa,Q]$, we deduce from Proposition \ref{prop2-6} that
\[ r^{\alpha}=r^{\frac{Q-\alpha}{Q-\kappa}\kappa+\frac{\alpha-\kappa}{Q-\kappa}Q}\leq \left(\frac{|B(x,r)|}{C}\right)^{\frac{Q-\alpha}{Q-\kappa}}\left(\frac{|B(x,r)|}{C} \right)^{\frac{\alpha-\kappa}{Q-\kappa}}=\frac{|B(x,r)|}{C}\qquad \forall~ x\in \overline{\Omega},~r>0, \]
which derives that $[\kappa,Q]\subset \mathcal{I}(\Omega)$, and $\mathcal{I}(\Omega)$ is a non-empty interval of the form either $(\inf\mathcal{I}(\Omega),Q]$ or $[\inf\mathcal{I}(\Omega), Q]$.  

Then, we verify that $\inf\mathcal{I}(\Omega)\geq \tilde{\nu}$. 
Suppose there exists $\kappa\in \mathcal{I}(\Omega)$ such that $\kappa<\tilde{\nu}$. Since $\nu(x)$ is an integer-valued function, there exists $x_0\in \overline{\Omega}$ such that $\nu(x_0)=\tilde{\nu}> \kappa$. Using Propositions \ref{prop2-6} and \ref{prop2-8}, we obtain that $f_{\nu(x_0)}(x_0)>0$ and
\begin{equation}\label{3-24}
  \Lambda(x_0,r)=\sum_{l=\nu(x_0)}^{Q}f_{l}(x_0)r^{l}\geq Cr^{\kappa}\qquad \forall~ r>0,
\end{equation}
where $C>0$ is a positive constant. It follows from \eqref{3-24} that 
\begin{equation}\label{3-25}
  \sum_{l=\nu(x_0)}^{Q}f_{l}(x_0)r^{l-\kappa}\geq C \qquad \forall~ r>0.
\end{equation}
Taking  $r\to 0^{+}$ in \eqref{3-25}, we get $C\leq 0$, which again leads to a contradiction. Consequently, $\inf\mathcal{I}(\Omega)\geq \tilde{\nu}$.

On the other hand, Proposition \ref{prop2-8} yields that $ \nu(0)=\max_{x\in \mathbb{R}^n}\nu(x)=Q$. This implies that $\mathcal{I}(\Omega)=\{Q\}$, provided  $0\in\overline{\Omega}$.
\end{proof}

\begin{lemma}
\label{lemma3-4}
For any  $\mathscr{A}\in \mathcal{G}$ and $t>0$, we have $\mathcal{I}(\Omega)=\mathcal{I}(\mathscr{A}(\Omega))=\mathcal{I}(\delta_{t}(\Omega))$. 
\end{lemma}
\begin{proof}
By  Proposition \ref{prop2-6}, $\kappa\in \mathcal{I}(\Omega)$ is equivalent to
\begin{equation}\label{3-26}
  \inf_{r>0,~x\in \overline{\Omega}}\frac{\Lambda(x,r)}{r^{\kappa}}>0.
\end{equation}
Using Proposition \ref{prop2-8}, we get $\Lambda(\mathscr{A}(x),r)=\Lambda(x,r)$, which yields
\begin{equation*}
\inf_{r>0,~y\in \overline{\mathscr{A}(\Omega)}}\frac{\Lambda(y,r)}{r^{\kappa}}=  \inf_{r>0,~x\in \overline{\Omega}}\frac{\Lambda(\mathscr{A}(x),r)}{r^{\kappa}}=\inf_{r>0,~x\in \overline{\Omega}}\frac{\Lambda(x,r)}{r^{\kappa}}.
\end{equation*}
Observing that $\Lambda(\delta_{t}(x),tr)=t^{Q}\Lambda(x,r)$, we get
\begin{equation*}
\inf_{r>0,~y\in \overline{\delta_{t}(\Omega)}}\frac{\Lambda(y,r)}{r^{\kappa}}=\inf_{r>0,~x\in \overline{\Omega}}\frac{\Lambda(\delta_{t}(x),r)}{r^{\kappa}}=t^{Q-\kappa}\inf_{r>0,~x\in \overline{\Omega}}\frac{\Lambda(x,r)}{r^{\kappa}}.
\end{equation*}
 Consequently, $\mathcal{I}(\Omega)=\mathcal{I}(\mathscr{A}(\Omega))=\mathcal{I}(\delta_{t}(\Omega))$.
\end{proof}

\begin{lemma}
\label{lemma3-5}
If the asymptotic singular
dimension  satisfies $\nu_{\rm sing}= \max_{z \in \Pi_\infty(\Omega)} \nu(z) \le \tilde{\nu}$, then 
 \[\mathcal{I}(\Omega)=[\tilde{\nu}, Q].\]
\end{lemma}
\begin{proof}
If $\tn=Q$, Lemma \ref{lemma3-3} gives that $\mathcal{I}(\Omega)=\{Q\}$. Let us examine the case that $\tn\leq Q-1$. Define the continuous function 
\[ g(x):=\sum_{k=n}^{\tn}f_k(x),\]
where $f_{k}(x)=\sum_{d(I)=k}|\lambda_I(x)|$ is the function introduced in Proposition \ref{prop2-6}. We then claim that $g(x)>0$ on $\overline{\Omega}$. 

Suppose there exists $x_0\in \overline{\Omega}$ such that $g(x_0)=0$. Then for $0<r<1$, by \eqref{2-8} we have
\[0<\Lambda(x_0,r)=\sum_{k=n}^{Q}f_k(x_{0})r^k\leq \sum_{k=n}^{\tn}f_k(x_0)+\sum_{k=\tn+1}^Q f_k(x_0)r^k=\sum_{k=\tn+1}^Q f_k(x_0)r^k. \]
It derives from Proposition \ref{prop2-8} that
\begin{equation}\label{3-A-25}
\begin{aligned}
\tn=\max_{x\in\overline{\Omega}}\nu(x)&\geq \nu(x_0)=\lim_{r\to 0^+}\frac{\log\Lambda(x_0,r)}{\log r}\geq \lim_{r\to 0^+}\frac{\log\left(\sum_{k=\tn+1}^Q f_k(x_0)r^k\right)}{\log r}\\
&= \min\{k|f_k(x_0)\neq 0,~\tn+1\leq k\leq Q\}\geq \tn+1,
\end{aligned}
\end{equation}
which yields a contradiction. Thus, $g(x)>0$ on $\overline{\Omega}$.

If $\overline{\Omega}$ is bounded, then $g(x)\geq C>0$ on $\overline{\Omega}$. Now, we consider the case that 
 $\overline{\Omega}$ is unbounded. To capture the asymptotic behavior, recall that
\[ \Pi_\infty(\Omega) = \bigcap_{R>0} Z_R, \quad \text{where} \quad Z_R = \overline{\left\{\delta_{\frac{1}{d(x)}}(x) \mathrel{\Big|} x \in \Omega, d(x) > R\right\}}\neq \varnothing. \]
From Proposition \ref{prop2-9}, for any $x_0\in \mathbb{R}^n$ and $r>0$, the closed subunit ball $\overline{B(x_0,r)}$ is compact. Thus,
 $\partial B(0,1)$ is a compact set, and $\{Z_R\}_{R>0}$ forms a decreasing family of compact subsets. 

Since $\nu_{\rm sing} = \max_{z \in \Pi_\infty(\Omega)} \nu(z) \leq \tn$, we have $\nu(z) \leq \tn$ for all $z \in \Pi_\infty(\Omega)$. The same argument as \eqref{3-A-25} ensures that $g(z) > 0$ for all $z \in \Pi_\infty(\Omega)$. Let $U:= \{z \in \partial B(0,1)| g(z) > 0\}$. Since $g$ is a continuous function, $U$ is an open set in $\partial B(0,1)$ that contains $\Pi_\infty(\Omega) = \bigcap_{R>0} Z_R$. By the finite intersection property of compact sets, there exists a sufficiently large $R_0 \geq 1$ such that $Z_{R_0} \subset U$. Since $Z_{R_0}$ is a compact subset of $U$ and $g>0$ on $U$, there exists a constant $c_0 > 0$ such that 
\begin{equation}\label{3-27}
g(z) \geq c_0 > 0 \qquad \forall z \in Z_{R_0}.
\end{equation}

Now, we partition $\overline{\Omega}$ as
\[ \overline{\Omega}=(\overline{\Omega}\cap \overline{B(0,R_0)})\cup (\overline{\Omega}\setminus \overline{B(0,R_0)}):=\overline{\Omega}_{1}\cup \overline{\Omega}_{2}. \]
Note that $\overline{\Omega}_{1}$ is a compact subset of $\overline{\Omega}$. Since $g(x) > 0$ on $\overline{\Omega}$, it attains a positive minimum on $\overline{\Omega}_1$:
\begin{equation}\label{3-28}
  c_1 := \min_{x \in \overline{\Omega}_1} g(x) > 0.
\end{equation}
 For any $x\in \overline{\Omega}_2$, we have $d(x) > R_0 \geq 1$, and  there exists a sequence $\{x_j\} \subset \Omega$ such that $x_j \to x$. For $j$ large enough, $d(x_j) > R_0$, which implies $\delta_{\frac{1}{d(x_j)}}(x_j) \in \{ \delta_{\frac{1}{d(y)}}(y)| y \in \Omega, d(y) > R_0 \}$. Taking the limit as $j \to \infty$, the continuity of dilation and subunit metric yields that the projected point $z := \delta_{\frac{1}{d(x)}}(x)$ belongs to $Z_{R_0}$. By the $\delta_t$-homogeneity of $f_k$ from Proposition \ref{prop2-6}, we have $f_k(\delta_t(x)) = t^{Q-k} f_k(x)$, which gives $f_k(x) = d(x)^{Q-k} f_k(z)$. Since $d(x) > 1$ and $Q-k \geq 0$ for $n \leq k \leq \tn$, it follows that $d(x)^{Q-k} \ge 1$. Therefore, we obtain
\begin{equation}\label{3-29}
g(x) = \sum_{k=n}^{\tn} d(x)^{Q-k} f_k(z) \geq \sum_{k=n}^{\tn} f_k(z) = g(z) \geq c_0 > 0 \qquad \forall x \in \overline{\Omega}_2.
\end{equation}
Combining \eqref{3-28} and \eqref{3-29}, we deduce that
\begin{equation}\label{3-30}
  g(x)\geq \min\left\{c_0, c_1 \right\}>0\qquad \forall x\in \overline{\Omega}.
\end{equation}

Owing to Proposition \ref{prop2-6} and \eqref{3-30}, for any $x\in \overline{\Omega}$ and $0<r<1$, we obtain
\begin{equation*}
  |B(x,r)|\geq C_{1}\sum_{k=n}^{\tn}f_k(x)r^{k}\geq  C_{1} r^{\tn}g(x)\geq  C r^{\tn}.
\end{equation*}
Additionally, for any $x\in \overline{\Omega}$ and $r\geq 1$, we have
$ |B(x,r)|\geq  C_{1}f_{Q}(0)r^{Q} \geq C_1f_{Q}(0)r^{\tn}$. As a result, $\tn\in \mathcal{I}(\Omega)$, and Lemma \ref{lemma3-3} indicates that $\mathcal{I}(\Omega)=[\tn,Q]$.
\end{proof}

Finally, we present an example illustrating the existence of homogeneous H\"{o}rmander vector fields $X$ and an unbounded domain $\Omega$ such that $\tilde{\nu}<\inf\mathcal{I}(\Omega)<Q$, with $\inf\mathcal{I}(\Omega)\notin \mathcal{I}(\Omega)$. Moreover, $\inf\mathcal{I}(\Omega)$ may be non-integer. 

\begin{ex}
\label{example3-1}
  Consider the following smooth vector fields defined on $\mathbb{R}^3$:
  \[ X_{1}=\partial_{x_{1}},\qquad X_{2}=x_{1}\partial_{x_{2}},\qquad X_{3}=x_{2}\partial_{x_{3}}. \]
These vector fields satisfy conditions (\hyperref[H1]{H.1}) and (\hyperref[H2]{H.2}). The associated dilations is given by $\delta_t(x)=(tx_1,t^2x_2,t^3x_3)$ and the homogeneous dimension $Q=6$.
  
  Let $\Omega=\{x=(x_1,x_2,x_3)\in \mathbb{R}^3 \mid x_2>f(x_1)>0, x_1\in \mathbb{R}^{+}, x_3\in \mathbb{R}\}$ be an unbounded domain, where
\[ f(x_{1})=\left\{
                \begin{array}{ll}
               \frac{1}{x_1^\beta}\cdot \frac{1}{(|\log x_1|+1)}   , & \hbox{$0<x_{1}<1$;} \\[2mm]
                   \frac{1}{x_1^\beta}, & \hbox{$x_{1}\geq 1$.}
                \end{array}
              \right. \]
              with $0<\beta\leq\frac{1}{2}$. It follows that $\overline{\Omega}=\{x=(x_1,x_2,x_3)\in \mathbb{R}^3\mid x_2\geq f(x_1)>0, x_1\in \mathbb{R}^{+}, x_3\in \mathbb{R}\}$, and Proposition \ref{prop2-6} gives that
\begin{equation}
|B(x,r)|\approx|x_{1}x_{2}|r^3+(|x_1|^2+|x_2|)r^4+|x_1|r^5+r^6,
\end{equation}
and $\tn=3$. For any $\varepsilon\in(0,\beta)$ and $x=(x_1,x_2,x_3)\in \overline{\Omega}$, we have
  \[ 
  \begin{aligned}
\frac{|B(x,r)|}{r^{4-\beta+\varepsilon}}&\approx\frac{|x_{1}x_{2}|}{r^{1-\beta+\varepsilon}}+(|x_1|^2+|x_2|)r^{ \beta-\varepsilon}+|x_1|r^{1+\beta-\varepsilon}+r^{2+\beta-\varepsilon}\\
&\geq \frac{|x_{1}x_{2}|}{r^{1-\beta+\varepsilon}}+(|x_1|^2+|x_2|)r^{ \beta-\varepsilon}\geq C_{\beta,\varepsilon}|x_{1}x_{2}|^{\beta-\varepsilon}(|x_1|^2+|x_2|)^{1-\beta+\varepsilon}\\
  &\geq C_{\beta,\varepsilon}|x_{1}f(x_1)|^{\beta-\varepsilon}(|x_1|^2+|f(x_1)|)^{1-\beta+\varepsilon}\\
  &\geq  \left\{
                \begin{array}{ll}
                C_{\beta,\varepsilon} \frac{ 1 }{x_{1}^{ \varepsilon}(|\log x_1|+1)}   , & \hbox{$0<x_{1}<1$;} \\[2mm]
                 C_{\beta,\varepsilon}  x_{1}^{(1-\beta)(\beta-\varepsilon)}, & \hbox{$x_{1}\geq 1$.}
                \end{array}
              \right.\geq  \left\{
                \begin{array}{ll}
           C_{\beta,\varepsilon}  \varepsilon e^{1-\varepsilon}  , & \hbox{$0<x_{1}<1$;} \\[2mm]
              C_{\beta,\varepsilon}  , & \hbox{$x_{1}\geq 1$.}
                \end{array}
              \right. ,
  \end{aligned}\]
  which implies that $4-\beta+\varepsilon\in\mathcal{I}(\Omega)$. 

On the other hand, 
\[ \frac{|B(x,r)|}{r^{4-\beta}}\approx\frac{|x_{1}x_{2}|}{r^{1-\beta}}+(|x_1|^2+|x_2|)r^{ \beta}+|x_1|r^{1+\beta}+r^{2+\beta}.\]
Letting $x_2= f(x_1)=\frac{1}{x_{1}^\beta(|\log x_1|+1)}$ and $r=x_1$, we obtain
\[ \frac{|B(x,r)|}{r^{4-\beta}}\approx \frac{2}{(|\log x_1|+1)} +3|x_1|^{2+\beta}\to 0~~\mbox{as}~~x_{1}\to 0^{+}, \]
which means $4-\beta\notin \mathcal{I}(\Omega)$. Since $4-\beta+\varepsilon\in\mathcal I(\Omega)$ for every $\varepsilon\in(0,\beta)$, while Lemma \ref{lemma3-3} shows that $\mathcal I(\Omega)$ is upward closed in $(0,Q]$, no exponent smaller than $4-\beta$ can belong to $\mathcal I(\Omega)$; otherwise $4-\beta$ would also belong to $\mathcal I(\Omega)$.  As a result, $\inf\mathcal{I}(\Omega)=4-\beta$ is not an integer. Moreover,     $\inf\mathcal{I}(\Omega)\notin \mathcal{I}(\Omega)$ and  $\inf\mathcal{I}(\Omega)>\tn=3$.  
\end{ex}

\subsection{Proof of Theorem \ref{thm3}}

The proof of Theorem \ref{thm3} is completed by the following Lemmas \ref{lemma3-6}-\ref{lemma3-11}.
\begin{lemma}
\label{lemma3-6}
$\mathcal{S}(\Omega)$ is a non-empty interval of the form either $(\inf\mathcal{S}(\Omega),Q]$ or $[\inf\mathcal{S}(\Omega), Q]$, and $\inf\mathcal{S}(\Omega)\geq \max_{x\in\Omega}\nu(x)$.
\end{lemma}
\begin{proof}
Lemmas \ref{lemma3-1} and \ref{lemma3-3} give that $Q\in \mathcal{I}(\Omega)\subset \mathcal{S}(\Omega)$. 
If $\kappa\in \mathcal{S}(\Omega)$ and $\kappa< Q$, then for any $\alpha\in (\kappa,Q)$, we have 
$ \frac{\alpha}{\alpha-1}=\frac{Q}{Q-1}(1-s)+\frac{\kappa}{\kappa-1}s$,  
where $s=\frac{\frac{\alpha}{\alpha-1}-\frac{Q}{Q-1}}{\frac{\kappa}{\kappa-1}-\frac{Q}{Q-1}}\in (0,1)$. By H\"{o}lder's inequality,
\[ \int_{\Omega}|u|^{\frac{\alpha}{\alpha-1}}dx\leq \left(\int_{\Omega}|u|^{\frac{\kappa}{\kappa-1}}dx\right)^{s} \left(\int_{\Omega}|u|^{\frac{Q}{Q-1}}dx\right)^{1-s}\leq C\left(\int_{\Omega}|Xu|dx\right)^{\frac{\alpha}{\alpha-1}}\quad \forall~ u\in C_{0}^{\infty}(\Omega).   \]
Thus, $\alpha\in \mathcal{S}(\Omega)$ and $[\kappa,Q]\subset \mathcal{S}(\Omega)$. This implies that $\mathcal{S}(\Omega)$ is a  non-empty interval of the form either $(\inf\mathcal{S}(\Omega),Q]$ or $[\inf\mathcal{S}(\Omega), Q]$.

We next prove that $\inf\mathcal{S}(\Omega)\geq \max_{x\in\Omega}\nu(x)$. Fix $\kappa\in\mathcal S(\Omega)$ and $x\in\Omega$. By the proof of Lemma \ref{lemma3-2}, there is a constant $c_\kappa>0$, independent of $x$ and $r$, such that
\[
 |B(x,r)|\geq c_\kappa r^\kappa
 \qquad 0<r<d(x,\Omega^c).
\]
On the other hand, Propositions \ref{prop2-6} and \ref{prop2-8} imply that, for $0<r\leq1$,
\[
 |B(x,r)|
 \leq C_2\sum_{\ell=\nu(x)}^Q f_\ell(x)r^\ell
 \leq C_x r^{\nu(x)},
 \qquad
 C_x:=C_2\sum_{\ell=\nu(x)}^Q f_\ell(x).
\]
For $0<r<\min\{1,d(x,\Omega^c)\}$, we therefore have
\[
 c_\kappa\leq C_x r^{\nu(x)-\kappa}.
\]
Letting $r\to 0$ shows that $\kappa\geq\nu(x)$. Since $x\in\Omega$ was arbitrary,
\[
 \kappa\geq\max_{x\in\Omega}\nu(x)
 \qquad\forall \kappa\in\mathcal S(\Omega).
\]
Taking the infimum over $\kappa\in\mathcal S(\Omega)$ proves the claim.
\end{proof}

\begin{lemma}
\label{lemma3-7}
For any $\mathscr{A}\in \mathcal{G}$ and $t>0$,  we have $\mathcal{S}(\Omega)=\mathcal{S}(\mathscr{A}(\Omega))=\mathcal{S}(\delta_{t}(\Omega))$.  
\end{lemma}
\begin{proof}
 $\kappa\in \mathcal{S}(\Omega)$ is equivalent to
\begin{equation}\label{3-31}
 \left(\int_{\Omega}|u|^{\frac{\kappa}{\kappa-1}}dx\right)^{\frac{\kappa-1}{\kappa}}\leq C\int_{\Omega}|Xu|dx \qquad\forall~ u\in C_0^\infty(\Omega).
\end{equation}
For any $v\in C_0^\infty(\mathscr{A}(\Omega))$, we have $v(\mathscr{A}(\cdot))\in C_0^\infty(\Omega)$. Then, \eqref{3-31} derives that
\begin{equation*}
\begin{aligned}
\left(\int_{\mathscr{A}(\Omega)}|v(x)|^{\frac{\kappa}{\kappa-1}}dx\right)^{\frac{\kappa-1}{\kappa}}&=\left(\int_{\Omega}|v(\mathscr{A}(x))|^{\frac{\kappa}{\kappa-1}}dx\right)^{\frac{\kappa-1}{\kappa}}\leq C\int_{\Omega}|X(v(\mathscr{A}(x)))|dx\\
&=C\int_{\mathscr{A}(\Omega)}|Xv|dx.
\end{aligned}
\end{equation*}
Similarly, for any $v\in C_0^\infty(\delta_{t}(\Omega))$, we get from \eqref{3-31} that
\begin{equation*}
\begin{aligned}
\left(\int_{\delta_{t}(\Omega)}|v(x)|^{\frac{\kappa}{\kappa-1}}dx\right)^{\frac{\kappa-1}{\kappa}}&=t^{Q\frac{\kappa-1}{\kappa}}\left(\int_{\Omega}|v(\delta_{t}(x))|^{\frac{\kappa}{\kappa-1}}dx\right)^{\frac{\kappa-1}{\kappa}}\leq Ct^{Q\frac{\kappa-1}{\kappa}}\int_{\Omega}|X(v(\delta_{t}(x)))|dx\\
&=Ct^{1-\frac{Q}{\kappa}}\int_{\delta_{t}(\Omega)}|Xv|dx.
\end{aligned}
\end{equation*}
Applying the same arguments to $\mathscr A^{-1}\in\mathcal G$ and to $\delta_{\frac{1}{t}}$ gives the reverse inclusions. Consequently, $\mathcal{S}(\Omega)=\mathcal{S}(\mathscr{A}(\Omega))=\mathcal{S}(\delta_{t}(\Omega))$.
\end{proof}

\begin{lemma}
\label{lemma3-8}
Suppose there exists a point $x_0 \in \overline{\Omega}$ with $\nu(x_0) = \tilde{\nu}$ satisfying the interior corkscrew condition described in Theorem \ref{thm3} (\hyperref[S3]{S3}). Then $\inf \mathcal{S}(\Omega) \ge \tilde{\nu}$.
\end{lemma}
\begin{proof}
Assume by contradiction that there exists $\kappa \in \mathcal{S}(\Omega)$ with $\kappa < \tilde{\nu}$. 
By Lemma \ref{lemma3-2}, we know $\mathcal{S}(\Omega)\subset\mathcal{I}_{\mathrm{int}}(\Omega)$. Thus, there exists a constant $C_1 > 0$ such that for any $y \in \Omega$ and any $0 < \rho < d(y, \Omega^c)$,
\begin{equation}\label{eq:Iint-lower}
|B(y, \rho)| \ge C_1 \rho^\kappa.
\end{equation}

By the interior corkscrew condition at $x_0$, there exist $c \in (0, 1)$ and $r_0 > 0$ such that for any $0 < r < r_0$, one can find $y_r \in \Omega$ satisfying $B(y_r, c r) \subset \Omega \cap B(x_0, r)$. Since $B(y_r, c r) \subset \Omega$, we have the subunit distance to the boundary $d(y_r, \Omega^c) \ge cr$. Taking $y = y_r$ and $\rho = \frac{c}{2}r < d(y_r, \Omega^c)$, the estimate \eqref{eq:Iint-lower} yields
\begin{equation}\label{eq:yr-lower}
|B(y_r, \textstyle{\frac{c}{2}}r)| \ge C_1 \left(\frac{cr}{2}\right)^\kappa = C_2 r^\kappa,
\end{equation}
where $C_2 = C_1 (c/2)^\kappa > 0$.  Moreover,  the inclusion $B(y_r, \frac{c}{2}r) \subset B(y_r, cr) \subset B(x_0, r)$  implies
\begin{equation}\label{eq:x0-lower}
|B(x_0, r)| \ge |B(y_r, \textstyle{\frac{c}{2}}r)| \ge C_2 r^\kappa. 
\end{equation}

According to Propositions \ref{prop2-6} and \ref{prop2-8}, we have
 $\Lambda(x_0, r) = \sum_{k=\nu(x_0)}^Q f_k(x_0) r^k$. Since $\nu(x_0) = \tilde{\nu}$, there exists a constant $C_3 > 0$ such that for all $0 < r \le \min\{1, r_0\}$,
\begin{equation}\label{eq:x0-upper}
|B(x_0, r)| \le C_3 \Lambda(x_0, r) \le C_3 r^{\tilde{\nu}} \sum_{k=\tilde{\nu}}^Q f_k(x_0) = C_4 r^{\tilde{\nu}},
\end{equation}
where $C_4 = C_3 \sum_{k=\tilde{\nu}}^Q f_k(x_0) > 0$. Combining \eqref{eq:x0-lower} and \eqref{eq:x0-upper}, we obtain for all sufficiently small $r > 0$,
$$ C_2 r^\kappa \le |B(x_0, r)| \le C_4 r^{\tilde{\nu}}, $$
which leads to
$ r^{\tilde{\nu} - \kappa} \ge \frac{C_2}{C_4} > 0 $. 
Since we assumed $\kappa < \tilde{\nu}$, taking the limit as $r \to 0^+$ gives $0 \ge \frac{C_2}{C_4} > 0$, which is a contradiction. Therefore, $\inf\mathcal{S}(\Omega) \ge \tilde{\nu}$.
\end{proof}

\begin{lemma}
\label{lemma3-9}
Assume that $\Omega \subset \mathbb{R}^n$ is an open set and its boundary $\partial\Omega$ is locally $C^1$ near a point $x_0 \in \partial\Omega$. If $x_0$ is a non-characteristic point with respect to $X = (X_1, \ldots, X_m)$, i.e., there exists $1\leq j\leq m$ such that $X_j(x_0)\not\in T_{x_0}(\partial\Omega)$, then $x_0$ satisfies the interior corkscrew condition described in  (\hyperref[S3]{S3}).  Consequently, if $\nu(x_0) = \tilde{\nu} = \max_{x \in \overline{\Omega}} \nu(x)$, then $\inf \mathcal{S}(\Omega) \ge \tilde{\nu}$.
\end{lemma}
\begin{proof}
Since $\partial\Omega$ is locally $C^1$ near $x_0$, there exist an open neighborhood $U$ of $x_0$ and a $C^1$ function $\phi: U \to \mathbb{R}$ such that $\Omega \cap U = \{x \in U |\phi(x) > 0\}$, $\partial\Omega \cap U = \{x \in U|\phi(x) = 0\}$, and the Euclidean gradient $\nabla \phi(x) \neq 0$ for all $x \in U$.

Because $x_0\in \partial\Omega$ is a non-characteristic point with respect to $X$, there exists at least one vector field $X_j$ such that $X_j\phi(x_0) \neq 0$. Let
\[ Y = \sum_{j=1}^m a_j X_j \]
such that $a_j = \frac{X_j\phi(x_0)}{\sqrt{\sum_{i=1}^m |X_i\phi(x_0)|^2}}$ satisfying $\sum_{j=1}^m a_j^2 = 1$. By construction, $Y$ points strictly into the interior of $\Omega$ at $x_0$, and we have
\begin{equation}\label{eq:transversal}
Y\phi(x_0) = \left( \sum_{i=1}^m |X_i\phi(x_0)|^2 \right)^{\frac{1}{2}} := \alpha > 0.
\end{equation}

Let $\gamma(t)=\exp(tY)(x_0)$ be the integral curve of $Y$ starting from $x_0$. Since $Y$ is a linear combination of $X_1, \ldots, X_m$ with coefficients satisfying $\sum_{j=1}^m a_j^2 = 1$,  the subunit distance satisfies 
\begin{equation}\label{eq:subunit_dist_gamma}
d(x_0, \gamma(t)) \le t \quad \forall t \ge 0.
\end{equation}
By the first-order Taylor expansion of the $C^1$ function $\phi$ along $\gamma(t)$ at $t=0$, we deduce from \eqref{eq:transversal} that
\begin{equation*}
\phi(\gamma(t)) = \phi(x_0) + t Y\phi(x_0) + o(t) = \alpha t + o(t) \quad \text{as } t \to 0^+.
\end{equation*}
Thus, there exists a sufficiently small $t_0 > 0$ such that $\gamma([0,t_0])\subset U$ and, for all $t \in (0, t_0]$,
\begin{equation}\label{eq:phi_lower_bound}
\phi(\gamma(t)) \ge \frac{\alpha}{2} t.
\end{equation}

On the other hand, since $\phi \in C^1(U)$, by \cite[Proposition 1.37]{Bramanti2024}, there exists a constant $L > 0$ and a small $R_0>0$ such that the subunit ball $B(x_0,R_0)\subset \subset U$ and
\begin{equation}\label{eq:lipschitz}
|\phi(x) - \phi(y)| \le L d(x, y)\qquad \forall x,y\in \overline{B(x_0,R_0)}.
\end{equation}

We now construct the corkscrew balls. Consider the subunit ball $B(\gamma(t), c t)$, where the constant $c=\min\left\{\frac{\alpha}{4L}, \frac{1}{2}\right\} \in (0, \frac{1}{2}]$. Let $ t_1 := \frac{1}{2}\min\left\{ t_0, \frac{R_0}{1+c} \right\} $. We show that $B(\gamma(t), c t)\subset B(x_0,R_0)\cap \Omega$ for $0<t\leq t_1$.

For any $0<t\leq t_1$ and $z\in B(\gamma(t), c t)$, we obtain from \eqref{eq:subunit_dist_gamma} that
\[ d(x_0, z) \le d(x_0, \gamma(t)) + d(\gamma(t), z)\leq (c+1)t<R_0, \]
which means $z\in B(x_0,R_0)$, and $B(\gamma(t), c t)\subset B(x_0,R_0)$ for all $0<t\leq t_1$. Combining \eqref{eq:phi_lower_bound} and \eqref{eq:lipschitz} gives that for any $z\in B(\gamma(t), c t)$ and $0<t\leq t_1$,
\begin{equation*}
\phi(z) \ge \phi(\gamma(t)) - |\phi(z) - \phi(\gamma(t))| \ge \frac{\alpha}{2} t - L d(z, \gamma(t)) > \left(\frac{\alpha}{2} - Lc\right) t\geq \frac{\alpha}{4} t > 0.
\end{equation*}
This indicates that $B(\gamma(t), c t)\subset \Omega$ for $0<t\leq t_1$. Consequently, $B(\gamma(t), c t)\subset B(x_0,R_0)\cap \Omega$ for $0<t\leq t_1$.

To verify the interior corkscrew condition for any arbitrary small radius $0 < r <r_0:= 2t_1$, we set  the center point $y_r = \gamma(\frac{r}{2})$. As established above, $B(y_r, \frac{cr}{2}) \subset \Omega$.
Moreover, for any $z \in B(y_r, \frac{cr}{2})$, the triangle inequality and \eqref{eq:subunit_dist_gamma} yield
\begin{equation*}
d(x_0, z) \le d(x_0, y_r) + d(y_r, z) < \frac{r}{2} + \frac{cr}{2} = \frac{1+c}{2} r < r,
\end{equation*}
where the last inequality follows from $c < 1$. Thus,  $B(y_r, \frac{cr}{2}) \subset \Omega \cap B(x_0, r)$ holds for $0<r<r_0$. 
This ensures that $x_0$ satisfies the interior corkscrew condition with corkscrew constant $\frac{c}{2}\in (0,1)$. 

If $\nu(x_0) = \tilde{\nu}$, applying Lemma \ref{lemma3-8} immediately yields $\inf\mathcal{S}(\Omega) \ge \tilde{\nu}$.
\end{proof}

\begin{lemma}
\label{lemma3-10}
If $0 \in \overline{\Omega}$ and there exists a non-empty open set $U \subset \Omega$ such that $\delta_t(U) \subset \Omega$ for all $0 < t \leq 1$, then the origin $0$ satisfies the interior corkscrew condition described in  (\hyperref[S3]{S3}), and
    \[
    \mathcal{S}(\Omega) = \{Q\}.
    \]
    In particular, $\mathcal{S}(\mathbb{R}^n) = \{Q\}$.
\end{lemma}
\begin{proof}

Since $U \subset \Omega$ is a non-empty open set, there exist $z\in U\setminus\{0\}$ and $\rho > 0$ such that the subunit ball $B(z, \rho) \subset U$. Let $R = d(0, z)$. By Proposition \ref{prop2-5},   we have $\delta_t(B(z, \rho)) = B(\delta_t(z), t\rho)$. Since $\delta_t(U) \subset \Omega$ for all $t \in (0, 1]$, we obtain
\begin{equation*}
B(\delta_t(z), t\rho) \subset \Omega \quad \forall t \in (0, 1].
\end{equation*}
For any  $w \in B(\delta_t(z), t\rho)$, the triangle inequality yields
\begin{equation*}
d(0, w) \le d(0, \delta_t(z)) + d(\delta_t(z), w) < t d(0, z) + t\rho = t(R + \rho).
\end{equation*}

Let $r_0 = R + \rho$ and $c = \frac{\rho}{R + \rho} \in (0, 1)$. We then set $y_r = \delta_{\frac{r}{R + \rho}}(z)$ with $\frac{r}{R + \rho} \in (0, 1)$. For any $w \in B(y_r, cr)$, we have $d(0, w) < r$, which implies $B(y_r, cr) \subset B(0, r)$. Coupled with the fact that $B(y_r, cr) \subset \Omega$, we conclude that
\begin{equation*}
B(y_r, cr) \subset \Omega \cap B(0, r) \quad \forall 0 < r < r_0.
\end{equation*}
This verifies the interior corkscrew condition at the origin. Since $0\in\overline\Omega$ and $\nu(0)=Q=\tilde\nu$, Lemma \ref{lemma3-8} gives $\inf\mathcal S(\Omega)\geq Q$. On the other hand, Lemmas \ref{lemma3-1} and \ref{lemma3-3} give $Q\in\mathcal I(\Omega)\subset\mathcal S(\Omega)$, while by definition $\mathcal S(\Omega)\subset(1,Q]$. Therefore $\mathcal S(\Omega)=\{Q\}$.
\end{proof}

\begin{lemma}
\label{lemma3-11}
If the asymptotic singular
dimension  satisfies $\nu_{\rm sing}=\max_{z \in \Pi_\infty(\Omega)} \nu(z) \le \tilde{\nu}$ and  the conditions stated  in (\hyperref[S3]{S3}) holds, then
    \[
    \mathcal{S}(\Omega) = [\tilde{\nu}, Q].
    \]
\end{lemma}
\begin{proof}
Theorems \ref{thm1} and \ref{thm2} give that $[\tn,Q]= \mathcal{I}(\Omega)\subset \mathcal{S}(\Omega)$. Moreover, if the conditions stated  in (\hyperref[S3]{S3}) holds, Lemma \ref{lemma3-8} yields that $\inf\mathcal{S}(\Omega)\geq \tn$, which derives that $\mathcal{S}(\Omega) = [\tilde{\nu}, Q]$. 
\end{proof}

\subsection{Proof of Corollary \ref{corollary1-1}}

\begin{proof}[Proof of Corollary \ref{corollary1-1}]
 If the family of vector fields $X=(X_{1},\ldots,X_m)$ is equiregular on $\mathbb{R}^n$, then 
 $\nu(x)\equiv Q$. By Propositions \ref{prop2-6}, we obtain $\mathcal{I}(\Omega)=\{Q\}$. Moreover, Theorem \ref{thm3} indicates that $\mathcal{S}(\Omega)=\{Q\}$, which yields conclusion (1).
 
 We next prove conclusion (2). Assume that  $\Omega$ is an exterior domain. Then it is unbounded, and there exists $R_0>0$ such that $\mathbb{R}^n\setminus B(0,R_0)\subset \Omega$. For any  $z \in \partial B(0,1)$ and any $R > 0$, choose a scaling factor $t > \max\{R, R_0\}$. The dilated point $x = \delta_t(z)$ satisfies $d(x) = t > R_0$, which guarantees $x \in \mathbb{R}^n\setminus B(0,R_0)\subset \Omega$. Moreover, since $d(x) = t > R$, its projection is  $\delta_{\frac{1}{d(x)}}(x) = \delta_{\frac{1}{t}}(\delta_t(z)) = z$. This implies $\Pi_\infty(\Omega) = \partial B(0,1)$. Therefore,
\begin{equation*}
\nu_{\rm sing} = \max_{z \in \Pi_\infty(\Omega)} \nu(z) = \max_{z \in \partial B(0,1)} \nu(z).
\end{equation*}

For any $z \in \partial B(0,1)$, we can again choose a scaling factor $t > R_0$. The dilated point $x_t = \delta_t(z)$ satisfies $d(x_t) = t > R_0$, ensuring $x_t \in \Omega$. According to Proposition \ref{prop2-8}, we have $\nu(z) = \nu(\delta_t(z)) = \nu(x_t)$. Since $x_t \in \Omega \subset \overline{\Omega}$, it trivially follows that $\nu(x_t) \le \max_{x \in \overline{\Omega}} \nu(x) = \tilde{\nu}$. Taking the maximum over all $z \in \partial B(0,1)$, we obtain
$ \nu_{\rm sing} \le \tilde{\nu}$. From Theorems \ref{thm1} and \ref{thm2}, we conclude $ \mathcal{I}(\Omega)=[\tilde{\nu},Q]\subset \mathcal{S}(\Omega)$. 
 
 For $\Omega=\mathbb R^n\setminus\{0\}$, take $U=\Omega$. Then $0\in\overline\Omega$ and $\delta_t(U)=U\subset\Omega$ for every $0<t\leq1$. Theorem \ref{thm3} therefore gives
\[
 \mathcal S(\Omega)=\{Q\}=[\tilde\nu,Q],
\]
because $\overline\Omega=\mathbb R^n$ and hence $\tilde\nu=Q$.

If there exists $x_0\in\overline\Omega$ with $\nu(x_0)=\tilde\nu$ satisfying the corkscrew condition in (\hyperref[S3]{S3}), then the already established inequality $\nu_{\rm sing}\leq\tilde\nu$ and Theorem \ref{thm3} (\hyperref[S6]{S6}) yield
\[
 \mathcal S(\Omega)=[\tilde\nu,Q].
\]

Finally, suppose that $0\in K^\circ$. Choose $x_0\in\overline\Omega$ such that $\nu(x_0)=\tilde\nu$. Since $0\notin\overline\Omega$, we have $x_0\neq0$. Because $K$ is compact, $\delta_{t_0}(x_0)\notin K$ for all sufficiently large $t_0$; fix such a $t_0$ and set $z:=\delta_{t_0}(x_0)\in\Omega$. By \eqref{2-12},
\[
 \nu(z)=\nu(x_0)=\tilde\nu.
\]
Moreover, $r_*:=d(z,\Omega^c)>0$, and for every $0<r<r_*$,
\[
 B(z,r/2)\subset\Omega\cap B(z,r).
\]
Thus the interior point $z$ satisfies the corkscrew condition in (\hyperref[S3]{S3}). Applying Theorem \ref{thm3} (\hyperref[S6]{S6}) again gives $\mathcal S(\Omega)=[\tilde\nu,Q]$, completing the proof of conclusion (2).

If $\Omega$ is a star-shaped domain with respect to the origin, then $0\in \Omega$ and $\nu(0)=\tn$. By Theorems \ref{thm2} and \ref{thm3}, we obtain $\mathcal{I}(\Omega)=\mathcal{S}(\Omega)=\{Q\}$, which proves conclusion (3). 
\end{proof}

\section{Analysis of the optimal Sobolev constant}
\label{Section4}

This section is dedicated to proving Theorems \ref{thm4} and \ref{thm5}. The proofs rely on an analysis of the optimal Sobolev constant, specifically its attainability and dependence on the domain.

\subsection{Refined concentration-compactness lemma}

We first construct a refined concentration-compactness lemma, which we then utilize to demonstrate the existence of Sobolev extremal functions. To this end, we recall some preliminary results from \cite[Section 1.9]{Willem2012}.

Let $U$ be an open subset of $\mathbb{R}^n$. We write $C_{0}(U)$ for the space of continuous functions on $U$ vanishing at infinity and denote by $\mathscr{M}(U)$ the Banach space of finite signed Radon measures on $U$, equipped with the total-variation norm
\[
 \|\mu\|:=\sup_{f\in C_{0}(U),\,\|f\|_{L^{\infty}(U)}\leq 1}
 \left|\int_{U}f\,d\mu\right|=|\mu|(U).
\]
 A sequence $\{\mu_k\}_{k=1}^{\infty}\subset \mathscr{M}(U)$ is said to converge vaguely to $\mu\in \mathscr{M}(U)$, denoted by
$ \mu_k\rightharpoonup \mu$ in $\mathscr{M}(U)$, if
\[ \int_{U}gd\mu_k\to \int_{U}gd\mu\quad \mbox{as}\qquad k \to \infty\quad \mbox{for every}~~g\in C_{0}(U).\]

If $\mu\in \mathscr{M}(U)$ such that $\int_{U}fd\mu \geq 0$ for any non-negative $f\in C_{0}(U)$, then we say $\mu$ is a positive Radon measures on $U$ with finite mass. We also note $\mathscr{M}^{+}(U)$ the cone of positive finite Radon measures over $U$, and $\delta_{x}$ the Dirac mass at point $x$. For each $\mathscr{M}^{+}(U)$, 
\[ \|\mu\|=\sup_{f\in C(U),~\|f\|_{L^{\infty}(U)}=1}\left|\int_{U}fd\mu\right|=\mu(U). \]

Next, we recall the following lemma due to Lions \cite{Lions1985}.
\begin{lemma}[see {\cite[Lemma 1.2]{Lions1985}}]
\label{lemma4-1}
Let $\omega_{1},\omega_2\in \mathscr{M}^{+}(\mathbb{R}^n)$ such that for some constant $A> 0$,
\[ \left(\int_{\mathbb{R}^n}|\varphi|^qd\omega_{1}\right)^{\frac{1}{q}}\leq A\left(\int_{\mathbb{R}^n}|\varphi|^pd\omega_{2}\right)^{\frac{1}{p}}\qquad \forall~ \varphi\in C_{0}^{\infty}(\mathbb{R}^n), \]
where $1\leq p<q< +\infty$. Then, there exists an at most countable set $J$, families $\{x_j\}_{j\in J}$ of distinct points in $\mathbb{R}^n$, $\{a_{j}\}_{j\in J}$ in $(0,+\infty)$ such that 
\[ \omega_{1}=\sum_{j\in J}a_{j}\delta_{x_{j}},\qquad \omega_{2}\geq A^{-p}\sum_{j\in J}a_{j}^{\frac{p}{q}}\delta_{x_{j}} \]
and
$\sum_{j\in J}a_{j}^{\frac{p}{q}}<\infty$. 
If $\|\omega_{1}\|^{\frac{1}{q}}\geq A\|\omega_{2}\|^{\frac{1}{p}}$, then either $\omega_1=\omega_2=0$, or $J$ is a singleton and there exist $x_0\in \mathbb{R}^n$ and $\gamma>0$ such that 
\[ \omega_{1}=\gamma\delta_{x_{0}},\qquad \omega_{2}=A^{-p}\gamma^{\frac{p}{q}}\delta_{x_{0}}.\]
\end{lemma}

By means of Lemma \ref{lemma4-1} and Proposition \ref{prop2-10}, we establish the following refined concentration-compactness lemma, which is adapted to the space $\mathcal{M}_{X,0}^{p,Q}(\mathbb{R}^n)$.
\begin{lemma}[Refined concentration-compactness lemma]
\label{lemma4-2}
Assume conditions (\hyperref[H1]{H.1}) and (\hyperref[H2]{H.2}). For $Q\geq 3$ and $1\leq p<Q$, we let $\{u_k\}_{k=1}^{\infty}\subset \mathcal{M}_{X,0}^{p,Q}(\mathbb{R}^n)$ be a sequence such that
\begin{enumerate}
  \item $u_k\to u$ a.e. on $\mathbb{R}^n$;
  \item $u_k\rightharpoonup u$ weakly in $\mathcal{M}_{X,0}^{p,Q}(\mathbb{R}^n)$;
  \item $|X(u_k-u)|^{p}dx\rightharpoonup \mu$ in $\mathscr{M}^{+}(\mathbb{R}^n)$;
  \item $ |Xu_k|^{p}dx\rightharpoonup \widetilde{\mu}$ in $\mathscr{M}^{+}(\mathbb{R}^n)$;
  \item $|u_k-u|^{p_{Q}^{*}}dx\rightharpoonup \omega$ in $\mathscr{M}^{+}(\mathbb{R}^n)$.
\end{enumerate}
Denote by
\begin{equation}\label{4-1}
  \mu_{\infty}:=\lim_{R\to\infty}\limsup_{k\to\infty}\int_{\mathbb{R}^{n}\setminus B(0,R)}|Xu_k|^{p}dx\quad\mbox{and}\quad \omega_{\infty}:=\lim_{R\to\infty}\limsup_{k\to\infty}\int_{\mathbb{R}^{n}\setminus B(0,R)}|u_k|^{p_{Q}^{*}}dx.
\end{equation}
Then it follows that 
\begin{enumerate}[(1)]
\item $ \|\omega\|^{\frac{p}{p_{Q}^{*}}}\leq C_{0}^{-1}\|\mu\|$;
  \item $\omega=\sum_{j\in J}a_{j}\delta_{x_{j}}$;
  \item $\mu\geq C_{0}\sum_{j\in J}a_{j}^{\frac{p}{{p_{Q}^{*}}}}\delta_{x_{j}}$ and $\sum_{j\in J}a_{j}^{\frac{p}{{p_{Q}^{*}}}}<\infty$;
\item $\omega_{\infty}^{\frac{p}{p_{Q}^{*}}}\leq C_{0}^{-1}\mu_{\infty}$;

\item $\widetilde{\mu}\geq |Xu|^{p}dx+C_{0}\sum_{j\in J}a_{j}^{\frac{p}{{p_{Q}^{*}}}}\delta_{x_{j}}$;
\item $ \limsup_{k\to\infty}\|u_{k}\|_{L^{p_{Q}^{*}}(\mathbb{R}^n)}^{p_{Q}^{*}}=\|u\|_{L^{p_{Q}^{*}}(\mathbb{R}^n)}^{p_{Q}^{*}}+\|\omega\|+\omega_{\infty}$;

\item $\limsup_{k\to\infty}\|Xu_{k}\|_{L^{p}(\mathbb{R}^n)}^{p}= \mu_{\infty}+\|\widetilde{\mu}\|$.

\end{enumerate}
Here, $C_{0}$ is the optimal constant in the Sobolev inequality \eqref{1-12}, as defined in \eqref{1-13}. Moreover, if
$\|\omega\|^{\frac{p}{p_{Q}^{*}}}=C_{0}^{-1}\|\mu\|$,
then either $\mu=\omega=0$, or there exist $x_{0}\in\mathbb{R}^{n}$ and $a>0$ such that
\[
 \omega=a\delta_{x_{0}},
 \qquad
 \mu=C_{0}a^{\frac{p}{p_{Q}^{*}}}\delta_{x_{0}}.
\]
\end{lemma}
\begin{proof}
Substituting $v_k=u_k-u$, it follows that $v_k\in \mathcal{M}_{X,0}^{p,Q}(\mathbb{R}^n)$ and $v_k\rightharpoonup0$ in this space. For any $\phi\in C_{0}^{\infty}(\mathbb{R}^n)$, by \eqref{1-13} we have
\begin{equation}\label{4-2}
  C_{0}^{\frac{1}{p}}\left(\int_{\mathbb{R}^n}|\phi v_{k}|^{p_{Q}^{*}}dx\right)^{\frac{1}{{p_{Q}^{*}}}}\leq \left(\int_{\mathbb{R}^n}|X(\phi v_{k})|^{p}dx\right)^{\frac{1}{p}}.
\end{equation}
An application of the Minkowski inequality then yields
\[ \begin{aligned}
\left(\int_{\mathbb{R}^n}|X(\phi v_{k})|^{p}dx\right)^{\frac{1}{p}}\leq \left(\int_{\mathbb{R}^n}|\phi|^{p} |Xv_{k}|^{p}dx\right)^{\frac{1}{p}} +\left(\int_{\mathbb{R}^n}|v_{k}|^{p} |X\phi|^{p}dx\right)^{\frac{1}{p}}.
\end{aligned}
\]
Since $|v_{k}|^{p_{Q}^{*}}dx\rightharpoonup \omega$ and $|Xv_k|^{p}dx\rightharpoonup \mu$ in $\mathscr{M}^{+}(\mathbb{R}^n)$, we obtain
\[ \lim_{k\to\infty}\left(\int_{\mathbb{R}^n}|\phi v_{k}|^{p_{Q}^{*}}dx\right)^{\frac{1}{{p_{Q}^{*}}}}=\left(\int_{\mathbb{R}^n}|\phi|^{p_{Q}^{*}}d\omega\right)^{\frac{1}{{p_{Q}^{*}}}},~~~\lim_{k\to\infty}\left(\int_{\mathbb{R}^n}|\phi|^{p} |Xv_{k}|^{p}dx\right)^{\frac{1}{p}}=\left(\int_{\mathbb{R}^n}|\phi|^{p} d\mu\right)^{\frac{1}{p}}. \]
Moreover,  since $v_k\rightharpoonup 0$ in $\mathcal{M}_{X,0}^{p,Q}(\mathbb{R}^n)$, we conclude from  Proposition \ref{prop2-10} that $v_k\to 0$ in $L_{\rm loc}^{p}(\mathbb{R}^n)$. As a result, letting $k\to\infty$ in \eqref{4-2} gives
\begin{equation}\label{4-3}
  C_{0}\left(\int_{\mathbb{R}^n}|\phi|^{p_{Q}^{*}}d\omega\right)^{\frac{p}{{p_{Q}^{*}}}}\leq \int_{\mathbb{R}^n}|\phi|^{p} d\mu\qquad \forall \phi\in C_{0}^{\infty}(\mathbb{R}^n).
\end{equation}
Choose $\chi_\ell\in C_0^\infty(\mathbb{R}^n)$ such that $0\leq\chi_\ell\leq1$ and $\chi_\ell\to1$ pointwise. Since $\mu$ and $\omega$ are finite, dominated convergence in \eqref{4-3} gives
\[
 \|\omega\|^{\frac{p}{p_{Q}^{*}}}\leq C_0^{-1}\|\mu\|,
\]
which proves (1). Applying Lemma \ref{lemma4-1} to \eqref{4-3} with $A=C_0^{-\frac{1}{p}}$ gives
\[ \omega=\sum_{j\in J}a_{j}\delta_{x_{j}},\qquad \mu\geq C_{0}\sum_{j\in J}a_{j}^{\frac{p}{{p_{Q}^{*}}}}\delta_{x_{j}}, \qquad \sum_{j\in J}a_{j}^{\frac{p}{{p_{Q}^{*}}}}<\infty,\]
thereby establishing conclusions (2) and (3).

If equality holds in (1), then $\|\omega\|^{\frac{1}{p_{Q}^{*}}}=C_0^{-\frac{1}{p}}\|\mu\|^{\frac{1}{p}}$.
The equality statement in Lemma \ref{lemma4-1} therefore yields either $\mu=\omega=0$, or
\[
 \omega=a\delta_{x_0},
 \qquad
 \mu=C_0a^{\frac {p}{p_{Q}^{*}}}\delta_{x_0}
\]
for some $x_0\in\mathbb{R}^n$ and $a>0$.

Let $R>0$ and choose a cutoff function $\phi_{R}\in C^{1}(\mathbb{R}^n)$ such that
\[ 0\leq\phi_{R}\leq 1,\qquad \mbox{and}\qquad \phi_{R}(x)=\left\{
                                \begin{array}{ll}
                                  1, & \hbox{$x\in \mathbb{R}^n\setminus B(0,R+1)$;} \\
                                  0, & \hbox{$x\in B(0,R)$.}
                                \end{array}
                              \right.\]
Then, we have $\phi_{R} v_{k}\in \mathcal{M}_{X,0}^{p,Q}(\mathbb{R}^n)$, and \eqref{1-13} gives that
\[ \begin{aligned} C_{0}^{\frac{1}{p}}\left(\int_{\mathbb{R}^n}|\phi_{R} v_{k}|^{p_{Q}^{*}}dx\right)^{\frac{1}{{p_{Q}^{*}}}}&\leq \left(\int_{\mathbb{R}^n}|X(\phi_{R} v_{k})|^{p}dx\right)^{\frac{1}{p}}\\
&\leq \left(\int_{\mathbb{R}^n}|v_k|^{p}|X\phi_{R}|^{p}dx\right)^{\frac{1}{p}}+\left(\int_{\mathbb{R}^n}|\phi_{R}|^{p}|Xv_{k}|^{p}dx\right)^{\frac{1}{p}}.
\end{aligned}\]
Observing that $v_k\to 0$ in $L_{\rm loc}^{p}(\mathbb{R}^n)$, we have
\begin{equation}\label{4-4}
 C_{0}\limsup_{k\to\infty}\left(\int_{\mathbb{R}^n}|\phi_{R} v_{k}|^{p_{Q}^{*}}dx\right)^{\frac{p}{{p_{Q}^{*}}}}\leq \limsup_{k\to\infty}\int_{\mathbb{R}^n}|\phi_{R}|^{p}|Xv_{k}|^{p}dx.
\end{equation}

Using the Brezis-Lieb lemma, one has
\[ \int_{\mathbb{R}^n\setminus B(0,R)}|u|^{p_{Q}^{*}}dx=\lim_{k\to \infty}\left(\int_{\mathbb{R}^n\setminus B(0,R)}|u_{k}|^{p_{Q}^{*}}dx-\int_{\mathbb{R}^n\setminus B(0,R)}|v_{k}|^{p_{Q}^{*}}dx \right),\]
which gives that 
\[ \omega_{\infty}=\lim_{R\to+\infty}\limsup_{k\to \infty}\int_{\mathbb{R}^n\setminus B(0,R)}|u_{k}|^{p_{Q}^{*}}dx=\lim_{R\to+\infty}\limsup_{k\to \infty}\int_{\mathbb{R}^n\setminus B(0,R)}|v_{k}|^{p_{Q}^{*}}dx. \]
Furthermore, since
\begin{equation}\label{4-5}
 \int_{\mathbb{R}^n\setminus B(0,R+1)}| v_{k}|^{p_{Q}^{*}}dx\leq\int_{\mathbb{R}^n}|\phi_{R} v_{k}|^{p_{Q}^{*}}dx\leq \int_{\mathbb{R}^n\setminus B(0,R)}| v_{k}|^{p_{Q}^{*}}dx,
\end{equation}
we obtain 
\begin{equation}\label{4-6}
 \omega_{\infty}=\lim_{R\to+\infty}\limsup_{k\to \infty}\int_{\mathbb{R}^n\setminus B(0,R)}|v_{k}|^{p_{Q}^{*}}dx=\lim_{R\to+\infty}\limsup_{k\to \infty}\int_{\mathbb{R}^n}|\phi_{R}v_{k}|^{p_{Q}^{*}}dx.
\end{equation}

For any $\varepsilon>0$ and $p\geq 1$, there exists $C(\varepsilon,p)>0$ such that
\begin{equation}\label{4-7}
  \bigl||\xi+\eta|^p-|\xi|^p\bigr|
 \leq \varepsilon|\xi|^p+C(\varepsilon,p)|\eta|^p
 \qquad \forall\,\xi,\eta\in\mathbb{R}^m.
\end{equation}
Applying \eqref{4-7} with $\xi=Xu_k$ and $\eta=-Xu$, we find
\[\begin{aligned}
&\left|\int_{\mathbb{R}^{n}\setminus B(0,R)}(|Xv_k|^{p}-|Xu_k|^{p})dx\right|\leq \varepsilon \int_{\mathbb{R}^{n}\setminus B(0,R)}|Xu_k|^{p}dx+C(\varepsilon,p)\int_{\mathbb{R}^{n}\setminus B(0,R)}|Xu|^{p}dx.
\end{aligned}\]
Since $u_k\rightharpoonup u$ in $\mathcal{M}_{X,0}^{p,Q}(\mathbb{R}^n)$, we have $\|Xu_k\|_{L^{p}(\mathbb{R}^n)}^{p}+\|Xu\|_{L^{p}(\mathbb{R}^n)}^{p}\leq C$. Thus, 
\[ \left|\int_{\mathbb{R}^{n}\setminus B(0,R)}|Xv_k|^{p}dx-\int_{\mathbb{R}^{n}\setminus B(0,R)}|Xu_k|^{p}dx\right|\leq C\varepsilon+ C(\varepsilon,p)\int_{\mathbb{R}^{n}\setminus B(0,R)}|Xu|^{p}dx,\]
which yields that
\[ \left|\lim_{R\to+\infty}\limsup_{k\to \infty}\int_{\mathbb{R}^{n}\setminus B(0,R)}|Xv_k|^{p}dx-\lim_{R\to+\infty}\limsup_{k\to \infty}\int_{\mathbb{R}^{n}\setminus B(0,R)}|Xu_k|^{p}dx\right|\leq C\varepsilon.\]
Letting $\varepsilon\to 0$, we conclude that
\[ \mu_{\infty}=\lim_{R\to+\infty}\limsup_{k\to \infty}\int_{\mathbb{R}^{n}\setminus B(0,R)}|Xu_k|^{p}dx=\lim_{R\to+\infty}\limsup_{k\to \infty}\int_{\mathbb{R}^{n}\setminus B(0,R)}|Xv_k|^{p}dx. \]
Observing that
\begin{equation*}
 \int_{\mathbb{R}^n\setminus B(0,R+1)}|Xv_{k}|^{p}dx\leq\int_{\mathbb{R}^n}|\phi_{R}|^{p}|Xv_{k}|^{p}dx\leq \int_{\mathbb{R}^n\setminus B(0,R)}|Xv_{k}|^{p}dx,
\end{equation*}
we further have
\begin{equation}\label{4-8}
 \mu_{\infty}=\lim_{R\to+\infty}\limsup_{k\to \infty}\int_{\mathbb{R}^{n}\setminus B(0,R)}|Xv_k|^{p}dx=\lim_{R\to+\infty}\limsup_{k\to \infty}\int_{\mathbb{R}^{n}}|\phi_{R}|^{p}|Xv_k|^{p}dx.
\end{equation}
Combining \eqref{4-4}, \eqref{4-6} and \eqref{4-8}, we obtain conclusion (4).

We next prove conclusion (5). For every $\eta\in C_0(\mathbb{R}^n)$ with $\eta\geq 0$, Minkowski's inequality gives
\[
 \left|
 \left(\int_{\mathbb{R}^n}\eta^p|Xu_k|^p\,dx\right)^{\frac1p}
 -
 \left(\int_{\mathbb{R}^n}\eta^p|Xv_k|^p\,dx\right)^{\frac1p}
 \right|
 \leq
 \left(\int_{\mathbb{R}^n}\eta^p|Xu|^p\,dx\right)^{\frac1p}.
\]
Passing to the limit in the two measure terms yields
\[
 \left|
 \left(\int_{\mathbb{R}^n}\eta^p\,d\widetilde\mu\right)^{\frac1p}
 -
 \left(\int_{\mathbb{R}^n}\eta^p\,d\mu\right)^{\frac1p}
 \right|
 \leq
 \left(\int_{\mathbb{R}^n}\eta^p|Xu|^p\,dx\right)^{\frac1p}.
\]
Fix $j\in J$ and choose $\eta_{j,r}\in C_0^\infty(\mathbb{R}^n)$ with
$0\leq\eta_{j,r}\leq1$, $\eta_{j,r}(x_j)=1$, and
$\operatorname{supp}\eta_{j,r}\subset B_{\rm E}(x_j,r)$, where $B_{\rm E}(x,r)$ denotes the Euclidean ball. Letting $r\to 0$ in the preceding inequality gives
\[
 \widetilde\mu(\{x_j\})=\mu(\{x_j\}).
\]
Consequently, (3) implies
\[
 \widetilde\mu(\{x_j\})\geq C_0a_j^{\frac{p}{p_{Q}^{*}}}
 \qquad(j\in J).
\]

Since the map $w\mapsto Xw$ is continuous from
$\mathcal{M}_{X,0}^{p,Q}(\mathbb{R}^n)$ to $L^p(\mathbb{R}^n;\mathbb{R}^m)$, we have $Xu_k\rightharpoonup Xu$ weakly in $L^p(\mathbb{R}^n;\mathbb{R}^m)$. Thus, for every nonnegative $\psi\in C_0^\infty(\mathbb{R}^n)$, weak lower semicontinuity gives
\[
 \int_{\mathbb{R}^n}\psi|Xu|^p\,dx
 \leq
 \liminf_{k\to\infty}\int_{\mathbb{R}^n}\psi|Xu_k|^p\,dx
 =
 \int_{\mathbb{R}^n}\psi\,d\widetilde\mu.
\]
Hence $\widetilde\mu\geq|Xu|^pdx$ as measures. Let $S:=\{x_j\mid j\in J\}$. Since $S$ is countable, $|Xu|^pdx(S)=0$. Therefore, for every Borel set $E\subset\mathbb{R}^n$,
\[
 \begin{aligned}
 \widetilde\mu(E)
 &=\widetilde\mu(E\setminus S)+\widetilde\mu(E\cap S)\\
 &\geq \int_{E\setminus S}|Xu|^pdx
 +C_0\sum_{x_j\in E}a_j^{\frac{p}{p_{Q}^{*}}}=\int_E|Xu|^pdx+C_0\sum_{x_j\in E}a_j^{\frac{p}{p_{Q}^{*}}}.
 \end{aligned}
\]
This proves conclusion (5).

For any $\varphi\in C_{0}(\mathbb{R}^n)$ with $\varphi\geq 0$, by the Brezis-Lieb lemma we have
\[ \int_{\mathbb{R}^n}\varphi|u|^{p_{Q}^{*}}dx=\lim_{k\to \infty}\left(\int_{\mathbb{R}^n}\varphi|u_{k}|^{p_{Q}^{*}}dx-\int_{\mathbb{R}^n}\varphi|v_{k}|^{p_{Q}^{*}}dx \right),\]
which implies that $|u_{k}|^{p_{Q}^{*}}dx\rightharpoonup |u|^{p_{Q}^{*}}dx+\omega$ in $\mathscr{M}^{+}(\mathbb{R}^n)$. Thus, for every $R>1$, we have 
\begin{equation}
\begin{aligned}\label{4-9}
\limsup_{k\to\infty}\int_{\mathbb{R}^n}|u_k|^{p_{Q}^{*}}dx&=\limsup_{k\to\infty}\left(\int_{\mathbb{R}^n}\phi_{R}|u_k|^{p_{Q}^{*}}dx+\int_{\mathbb{R}^n}(1-\phi_{R})|u_k|^{p_{Q}^{*}}dx   \right)\\
&=\limsup_{k\to\infty}\int_{\mathbb{R}^n}\phi_{R}|u_k|^{p_{Q}^{*}}dx+\int_{\mathbb{R}^n}(1-\phi_{R})|u|^{p_{Q}^{*}}dx+\int_{\mathbb{R}^n}(1-\phi_{R})d\omega.
\end{aligned}
\end{equation}
Note that 
\[ \lim_{R\to+\infty}\limsup_{k\to\infty}\int_{\mathbb{R}^n}\phi_{R}|u_k|^{p_{Q}^{*}}dx=\lim_{R\to\infty}\limsup_{k\to\infty}\int_{\mathbb{R}^{n}\setminus B(0,R)}|u_k|^{p_{Q}^{*}}dx=\omega_{\infty},\]
\[ \lim_{R\to+\infty}\int_{\mathbb{R}^n}(1-\phi_{R})|u|^{p_{Q}^{*}}dx=\int_{\mathbb{R}^n}|u|^{p_{Q}^{*}}dx,\quad \lim_{R\to+\infty}\int_{\mathbb{R}^n}(1-\phi_{R})d\omega=\|\omega\|. \]
Letting $R\to+\infty $ in \eqref{4-9}, it follows that 
$ \limsup_{k\to\infty}\int_{\mathbb{R}^n}|u_k|^{p_{Q}^{*}}dx=\int_{\mathbb{R}^n}|u|^{p_{Q}^{*}}dx+\omega_{\infty}+\|\omega\|$, 
which gives conclusion (6).

Similarly, for every $R>1$, we have 
\begin{equation}
\begin{aligned}\label{4-10}
\limsup_{k\to\infty}\int_{\mathbb{R}^n}|Xu_k|^{p}dx&=\limsup_{k\to\infty}\left(\int_{\mathbb{R}^n}\phi_{R}|Xu_k|^{p}dx+\int_{\mathbb{R}^n}(1-\phi_{R})|Xu_k|^{p}dx   \right)\\
&=\limsup_{k\to\infty}\int_{\mathbb{R}^n}\phi_{R}|Xu_k|^{p}dx+\int_{\mathbb{R}^n}(1-\phi_{R})d\widetilde{\mu}.
\end{aligned}
\end{equation}
Recall that
\[ \lim_{R\to+\infty}\limsup_{k\to\infty}\int_{\mathbb{R}^n}\phi_{R}|Xu_k|^{p}dx=\lim_{R\to\infty}\limsup_{k\to\infty}\int_{\mathbb{R}^{n}\setminus B(0,R)}|Xu_k|^{p}dx=\mu_{\infty}, \]
and
$ \lim_{R\to+\infty}\int_{\mathbb{R}^n}(1-\phi_{R})d\widetilde{\mu}=\|\widetilde{\mu}\|$. 
Taking $R\to+\infty$ in \eqref{4-10}, we achieve conclusion (7).
\end{proof}

\subsection{Proofs of Theorem \ref{thm4} and Corollary \ref{corollary1-3}}

\begin{proposition}
\label{prop4-1}
Assume conditions (\hyperref[H1]{H.1}) and (\hyperref[H2]{H.2}) hold. Then, the maximal level set
$H=\{x\in \mathbb{R}^n\mid \nu(x)=Q\}$ is an embedded smooth submanifold of $\mathbb R^n$. Moreover, there exists a map 
$T:H\times \mathbb{R}^n\to \mathbb{R}^n$ such that 
\begin{itemize}
 \item $T\in C^{\infty}(H\times \mathbb{R}^n)$.
  \item For each $w\in H$, the map $T(w,\cdot)\in \mathcal{G}$ and satisfies $T(w,0)=w$.
\end{itemize}
\end{proposition}
\begin{proof}
\textbf{Step 1.} Let $\mathfrak g:=\operatorname{Lie}(X_1,\ldots,X_m)$ be the Lie algebra generated by $X=(X_1,\ldots,X_m)$. The $\delta_t$-homogeneity of degree $1$ of each $X_j$ implies that every commutator of length $k$ is $\delta_t$-homogeneous of degree $k$. Hence,
\[
\mathfrak g=\bigoplus_{k=1}^{\alpha_n}\mathfrak g_k,
\qquad
[\mathfrak g_i,\mathfrak g_j]\subset \mathfrak g_{i+j},
\]
where $\mathfrak g_k$ denotes the vector space spanned by commutators of length $k$, and $\mathfrak g_k=\{0\}$ for $k>\alpha_n$. It follows that $\mathfrak g$ is a finite-dimensional nilpotent Lie algebra. Moreover, $\operatorname{div}Y=0$ for all $Y\in\mathfrak g$, so the flow of every $Y\in\mathfrak g$ preserves Lebesgue measure.

\noindent
\textbf{Step 2.}
As $\dim\mathfrak g<+\infty$, the global version of the Third Fundamental Theorem of Lie (see, e.g., \cite[Theorem 3.15.1]{Varadarajan1984}) yields a unique, up to isomorphism, connected and simply connected Lie group $G$ whose Lie algebra is isomorphic to $\mathfrak g$. Moreover, every
$Y=\sum_{i=1}^n Y_i(x)\partial_{x_i}\in\mathfrak g$ has triangular polynomial form, in the sense that each coefficient $Y_i(x)$ depends only on the preceding variables $x_1,\dots,x_{i-1}$; hence its flow is complete. The Lie-Palais theorem \cite[Theorem 7.6]{Tuynman2016} therefore gives a unique smooth global left action
$\Phi:G\times\mathbb R^n\to\mathbb R^n$ such that
\[
\left.\frac{d}{dt}\right|_{t=0}
\Phi(\exp_G(-tY),x)=Y(x),
\qquad
Y\in\mathfrak g.
\]
Here, $\exp_G$ denotes the exponential map of the Lie group $G$.

We convert this left action into a right action by setting
\begin{equation}\label{right-G-action}
x\cdot g:=\Phi(g^{-1},x),
\qquad
x\in\mathbb R^n,\ g\in G.
\end{equation}
Then $(x,g)\mapsto x\cdot g$ defines a smooth global right action of $G$ on $\mathbb R^n$. Furthermore, for any $Y\in\mathfrak g$,
\[
\begin{aligned}
\left.\frac{d}{dt}\right|_{t=0}
x\cdot\exp_G(tY)=\left.\frac{d}{dt}\right|_{t=0}\Phi(\exp_G(tY)^{-1},x)=\left.\frac{d}{dt}\right|_{t=0}
\Phi(\exp_G(-tY),x)=Y(x).
\end{aligned}
\]

For $x\in\mathbb R^n$, define the orbit map $F_x:G\to\mathbb R^n$ by
$F_x(g)=x\cdot g$. For any $Y\in\mathfrak g$,
\[
(dF_x)_e(Y)
=
\left.\frac{d}{dt}\right|_{t=0}
x\cdot\exp_G(tY)
=
Y(x)
=
\operatorname{ev}_x(Y),
\]
where $e$ denotes the identity element of $G$. By Proposition \ref{prop2-3},
$\operatorname{ev}_x(\mathfrak g)=T_x\mathbb R^n$ for all $x\in\mathbb R^n$, and hence $(dF_x)_e$ is surjective. The submersion theorem then implies that the orbit
$\mathcal O_x:=x\cdot G$ contains an open neighborhood of $x$.

If $y\in\mathcal O_x$, then $y=x\cdot g$ for some $g\in G$, and therefore
\[
\mathcal O_y
=
y\cdot G
=
(x\cdot g)\cdot G
=
x\cdot(gG)
=
x\cdot G
=
\mathcal O_x.
\]
Applying the preceding argument at $y$, we see that $\mathcal O_x$ contains an open neighborhood of every point $y\in\mathcal O_x$. Thus $\mathcal O_x$ is open for all $x\in\mathbb R^n$. Since the orbits form a disjoint partition of the connected space $\mathbb R^n$, there is only one orbit. Consequently, the right $G$-action $(x,g)\mapsto x\cdot g$  defined by \eqref{right-G-action} is transitive.

\noindent
\textbf{Step 3.}
Let
\[
K:=\{g\in G\mid 0\cdot g=0\}
\]
be the stabilizer of the origin. We first note that $K$ is a closed subgroup of $G$. It is a subgroup because $0\cdot e=0$, and, whenever $g,h\in K$, $0\cdot(gh)=(0\cdot g)\cdot h=0\cdot h=0$, while $0\cdot g^{-1}=(0\cdot g)\cdot g^{-1}=0\cdot e=0$.
To prove closedness, let $F_0:G\to\mathbb R^n$ be the orbit map $F_0(g)=0\cdot g$. Then
$K=F_0^{-1}(\{0\})$. As $F_0$ is smooth and $\{0\}$ is closed in $\mathbb R^n$, the subgroup $K$ is closed in $G$. Cartan's closed subgroup theorem (see, e.g., \cite[Theorem 3.42]{Warner1983}) implies that $K$ is an embedded Lie subgroup of $G$.

We next identify the Lie algebra of $K$. We claim that
\begin{equation}\label{Lie-K}
\operatorname{Lie}(K)=\mathfrak k_0=\{Y\in\mathfrak g\mid Y(0)=0\}.
\end{equation}
Indeed, if $Y\in\operatorname{Lie}(K)$, then $\exp_G(tY)\in K$ for all $t\in\mathbb R$, which gives
$0\cdot\exp_G(tY)=0$ for all $t\in\mathbb R$. Differentiating at $t=0$, we obtain
\[
Y(0)
=
\left.\frac{d}{dt}\right|_{t=0}0\cdot\exp_G(tY)
=
0.
\]
Thus $Y\in\mathfrak k_0$.

Conversely, suppose that $Y\in\mathfrak k_0$, namely $Y(0)=0$. The curve
$\gamma(t):=0\cdot\exp_G(tY)$ satisfies
\[
\gamma'(t)=Y(\gamma(t)),
\qquad
\gamma(0)=0.
\]
The constant curve $\gamma_0(t)\equiv0$ solves the same ODE with the same initial value. By uniqueness of solutions, $\gamma(t)=0$ for all $t\in\mathbb R$. Hence
\[
0\cdot\exp_G(tY)=0
\qquad
\text{for all }t\in\mathbb R,
\]
so $\exp_G(tY)\in K$ for all $t$, and therefore $Y\in\operatorname{Lie}(K)$. This proves \eqref{Lie-K}.

\noindent
\textbf{Step 4.}
We identify the orbit of $0$ with the homogeneous space $K\backslash G$. Define
\begin{equation}\label{theta-A}
\Theta:K\backslash G\to\mathbb R^n,
\qquad
\Theta(Kg):=0\cdot g.
\end{equation}
This map is well-defined. Indeed, if $Kg=Kh$, then $h=kg$ for some $k\in K$, and hence
\[
0\cdot h=0\cdot(kg)=(0\cdot k)\cdot g=0\cdot g.
\]
The transitivity of the right $G$-action gives the surjectivity of $\Theta$. Moreover, if $0\cdot g=0\cdot h$, then
\[
0=(0\cdot g)\cdot h^{-1}=0\cdot(gh^{-1}),
\]
so $gh^{-1}\in K$. Hence $Kg=Kh$, and $\Theta$ is also injective. Thus $\Theta$ is bijective.

We now prove that $\Theta$ is a diffeomorphism. Let
\[
\Pi:G\to K\backslash G,
\qquad
\Pi(g):=Kg,
\]
be the quotient map. Since $K$ is a closed Lie subgroup of $G$, the space $K\backslash G$ carries the standard quotient smooth structure, and $\Pi$ is a smooth surjective submersion. Because $\Theta\circ\Pi=F_0$ and $F_0$ is smooth, the quotient smooth structure implies that $\Theta$ is smooth.

Consider the differential of $\Theta$ at the base point $Ke$. The differential of the quotient map at the identity is
\[
(d\Pi)_e:\mathfrak g=T_eG\to T_{Ke}(K\backslash G).
\]
The map $(d\Pi)_e$ is surjective, and $\ker(d\Pi)_e=T_eK=\operatorname{Lie}(K)=\mathfrak k_0$.
Therefore $(d\Pi)_e$ induces a canonical linear isomorphism
\[
\iota:\mathfrak g/\mathfrak k_0\to T_{Ke}(K\backslash G),
\qquad
\iota([Y]):=(d\Pi)_e(Y).
\]
Differentiating the identity $F_0=\Theta\circ\Pi$ at $e\in G$, we obtain
\[
(dF_0)_e=(d\Theta)_{Ke}\circ(d\Pi)_e.
\]
For every $Y\in\mathfrak g$,
\[
\begin{aligned}
(dF_0)_e(Y)
&=
\left.\frac{d}{dt}\right|_{t=0}F_0(\exp_G(tY))
=
\left.\frac{d}{dt}\right|_{t=0}0\cdot\exp_G(tY)
=
Y(0).
\end{aligned}
\]
Consequently,
\[
(d\Theta)_{Ke}\bigl((d\Pi)_eY\bigr)=Y(0),
\qquad
\forall Y\in\mathfrak g.
\]
Under the canonical identification
$T_{Ke}(K\backslash G)\simeq\mathfrak g/\mathfrak k_0$, the differential $(d\Theta)_{Ke}$ is represented by the induced evaluation map
\[
\overline{\operatorname{ev}}_0:
\mathfrak g/\mathfrak k_0\to T_0\mathbb R^n,
\qquad
\overline{\operatorname{ev}}_0([Y])=Y(0).
\]
This map is well-defined because $\ker\operatorname{ev}_0=\{Y\in\mathfrak g\mid Y(0)=0\}=\mathfrak k_0$.
Furthermore, $\operatorname{ev}_0(\mathfrak g)=T_0\mathbb R^n$, so the induced map $\overline{\operatorname{ev}}_0$ is an isomorphism. Hence
\[
(d\Theta)_{Ke}:T_{Ke}(K\backslash G)\to T_0\mathbb R^n
\]
is a linear isomorphism.

We next show that $d\Theta$ is an isomorphism at every point of $K\backslash G$. For $g\in G$, define
\[
\bar R_g:K\backslash G\to K\backslash G,
\qquad
\bar R_g(Kh):=Khg,
\]
and
\[
R_g:\mathbb R^n\to\mathbb R^n,
\qquad
R_g(x):=x\cdot g.
\]
The map $\bar R_g$ is well-defined: if $Kh=Kh'$, then $h'=kh$ for some $k\in K$, and hence $Kh'g=Kkhg=Khg$. Moreover, $\bar R_g$ is a diffeomorphism with inverse $\bar R_{g^{-1}}$, and $R_g$ is a diffeomorphism with inverse $R_{g^{-1}}$. By the definition of $\Theta$, we have the equivariance relation
\[
\Theta\circ\bar R_g=R_g\circ\Theta.
\]
Indeed, for every $Kh\in K\backslash G$,
\[
(\Theta\circ\bar R_g)(Kh)
=
\Theta(Khg)
=
0\cdot(hg),
\]
whereas
\[
(R_g\circ\Theta)(Kh)
=
R_g(0\cdot h)
=
(0\cdot h)\cdot g
=
0\cdot(hg).
\]
Differentiating $\Theta\circ\bar R_g=R_g\circ\Theta$ at the base point $Ke$ gives
\[
(d\Theta)_{Kg}\circ(d\bar R_g)_{Ke}
=
(dR_g)_0\circ(d\Theta)_{Ke}.
\]
Thus
\[
(d\Theta)_{Kg}
=
(dR_g)_0\circ(d\Theta)_{Ke}\circ
\bigl((d\bar R_g)_{Ke}\bigr)^{-1}.
\]
Since $(d\bar R_g)_{Ke}$, $(dR_g)_0$, and $(d\Theta)_{Ke}$ are all linear isomorphisms, it follows that
\[
(d\Theta)_{Kg}:T_{Kg}(K\backslash G)\to T_{0\cdot g}\mathbb R^n
\]
is a linear isomorphism for every $g\in G$. Thus $d\Theta$ is an isomorphism at every point of $K\backslash G$. By the inverse function theorem, $\Theta$ is a local diffeomorphism.

As $\Theta$ is both bijective and a local diffeomorphism, it is a global diffeomorphism. Indeed, for every $p\in K\backslash G$, there exist open neighborhoods
$U_p\subset K\backslash G$ and $V_p\subset\mathbb R^n$ such that
$\Theta|_{U_p}:U_p\to V_p$ is a diffeomorphism. By bijectivity, the global inverse $\Theta^{-1}$ agrees on $V_p$ with the smooth local inverse
$(\Theta|_{U_p})^{-1}:V_p\to U_p$. Hence $\Theta^{-1}$ is smooth locally around every point of $\mathbb R^n$, and therefore
\[
\Theta:K\backslash G\to\mathbb R^n
\]
is a diffeomorphism.

It remains in this step to prove that $K$ is connected. Since $G$ is a connected simply connected nilpotent Lie group, the exponential map
$\exp_G:\mathfrak g\to G$ is a global diffeomorphism. Hence
$G\simeq\mathbb R^{\dim\mathfrak g}$. In particular, the fundamental group $\pi_1(G)$ is trivial and the zeroth homotopy set $\pi_0(G)$ consists of a single point. On the other hand,
$K\backslash G\simeq\mathbb R^n$, so $\pi_1(K\backslash G)$ is also trivial. Since $K$ is a closed Lie subgroup of $G$, the projection $G\to K\backslash G$ is a locally trivial principal $K$-bundle. The fibration
$K\hookrightarrow G\to K\backslash G$ gives the exact sequence
\[
\pi_1(G)\longrightarrow\pi_1(K\backslash G)\longrightarrow\pi_0(K)\longrightarrow\pi_0(G).
\]
Thus $\pi_0(K)$ is trivial, and $K$ is connected.

\noindent
\textbf{Step 5.}
We now characterize $H$ in terms of the normalizer of $K$. For $1\le k\le\alpha_n$, set
\[
W_k:=\operatorname{span}\{\partial_{x_i}\mid \alpha_i=k\},
\qquad
W_{\le k}:=\bigoplus_{\ell=1}^k W_\ell,
\qquad
\mathfrak g_{\le k}:=\bigoplus_{\ell=1}^k\mathfrak g_\ell.
\]
For every $Y\in\mathfrak g_k$, one has $Y(0)\in W_k$. Therefore
\[
V_k(0)=\operatorname{ev}_0(\mathfrak g_{\le k})\subset W_{\le k},
\]
where $\operatorname{ev}_x$ denotes the evaluation map
$\operatorname{ev}_x(Y)=Y(x)$. Hence
$\nu_k(0)=\dim V_k(0)\leq \dim W_{\le k}$. Since
\[
Q=\sum_{i=1}^n\alpha_i
=
\sum_{k=1}^{\alpha_n} k\cdot \dim W_k
=
\alpha_n n-\sum_{k=1}^{\alpha_n-1}\dim W_{\le k},
\]
while
\[
\nu(0)
=
\sum_{k=1}^{\alpha_n}k(\nu_k(0)-\nu_{k-1}(0))
=
\alpha_n n-\sum_{k=1}^{\alpha_n-1}\nu_k(0),
\]
and $\nu(0)=Q$, we have
\[
\sum_{k=1}^{\alpha_n-1} (\dim W_{\le k}-\nu_k(0))=0.
\]
Thus
\begin{equation}\label{4-A-12}
\nu_k(0)=\dim W_{\le k},
\qquad
V_k(0)=W_{\le k},
\qquad
1\le k\le \alpha_n.
\end{equation}

For each $x\in\mathbb R^n$, note that
$T_x\mathbb R^n=W_{\leq k}\oplus W_{\leq k}^{\perp}=W_k\oplus W_k^{\perp}$. We may therefore define projections
$\Pi_{\le k}:T_x\mathbb R^n\to W_{\le k}$ and
$\Pi_k:T_x\mathbb R^n\to W_k$ by
\[
\Pi_{\le k}(Y(x))
:=
\Pi_{\le k}\left(\sum_{i=1}^n Y_i(x)\partial_{x_i}\right)
=
\sum_{{i:\alpha_i\le k}}Y_i(x)\partial_{x_i},
\]
and
\[
\Pi_k(Y(x))
:=
\Pi_k\left(\sum_{i=1}^n Y_i(x)\partial_{x_i}\right)
=
\sum_{{i:\alpha_i=k}}Y_i(x)\partial_{x_i}.
\]

For each $x\in\mathbb R^n$, define
$P_{k,x}:\mathfrak g_{\le k}\to W_{\le k}$ by
\[
P_{k,x}(Y):=\Pi_{\le k}(Y(x)).
\]
The decompositions
\[
\mathfrak g_{\le k}=\bigoplus_{\ell\le k}\mathfrak g_\ell,
\qquad
W_{\le k}=\bigoplus_{\ell\le k}W_\ell
\]
allow us to write $P_{k,x}$ as a block matrix
$(M_{j,\ell})_{1\le j,\ell\le k}$, where each block
$M_{j,\ell}:\mathfrak g_\ell\to W_j$ is given by
$Y\mapsto \Pi_j(Y(x))$ for $Y\in\mathfrak g_\ell$. Proposition \ref{prop2-2} gives $M_{j,\ell}=0$ for $j<\ell$, so $P_{k,x}$ is block lower-triangular. Moreover, the diagonal block
$M_{\ell,\ell}:\mathfrak g_\ell\to W_\ell$ is independent of $x$ and is given by
$Y\mapsto \Pi_\ell(Y(0))$.

At $x=0$, the map $P_{k,0}:\mathfrak g_{\le k}\to W_{\le k}$ is block diagonal. In view of \eqref{4-A-12}, $P_{k,0}$ is surjective; hence the diagonal blocks $M_{\ell,\ell}$, $1\le\ell\le k$, are surjective for all $x\in\mathbb R^n$, since they are independent of $x$. Consequently,
\[
P_{k,x}=\Pi_{\le k}\circ \operatorname{ev}_x
\]
is surjective for all $x\in\mathbb R^n$. It follows that
$\Pi_{\le k}(V_k(x))=W_{\le k}$ and
\[
\nu_k(x)=\dim V_k(x)\ge \dim W_{\le k}=\nu_k(0).
\]
Using \eqref{2-1}, we obtain
\[
\nu(x)=\alpha_n n-\sum_{k=1}^{\alpha_n-1}\nu_k(x)
\le
\alpha_n n-\sum_{k=1}^{\alpha_n-1}\nu_k(0)
=\nu(0)=Q.
\]
Therefore,
\begin{equation}\label{4-B-15}
x\in H
\qquad \Longleftrightarrow \qquad
\nu_k(x)=\nu_k(0)\quad \forall 1\leq k\leq \alpha_n.
\end{equation}
In particular, if $x\in H$, then
\[
\Pi_{\le k}\big|_{V_k(x)}:V_k(x)\to W_{\le k}
\]
is an isomorphism for $1\le k\le\alpha_n$.

For $x\in\mathbb R^n$, define the isotropy subalgebra
\[
\mathfrak k_x:=\{Y\in\mathfrak g\mid Y(x)=0\}.
\]
By Proposition \ref{prop2-3}, the evaluation map
$\operatorname{ev}_x:\mathfrak g\to T_x\mathbb R^n$ is surjective for every $x\in\mathbb R^n$, and hence
\begin{equation}\label{4-A-13}
\dim\mathfrak k_x=\dim\mathfrak g-n.
\end{equation}
We claim that
\begin{equation}\label{4-A-14}
x\in H\quad \Longleftrightarrow\quad \mathfrak k_x=\mathfrak k_0.
\end{equation}

Assume first that $x\in H$. Since
\[
\mathfrak k_0=\bigoplus_{k=1}^{\alpha_n}(\mathfrak k_0\cap\mathfrak g_k),
\]
it suffices to consider $Y\in\mathfrak k_0\cap\mathfrak g_k$. Then $Y(0)=0$, and homogeneity gives
\[
\Pi_{\le k}(Y(x))=0,
\qquad
x\in\mathbb R^n.
\]
Since $Y(x)\in V_k(x)$ and
$\Pi_{\le k}|_{V_k(x)}$ is an isomorphism, we get $Y(x)=0$. Thus
$\mathfrak k_0\subset\mathfrak k_x$. By \eqref{4-A-13}, these two spaces have the same dimension, and hence
$\mathfrak k_x=\mathfrak k_0$.

Conversely, suppose that $\mathfrak k_x=\mathfrak k_0$. Then, for every $1\leq k\leq\alpha_n$,
\[
\begin{aligned}
\nu_k(x)
&=
\dim\operatorname{ev}_x(\mathfrak g_{\le k})
=
\dim\mathfrak g_{\le k}
-\dim(\mathfrak g_{\le k}\cap\mathfrak k_x) \\
&=
\dim\mathfrak g_{\le k}
-\dim(\mathfrak g_{\le k}\cap\mathfrak k_0)
=
\dim\operatorname{ev}_0(\mathfrak g_{\le k})
=
\nu_k(0).
\end{aligned}
\]
Thus $\nu(x)=\nu(0)=Q$, and $x\in H$. This proves \eqref{4-A-14}.

Let
\begin{equation}\label{eq-prop41-N}
N_G(K):=\{a\in G\mid a^{-1}Ka=K\}
\end{equation}
be the normalizer of $K$ in $G$. If $x=0\cdot g$, then the stabilizer of $x$ is $g^{-1}Kg$, and hence
\begin{equation*}
\mathfrak k_x=\operatorname{Ad}_{g^{-1}}\mathfrak k_0.
\end{equation*}
Since $K$ is connected and $G$ is simply connected and nilpotent, connected Lie subgroups of $G$ are uniquely determined by their Lie algebras. Therefore, by \eqref{4-A-14},
\begin{equation}\label{eq-prop41-H-normalizer}
x=0\cdot g\in H
\quad\Longleftrightarrow\quad
g^{-1}Kg=K
\quad\Longleftrightarrow\quad
g\in N_G(K).
\end{equation}
Consequently,
\begin{equation}\label{eq-prop41-H-ThetaN}
H=\Theta(K\backslash N_G(K))=\{0\cdot a\mid a\in N_G(K)\}.
\end{equation}

The subgroup $N_G(K)$ is closed in $G$, and hence is an embedded Lie subgroup. Since $K\subset N_G(K)$ is a closed Lie subgroup, the natural inclusion
$K\backslash N_G(K)\hookrightarrow K\backslash G$ is an embedded smooth map. The diffeomorphism
$\Theta:K\backslash G\to\mathbb R^n$ then shows that
\[
H=\Theta(K\backslash N_G(K))
\]
is an embedded smooth submanifold of $\mathbb R^n$. We denote the restricted diffeomorphism by
\begin{equation}\label{eq-prop41-ThetaN}
\Theta_N:K\backslash N_G(K)\to H,
\qquad
\Theta_N(Ka):=0\cdot a.
\end{equation}

\noindent
\textbf{Step 6.}
For $a\in N_G(K)$, define
\[
\bar A_a:K\backslash G\to K\backslash G,
\qquad
\bar A_a(Kg):=Kag.
\]
If $Kg=Kh$, then $h=kg$ for some $k\in K$. Because $a\in N_G(K)$, we have $aka^{-1}\in K$, and hence
\[
Kah=Kakg=K(aka^{-1})ag=Kag.
\]
Thus $\bar A_a(Kh)=\bar A_a(Kg)$, so $\bar A_a$ is well-defined.

We next show that $\bar A_a$ is a diffeomorphism on $K\backslash G$. Recall that $K$ is a closed subgroup of $G$, that $K\backslash G$ is a smooth quotient manifold, and that
$\Pi:G\to K\backslash G$ is a smooth submersion. Let
\[
L_a:G\to G,
\qquad
L_a(g):=ag.
\]
Then $L_a$ is a smooth diffeomorphism of $G$, and
\[
(\bar A_a\circ\Pi)(g)
=
\bar A_a(Kg)
=
Kag
=
\Pi(ag)
=
(\Pi\circ L_a)(g).
\]
Thus $\bar A_a\circ\Pi=\Pi\circ L_a$. Since $\Pi\circ L_a$ is smooth and $\Pi$ is a surjective submersion, $\bar A_a$ is smooth. Moreover, $a^{-1}\in N_G(K)$, so $\bar A_{a^{-1}}$ is also well-defined and smooth. For every $Kg\in K\backslash G$,
\[
\bar A_{a^{-1}}(\bar A_a(Kg))
=
\bar A_{a^{-1}}(Kag)
=
Ka^{-1}ag
=
Kg,
\]
and similarly
\[
\bar A_a(\bar A_{a^{-1}}(Kg))
=
\bar A_a(Ka^{-1}g)
=
Kaa^{-1}g
=
Kg.
\]
Therefore $\bar A_a^{-1}=\bar A_{a^{-1}}$, and $\bar A_a$ is a diffeomorphism of $K\backslash G$.

For $a\in N_G(K)$, define the corresponding map on $\mathbb R^n$ by
\begin{equation}\label{A-a-map}
A_a:\mathbb R^n\to\mathbb R^n,
\qquad
A_a:=\Theta\circ \bar A_a\circ \Theta^{-1}.
\end{equation}
Then $A_a$ is a smooth diffeomorphism of $\mathbb R^n$.

We prove that $A_a$ is an automorphism of the vector fields $X_1,\ldots,X_m$. First, $A_a$ commutes with the right $G$-action on $\mathbb R^n$. Indeed, by transitivity of the right $G$-action in Step 4, every point $x\in\mathbb R^n$ can be written as $x=0\cdot g$ for some $g\in G$. Hence
\[
A_a(R_h(x))
=
A_a((0\cdot g)\cdot h)
=
A_a(0\cdot gh)
=
0\cdot a(gh)
=
0\cdot agh.
\]
On the other hand,
\[
R_h(A_a(x))
=
R_h(A_a(0\cdot g))
=
R_h(0\cdot ag)
=
(0\cdot ag)\cdot h
=
0\cdot agh.
\]
Therefore $A_a\circ R_h=R_h\circ A_a$ for all $h\in G$.

Recall that, for every $Y\in\mathfrak g$,
\[
Y(x)=\left.\frac{d}{dt}\right|_{t=0}x\cdot\exp_G(tY).
\]
Using the commutation relation $A_a\circ R_h=R_h\circ A_a$, we obtain, for every $x\in\mathbb R^n$ and $Y\in\mathfrak g$,
\[
\begin{aligned}
d(A_a)_x(Y(x))
&=
d(A_a)_x
\left(
\left.\frac{d}{dt}\right|_{t=0}
x\cdot\exp_G(tY)
\right)  \\
&=
\left.\frac{d}{dt}\right|_{t=0}
A_a\bigl(x\cdot\exp_G(tY)\bigr)  \\
&=
\left.\frac{d}{dt}\right|_{t=0}
A_a(x)\cdot\exp_G(tY)
=
Y(A_a(x)).
\end{aligned}
\]
Thus
\[
d(A_a)_x(Y(x))=Y(A_a(x)),
\qquad
x\in\mathbb R^n,\quad Y\in\mathfrak g.
\]
Taking $Y=X_j$ gives $(A_a)_*X_j=X_j$ for all $1\le j\le m$. Equivalently, for any
$f\in C^\infty(\mathbb R^n)$ and $1\le j\le m$,
\[
\begin{aligned}
X_j(f\circ A_a)(x)
&=
d(f\circ A_a)_x(X_j(x))
=
df_{A_a(x)}\bigl(d(A_a)_x(X_j(x))\bigr)  \\
&=
df_{A_a(x)}\bigl(X_j(A_a(x))\bigr)
=
(X_jf)(A_a(x))
=
\bigl((X_jf)\circ A_a\bigr)(x).
\end{aligned}
\]
Hence $A_a$ is an automorphism of the vector fields $X_1,\ldots,X_m$.

We next show that $A_a$ is volume-preserving, namely,
\[
A_a^*dx=dx.
\]
We begin with the invariance of the Lebesgue measure $dx$ under the right $G$-action. Let
$R_h:\mathbb R^n\to\mathbb R^n$ be the right action given by \eqref{right-G-action}, that is,
\[
R_h(x):=x\cdot h
\qquad
\forall h\in G.
\]
For any $Y\in\mathfrak g$, let
$\Phi_t^Y:\mathbb R^n\to\mathbb R^n$ be the flow of $Y$. The right action law gives
\[
\frac{d}{dt}\bigl(x\cdot\exp_G(tY)\bigr)
=
Y\bigl(x\cdot\exp_G(tY)\bigr),
\]
and hence $\Phi_t^Y=R_{\exp_G(tY)}$. Since $\operatorname{div}Y=0$,
\[
\frac{d}{dt}(\Phi_t^Y)^*dx
=
(\Phi_t^Y)^*((\operatorname{div}Y)dx)
=
0.
\]
As $(\Phi_0^Y)^*dx=dx$, we obtain $(\Phi_t^Y)^*dx=dx$.
The exponential map $\exp_G:\mathfrak g\to G$ is a global diffeomorphism, so every $h\in G$ can be written as $h=\exp_G(Y)$ for some $Y\in\mathfrak g$. Therefore
\[
R_h^*dx=dx,
\qquad
h\in G.
\]
Thus the Lebesgue measure $dx$ is invariant under the right $G$-action.

We now verify that $A_a^*dx$ is also invariant under the right $G$-action. As shown above,
$A_a\circ R_h=R_h\circ A_a$ for all $h\in G$. Consequently,
\[
\begin{aligned}
R_h^*(A_a^*dx)
&=
(A_a\circ R_h)^*dx
=
(R_h\circ A_a)^*dx
=
A_a^*(R_h^*dx)
=
A_a^*dx.
\end{aligned}
\]
Thus $A_a^*dx$ is right $G$-invariant.

Because $A_a$ is a smooth diffeomorphism of $\mathbb R^n$, there exists a smooth positive function
\[
J_a(x):=|\det J_{A_a}(x)|
\]
such that $A_a^*dx=J_a\,dx$. Using the right invariance of both $A_a^*dx$ and $dx$, we get
\[
(J_a\circ R_h)dx
=
R_h^*(J_a\,dx)
=
R_h^*(A_a^*dx)
=
A_a^*dx
=
J_a\,dx.
\]
By uniqueness of the Radon--Nikodym derivative with respect to Lebesgue measure,
\[
J_a\circ R_h=J_a
\qquad
\text{a.e. on }\mathbb R^n.
\]
Since $J_a\circ R_h$ and $J_a$ are continuous, the equality holds everywhere:
\[
J_a(x\cdot h)=J_a(x),
\qquad
x\in\mathbb R^n,\ h\in G.
\]
The right $G$-action on $\mathbb R^n$ is transitive, so $J_a$ is constant on $\mathbb R^n$. Therefore there exists a constant $c(a)>0$ such that
\[
A_a^*dx=c(a)\,dx.
\]

It remains to prove that $c(a)=1$. Consider the map
\[
\Psi_a:\mathbb R^n\to\mathbb R^n,
\qquad
\Psi_a(x):=A_a(x)\cdot a^{-1}.
\]
Then
\[
\Psi_a(0)
=
A_a(0)\cdot a^{-1}
=
(0\cdot a)\cdot a^{-1}
=
0.
\]
Moreover, because the right action preserves $dx$,
\[
\Psi_a^*dx
=
(R_{a^{-1}}\circ A_a)^*dx
=
A_a^*(R_{a^{-1}}^*dx)
=
A_a^*dx
=
c(a)\,dx.
\]
Evaluating this identity at the origin gives
\[
c(a)=|\det J_{\Psi_a}(0)|.
\]

We now compute $\det J_{\Psi_a}(0)$. Under the diffeomorphism
\[
\Theta:K\backslash G\to\mathbb R^n,
\qquad
\Theta(Kg)=0\cdot g,
\]
the map $A_a$ is induced by $\bar A_a(Kg)=Kag$, while the right action by $a^{-1}$ is induced by
\[
\bar R_{a^{-1}}(Kg)=Kga^{-1}.
\]
Hence $\Psi_a=R_{a^{-1}}\circ A_a$ corresponds on $K\backslash G$ to
\[
\bar\Psi_a(Kg)
=
Kaga^{-1}.
\]
In particular, $\bar\Psi_a(Ke)=Ke$.

Recall that
\[
T_{Ke}(K\backslash G)\simeq\mathfrak g/\mathfrak k_0.
\]
With respect to this identification, the differential of $\bar\Psi_a$ at $Ke$ is induced by $\operatorname{Ad}_a$. Indeed, for $Y\in\mathfrak g$,
\[
\bar\Psi_a(K\exp_G(tY))
=
Ka\exp_G(tY)a^{-1}
=
K\exp_G(t\,\operatorname{Ad}_aY).
\]
Therefore $J_{\Psi_a}(0)$ is identified with the induced linear map
\[
\overline{\operatorname{Ad}}_a:
\mathfrak g/\mathfrak k_0\to \mathfrak g/\mathfrak k_0,
\qquad
\overline{\operatorname{Ad}}_a([Y])=[\operatorname{Ad}_aY].
\]
This induced map is well-defined because $a\in N_G(K)$, and hence
$aKa^{-1}=K$. Passing to Lie algebras gives
\[
\operatorname{Ad}_a\mathfrak k_0=\mathfrak k_0.
\]

Write $a=\exp_G(Z)$ for some $Z\in\mathfrak g$. Then
\[
\operatorname{Ad}_a=\exp(\operatorname{ad}_Z).
\]
The nilpotency of $\mathfrak g$ implies that $\operatorname{ad}_Z$ is nilpotent; hence $\operatorname{Ad}_a$ is unipotent. Since
$\operatorname{Ad}_a\mathfrak k_0=\mathfrak k_0$, the induced map
$\overline{\operatorname{Ad}}_a$ on $\mathfrak g/\mathfrak k_0$ is also unipotent. Thus
\[
\det\overline{\operatorname{Ad}}_a=1,
\]
which yields
\[
|\det J_{\Psi_a}(0)|
=
|\det\overline{\operatorname{Ad}}_a|
=
1.
\]
Therefore $c(a)=1$ and $A_a^*dx=dx$. This proves that $A_a$ is volume-preserving.

\noindent
\textbf{Step 7.}
Step 6 shows that
\begin{equation}\label{4-A-23}
A_a\in\mathcal G\qquad\text{for all }a\in N_G(K).
\end{equation}
It remains to construct the map $T$. We first construct a smooth section of the quotient map
\[
\pi_N:N_G(K)\to K\backslash N_G(K),
\qquad
\pi_N(a):=Ka.
\]
Let
\[
\mathfrak n
:=
\operatorname{Lie}(N_G(K))
=
\{Z\in\mathfrak g\mid [Z,\mathfrak k_0]\subset\mathfrak k_0\}.
\]
Because $G$ is connected, simply connected, and nilpotent, the exponential map
$\exp_G:\mathfrak g\to G$ is a global diffeomorphism. For
$a=\exp_G Z$, one has
\[
\operatorname{Ad}_a=\exp(\operatorname{ad}_Z).
\]
Since $\operatorname{ad}_Z$ is nilpotent, the condition
$\operatorname{Ad}_a\mathfrak k_0=\mathfrak k_0$ is equivalent to
$[Z,\mathfrak k_0]\subset\mathfrak k_0$.
Thus
\[
N_G(K)=\exp_G(\mathfrak n).
\]
Therefore $N_G(K)$ is a connected simply connected nilpotent Lie subgroup of $G$. Since $K\triangleleft N_G(K)$, the quotient
$K\backslash N_G(K)$ is a connected simply connected nilpotent Lie group with Lie algebra $\mathfrak n/\mathfrak k_0$.

Choose a linear complement $\mathfrak m$ of $\mathfrak k_0$ in $\mathfrak n$, so that
\[
\mathfrak n=\mathfrak k_0\oplus\mathfrak m.
\]
Let $\ell:\mathfrak n/\mathfrak k_0\to\mathfrak m$ be the inverse of the canonical linear isomorphism
$\mathfrak m\to\mathfrak n/\mathfrak k_0$. Since the exponential maps of connected simply connected nilpotent Lie groups are global diffeomorphisms, define
\[
\sigma\bigl(\exp_{K\backslash N_G(K)}(\xi)\bigr)
:=
\exp_{N_G(K)}(\ell(\xi)),
\qquad
\xi\in\mathfrak n/\mathfrak k_0.
\]
Then $\sigma:K\backslash N_G(K)\to N_G(K)$ is smooth. Moreover, because
\[
\pi_N\circ\exp_{N_G(K)}
=
\exp_{K\backslash N_G(K)}\circ d(\pi_N)_e,
\]
we have
\begin{equation}\label{4-A-24}
\pi_N\circ\sigma=\operatorname{id}_{K\backslash N_G(K)}.
\end{equation}

For $w\in H$, define
\[
a(w):=\sigma(\Theta_N^{-1}(w)).
\]
The smoothness of $\sigma$ and $\Theta_N^{-1}$ implies that
$a:H\to N_G(K)$ is smooth. Moreover, if $\Theta_N^{-1}(w)=Kb$, then \eqref{4-A-24} gives $Ka(w)=Kb$, and therefore
\begin{equation}\label{eq-prop41-aw-origin}
0\cdot a(w)=0\cdot b=w.
\end{equation}

Finally, define
\[
\mathcal A:N_G(K)\times\mathbb R^n\to\mathbb R^n,
\qquad
\mathcal A(a,x):=A_a(x).
\]
We claim that $\mathcal A$ is smooth. To see this, define
\[
B:N_G(K)\times K\backslash G\to K\backslash G,
\qquad
B(a,Kg):=Kag.
\]
The map $B$ is well-defined and smooth, because
\[
B\circ(\operatorname{id}_{N_G(K)}\times\Pi)(a,g)=\Pi(ag),
\]
and
\[
\operatorname{id}_{N_G(K)}\times\Pi:
N_G(K)\times G\to N_G(K)\times K\backslash G
\]
is a surjective submersion. Since
\[
\mathcal A
=
\Theta\circ B\circ
(\operatorname{id}_{N_G(K)}\times\Theta^{-1}),
\]
we conclude that
\[
\mathcal A\in C^\infty(N_G(K)\times\mathbb R^n).
\]

Define
\[
T:H\times\mathbb R^n\to\mathbb R^n,
\qquad
T(w,x):=A_{a(w)}(x)=\mathcal A(a(w),x).
\]
Since $a:H\to N_G(K)$ and
$\mathcal A:N_G(K)\times\mathbb R^n\to\mathbb R^n$ are smooth, we obtain
\[
T\in C^\infty(H\times\mathbb R^n).
\]
For every fixed $w\in H$, $a(w)\in N_G(K)$, and hence
\[
T(w,\cdot)=A_{a(w)}\in\mathcal G.
\]
Finally, by \eqref{eq-prop41-aw-origin},
\[
T(w,0)=A_{a(w)}(0)=0\cdot a(w)=w.
\]
The proof of Proposition \ref{prop4-1} is complete.
\end{proof}
\begin{remark}
The proof of Proposition \ref{prop4-1} implies that the volume-preserving automorphism group $\mathcal{G}$  acts transitively on $H$. Indeed, by \eqref{4-A-23} we have $A_a\in \mathcal{G}$ for all $a\in N_G(K)$.  Let $p,q\in H$. According to \eqref{eq-prop41-H-ThetaN}, 
 there exist $a,b\in N_G(K)$ such that
$ p=0\cdot a$ and $q=0\cdot b$. Now, substituting $c:=ba^{-1}\in N_G(K)$, it follows that  $A_c\in\mathcal G$, and
\[A_c(p)=A_c(0\cdot a)=0\cdot (ca)=0\cdot(ba^{-1}a)=0\cdot b=q.\]
Consequently, $\mathcal G$ acts transitively on $H$. 
\end{remark}

We now turn to the proof of Theorem \ref{thm4}, which exploits the concentration-compactness Lemma \ref{lemma4-2} together with the Sobolev inequalities for H\"{o}rmander vector fields established in \cite{chen-chen-li2024}.

\begin{proof}[Proof of Theorem \ref{thm4}]
For $Q\geq 3$ and $1<p<Q$, we let $\{v_k\}_{k=1}^{\infty}\subset \mathcal{M}_{X,0}^{p,Q}(\mathbb{R}^n)$ be a minimizing sequence for the variational problem \eqref{1-13} such that
\begin{equation}
  \|v_k\|_{L^{p_{Q}^{*}}(\mathbb{R}^n)}=1,\qquad \|Xv_k\|_{L^{p}(\mathbb{R}^n)}^{p}\to C_{0}.
\end{equation}
Then, we introduce the L\'{e}vy concentration function
\begin{equation}\label{4-12}
 Q_{k}(\rho):=\sup_{w\in H}\int_{B(w,\rho)}|v_{k}|^{p_{Q}^{*}}dx,
\end{equation}
where $B(w,\rho)=\{y\in \mathbb{R}^n\mid d(y,w)<\rho\}$ is the subunit ball, and $H=\{x\in \mathbb{R}^n\mid\nu(x)=Q\}$.

Observe first that for each $k\in \mathbb{N}^{+}$, the function $\rho\mapsto Q_{k}(\rho)$ is increasing on $(0,+\infty)$.  By Proposition \ref{prop4-1}, for every $w\in H$ there is a volume-preserving automorphism $T(w,\cdot)\in\mathcal G$ with $T(w,0)=w$. Hence, $T(w,(B(0,r)))=B(w,r)$ and
\[
 |B(w,r_2)\setminus B(w,r_1)|=|B(0,r_2)\setminus B(0,r_1)|=(r_2^Q-r_1^Q)|B(0,1)|
\]
for $0<r_1<r_2$. The absolute continuity of the integral of $|v_k|^{p_Q^*}\in L^1(\mathbb{R}^n)$ indicates that
\[
0\le Q_k(r_2)-Q_k(r_1)\le \sup_{w\in H}\int_{B(w,r_2)\setminus B(w,r_1)} |v_k|^{p_Q^*} dx\to0
\]
as $r_2\to r_1$. Thus $Q_k\in C((0,+\infty))$. The same uniform volume estimate also gives that $ \lim_{\rho\to 0^+}Q_k(\rho)=0$.
 Since $0\in H$ and $\|v_k\|_{L^{p_Q^*}}=1$, we have $\lim_{\rho\to+\infty}Q_k(\rho)=1$. Therefore, there exists a $\rho_k>0$ such that $Q_{k}(\rho_{k})=\frac{1}{2}$.  For this fixed $\rho_k$, the supremum over $H$ is attained. Indeed, the map
\[
F_{k,\rho_k}(w):=\int_{B(w,\rho_k)} |v_k|^{p_Q^*}  dx=\int_{B(0,\rho_k)} |v_k(T(w,y))|^{p_Q^*}\,dy
\]
is continuous on $H$. Moreover, observing that $B(w,\rho_k)\subset \mathbb R^n\setminus B(0,d(0,w)-\rho_k)$, so $F_{k,\rho_k}(w)\to0$ as $d(0,w)\to\infty$, $w\in H$. Since $H$ is closed by the upper semi-continuity of $\nu$ and the closed subunit balls are compact by Proposition \ref{prop2-9}, the maximum is reached on $H\cap\overline{B(0,R)}$ for $R$ sufficiently large. Consequently, there exists $w_{k}\in H$ such that
\begin{equation}\label{4-13}
  \int_{B(w_{k},\rho_k)}|v_{k}|^{p_{Q}^{*}}dx=Q_{k}(\rho_k)=\frac{1}{2}.
\end{equation}

According to Proposition \ref{prop4-1}, $H=\{x\in \mathbb{R}^n\mid\nu(x)=Q\}$ is an embedded smooth submanifold of $\mathbb R^n$. Furthermore,  there exists a smooth map $T:H\times \mathbb{R}^n\to \mathbb{R}^n$ such that $T(w,\cdot)\in \mathcal{G}$ and $T(w,0)=w$ holds for all $w\in H$. Let 
\[ u_{k}:=v_{k}^{w_{k},\rho_k}=\rho_{k}^{\frac{Q-p}{p}}v_{k}(T(w_{k},\delta_{\rho_{k}}(x))). \] A direct calculation yields that 
\[ \int_{\mathbb{R}^n}|u_{k}|^{p_{Q}^{*}}dx=1,\quad \mbox{and}\quad  \int_{\mathbb{R}^n}|Xu_{k}|^{p}dx= \int_{\mathbb{R}^n}|Xv_{k}|^{p}dx\to C_{0}\qquad \mbox{as}~~k \to+\infty.\]

We now verify the normalization identity used in the compactness argument. Since $T(w_k,\cdot)$ is a volume-preserving isometry for the subunit metric and $\delta_{\rho_k}(B(0,1))=B(0,\rho_k)$, we have
\[
T(w_k,\delta_{\rho_k}(B(0,1)))=B(w_k,\rho_k).
\]
Thus Proposition \ref{prop2-4} and \eqref{4-13} give
\[
  \int_{B(0,1)}|u_{k}|^{p_{Q}^{*}}dx=\int_{B(w_k,\rho_k)}|v_k|^{p_Q^*}dx=\frac12.
\]
Moreover, Proposition \ref{prop2-8} implies that $\delta_{\rho_k}$ and $T(w_k,\cdot)$ are bijections from $H$ onto $H$. Hence, as $w$ ranges over $H$, the point $z=T(w_k,\delta_{\rho_k}(w))$ also ranges over $H$, and
$T(w_k,\delta_{\rho_k}(B(w,1)))=B(z,\rho_k)$. Therefore
\begin{equation}\label{4-14}
  \int_{B(0,1)}|u_{k}|^{p_{Q}^{*}}dx=\frac{1}{2}=\sup_{w\in H}\int_{B(w,1)}|u_{k}|^{p_{Q}^{*}}dx.
\end{equation}

We now prove the relative compactness of the normalized sequence. Let $\{u_{k_\ell}\}_{\ell=1}^{\infty}$ be an arbitrary subsequence of $\{u_k\}_{k=1}^{\infty}$. It still satisfies
\[
\|u_{k_\ell}\|_{L^{p_Q^*}(\mathbb R^n)}=1,\qquad \|Xu_{k_\ell}\|_{L^p(\mathbb R^n)}^p\to C_0,
\]
and the normalization condition \eqref{4-14}. Relabeling this arbitrary subsequence as $\{u_k\}_{k=1}^{\infty}$, it is enough to show that it contains a strongly convergent subsequence.

Since $1<p<Q$, the space $\mathcal M_{X,0}^{p,Q}(\mathbb R^n)$ is reflexive. By Proposition \ref{prop2-10} and \cite[Theorem 1.39]{Willem2012}, the bounded sequence  $\{u_k\}_{k=1}^{\infty}\subset \mathcal{M}_{X,0}^{p,Q}(\mathbb{R}^n)$ admits a subsequence (still denoted by $\{u_k\}_{k=1}^{\infty}$) such that 
\begin{enumerate}
  \item $u_k\to u$ a.e. on $\mathbb{R}^n$;
  \item $u_k\rightharpoonup u$ weakly in $\mathcal{M}_{X,0}^{p,Q}(\mathbb{R}^n)$;
  \item $|X(u_k-u)|^{p}dx\rightharpoonup \mu$ in $\mathscr{M}^{+}(\mathbb{R}^n)$;\
  \item $ |Xu_k|^{p}dx\rightharpoonup \widetilde{\mu}$ in $\mathscr{M}^{+}(\mathbb{R}^n)$;\
  \item $|u_k-u|^{p_{Q}^{*}}dx\rightharpoonup \omega$ in $\mathscr{M}^{+}(\mathbb{R}^n)$.
\end{enumerate}
On the other hand, since $u\in \mathcal{M}_{X,0}^{p,Q}(\mathbb{R}^n)$,  \eqref{1-13} indicates that 
\[ C_{0}^{\frac{1}{p}}\|u\|_{L^{p_{Q}^{*}}(\mathbb{R}^n)}\leq \|Xu\|_{L^{p}(\mathbb{R}^n)}\leq \liminf_{k\to \infty}\|Xu_{k}\|_{L^{p}(\mathbb{R}^n)}=C_{0}^{\frac{1}{p}},\]
which gives $\|u\|_{L^{p_{Q}^{*}}(\mathbb{R}^n)}\leq 1$.  Hence, once we establish that 
 $\|u\|_{L^{p_{Q}^{*}}(\mathbb{R}^n)}=1$, it follows that $u$ is a minimizer for the variational problem \eqref{1-13}.

We now prove that $\|u\|_{L^{p_{Q}^{*}}(\mathbb{R}^n)}=1$. Using Lemma \ref{lemma4-2} and \eqref{1-13}, 
\begin{equation*}
\begin{aligned}
C_{0}&=\lim_{k\to\infty}\|Xu_{k}\|_{L^{p}(\mathbb{R}^n)}^{p}= \mu_{\infty}+\|\widetilde{\mu}\|\geq  C_{0}\omega_{\infty}^{\frac{p}{p_{Q}^{*}}}+ \|Xu\|_{L^{p}(\mathbb{R}^n)}^{p}+C_{0}\sum_{j\in J}a_{j}^{\frac{p}{p_{Q}^{*}}}\\
&\geq C_{0}\left(\omega_{\infty}^{\frac{p}{p_{Q}^{*}}}+\left(\|u\|_{L^{p_{Q}^{*}}(\mathbb{R}^n)}^{p_{Q}^{*}}\right)^{\frac{p}{p_{Q}^{*}}}+\sum_{j\in J}a_{j}^{\frac{p}{p_{Q}^{*}}}\right),
\end{aligned}
\end{equation*}
so
\begin{equation}\label{4-15}
  \omega_{\infty}^{\frac{p}{p_{Q}^{*}}}+\left(\|u\|_{L^{p_{Q}^{*}}(\mathbb{R}^n)}^{p_{Q}^{*}}\right)^{\frac{p}{p_{Q}^{*}}}+\sum_{j\in J}a_{j}^{\frac{p}{p_{Q}^{*}}}\leq 1. 
\end{equation}
Owing to Lemma \ref{lemma4-2}, we have
\begin{equation}\label{4-16}
  1=\lim_{k\to\infty}\|u_{k}\|_{L^{p_{Q}^{*}}(\mathbb{R}^n)}^{p_{Q}^{*}}=\|u\|_{L^{p_{Q}^{*}}(\mathbb{R}^n)}^{p_{Q}^{*}}+\|\omega\|+\omega_{\infty}=\|u\|_{L^{p_{Q}^{*}}(\mathbb{R}^n)}^{p_{Q}^{*}}+\sum_{j\in J}a_{j}+\omega_{\infty}.
\end{equation}
Combining \eqref{4-15}, \eqref{4-16} and the fact $0<\frac{p}{p_{Q}^{*}}<1$, we deduce that each of $\|u\|_{L^{p_{Q}^{*}}(\mathbb{R}^n)}^{p_{Q}^{*}}$, $\omega_{\infty}$ $a_{j}~(j\in J)$ must be either $0$ or $1$. But from \eqref{4-1} and \eqref{4-14} we know  $\omega_{\infty}\leq \frac{1}{2}$, and thus $\omega_{\infty}=0$. If $a_{j_{0}}=1$ for some $j_0\in J$, then $a_{j}=0$ for all $j\in J\setminus\{j_0\}$ and $u=0$. In this case, the measure $\omega$ is concentrated at a single point $z_{0}\in \mathbb{R}^n$, i.e., $\omega=\delta_{z_{0}}$.

We next show that if $\omega=\delta_{z_{0}}$, then  $z_{0}\in H$. Assume, for contradiction, that
$z_{0}\notin H$. Since $\nu(z_0)<Q$ and $\nu(\cdot)$ is an upper semi-continuous function, we can 
choose  $\varepsilon_0>0$ such that $\overline{B(z_0,\varepsilon_{0})}\cap H=\varnothing$. Letting 
\[ \tilde{\nu}_{z_{0}}:=\max_{x\in \overline{B(z_0,\varepsilon_{0})}}\nu(x), \]
it follows that $\tilde{\nu}_{z_{0}}<Q$. Then\\ 
\emph{\textbf{Case 1:}} $1<p\leq\tilde{\nu}_{z_{0}}$. Set
\[ s:=\left\{
       \begin{array}{ll}
        \frac{p\tilde{\nu}_{z_0}}{\tilde{\nu}_{z_0}-p} , & \hbox{if $1<p<\tilde{\nu}_{z_{0}}$;} \\[3mm]
        \frac{pQ}{Q-p}+1 , & \hbox{if $p=\tilde{\nu}_{z_{0}}$.}
       \end{array}
     \right.\]
Clearly, we have $s>\frac{pQ}{Q-p}$.  For any $0<\varepsilon\leq \frac{1}{2}\varepsilon_0$, by H\"{o}lder's inequality we have
\begin{equation}\label{4-17}
\int_{B(z_0,\varepsilon)}|u_{k}|^{\frac{pQ}{Q-p}}dx\leq \left(\int_{B(z_0,\varepsilon)}|u_{k}|^{s}dx\right)^{\frac{pQ}{s(Q-p)}}\left(\int_{B(z_0,\varepsilon)}1dx\right)^{1-\frac{pQ}{s(Q-p)}}.
\end{equation}
Choose a cut-off function $\phi\in C_{0}^{\infty}(\mathbb{R}^n)$ such that $0\leq \phi\leq 1$, $\phi\equiv1$ on $B(z_0,\frac{\varepsilon_{0}}{2})$, ${\rm supp}~\phi \subset\subset B(z_0,\varepsilon_{0})$. It then follows that $\|X\phi\|_{L^{\infty}(\mathbb{R}^n)}\leq C$. Note that $\|u_k\|_{\mathcal{M}_{X,0}^{p,Q}(\mathbb{R}^n)}\leq C$ for any $k\geq 1$. The Sobolev inequality in \cite{chen-chen-li2024} gives that
\begin{equation}\label{4-18}
\begin{aligned} \int_{B(z_0,\varepsilon)}|u_{k}|^{s}dx&\leq \int_{B(z_0,\varepsilon_{0})}|\phi u_{k}|^{s}dx\leq C\|X(\phi u_k)\|_{L^{p}(B(z_0,\varepsilon_{0}))}^{s}\\
&\leq C\left(\|u_kX\phi\|_{L^p(B(z_0,\varepsilon_{0}))}+\|\phi Xu_{k}\|_{L^p(B(z_0,\varepsilon_{0}))}  \right)^{s}\\
&\leq C\left(\|u_k\|_{L^p(B(z_0,\varepsilon_{0}))}+\|Xu_{k}\|_{L^p(B(z_0,\varepsilon_{0}))}  \right)^{s}\\
&\leq C\left(|B(z_0,\varepsilon_0)|^{\frac{1}{Q}}\|u_k\|_{L^{p_{Q}^{*}}(B(z_0,\varepsilon_{0}))}+\|Xu_{k}\|_{L^p(B(z_0,\varepsilon_{0}))}  \right)^{s}\leq C,
\end{aligned}
\end{equation}
where $C>0$ is a positive constant independent of $\varepsilon$ and $k$. Thus, for any $0<\varepsilon<\frac{1}{2}\varepsilon_0$ and $1<p\leq\tilde{\nu}_{z_{0}}$, we deduce from \eqref{4-17} and \eqref{4-18} that
\begin{equation}\label{4-19}
 \int_{B(z_0,\varepsilon)}|u_{k}|^{\frac{pQ}{Q-p}}dx\leq C|B(z_0,\varepsilon)|^{1-\frac{pQ}{s(Q-p)}}\qquad \forall~ k\geq 1. 
\end{equation}
\noindent\emph{\textbf{Case 2:}}~$\tilde{\nu}_{z_{0}}<p<Q$. For any $0<\varepsilon\leq \frac{1}{2}\varepsilon_0$, The Sobolev inequality in \cite{chen-chen-li2024} and estimate \eqref{4-18} give that
\[ \|u_{k}\|_{L^{\infty}(B(z_0,\varepsilon))}\leq\|\phi u_{k}\|_{L^{\infty}(B(z_0,\varepsilon_0))}\leq C\|X(\phi u_k)\|_{L^{p}(B(z_0,\varepsilon_0))}\leq C,\]
where $C>0$ is a positive constant independent of $\varepsilon$ and $k$. Hence,
\begin{equation}\label{4-20}
\int_{B(z_0,\varepsilon)}|u_{k}|^{\frac{pQ}{Q-p}}dx\leq C|B(z_0,\varepsilon)|. 
\end{equation}
As a result of \eqref{4-19} and \eqref{4-20}, we can choose a small $\varepsilon>0$ such that 
\[ \int_{B(z_0,\varepsilon)}|u_{k}|^{\frac{pQ}{Q-p}}dx\leq \frac{1}{2}\qquad \forall~ k\geq 1. \]
This contradicts the fact that $|u_{k}|^{p_{Q}^{*}}dx\rightharpoonup \omega=\delta_{z_{0}}$ in $\mathscr{M}^{+}(\mathbb{R}^n)$. Consequently, $\omega=\delta_{z_{0}}$ forces that $z_{0}\in H$.

Going back to \eqref{4-14}, we derive a contradiction for the case of $\omega=\delta_{z_{0}}$, 
\begin{equation}
  \int_{B(0,1)}|u_{k}|^{p_{Q}^{*}}dx=\frac{1}{2}=\sup_{w\in H}\int_{B(w,1)}|u_{k}|^{p_{Q}^{*}}dx\geq \int_{B(z_{0},1)}|u_{k}|^{p_{Q}^{*}}dx\to 1.
\end{equation}
Thus, we conclude from \eqref{4-16} that $\omega=0$ and $\|u\|_{L^{p_{Q}^{*}}(\mathbb{R}^n)}=1$, which means that $u$ is a minimizer.

 The Brezis--Lieb lemma yields
\[ \int_{\mathbb{R}^n}|u|^{p_{Q}^{*}}dx=\lim_{k\to \infty}\left(\int_{\mathbb{R}^n}|u_{k}|^{p_{Q}^{*}}dx-\int_{\mathbb{R}^n}|u_k-u|^{p_{Q}^{*}}dx \right),\]
so $\lim_{k\to\infty}\int_{\mathbb{R}^n}|u_k-u|^{p_{Q}^{*}}dx=0$. Moreover, since the map $w\mapsto Xw$ is continuous from
$\mathcal{M}_{X,0}^{p,Q}(\mathbb{R}^n)$ to $L^p(\mathbb{R}^n;\mathbb{R}^m)$, we have $Xu_k\rightharpoonup Xu$ in $L^p(\mathbb{R}^n;\mathbb{R}^m)$ and $\|Xu_k\|_{L^p(\mathbb{R}^n)}\to\|Xu\|_{L^p(\mathbb{R}^n)}$. Because $L^p(\mathbb{R}^n;\mathbb{R}^m)$ is uniformly convex for $1<p<\infty$, the Radon--Riesz property gives $Xu_k\to Xu$ strongly in $L^p(\mathbb{R}^n;\mathbb{R}^m)$. Hence, $u_k\to u$ in $\mathcal M_{X,0}^{p,Q}(\mathbb{R}^n)$. 

The preceding argument applies to every subsequence of the normalized sequence. Therefore every subsequence has a strongly convergent subsubsequence, and $\{v_k^{w_k,\rho_k}\}$ is relatively compact in $\mathcal M_{X,0}^{p,Q}(\mathbb{R}^n)$.
\end{proof}

\begin{proof}[Proof of Corollary \ref{corollary1-3}]
By Theorem \ref{thm4}, there exists a minimizer $u\in \mathcal{M}_{X,0}^{p,Q}(\mathbb{R}^n)$ for the variational problem \eqref{1-13} such that
\[
 \|u\|_{L^{p_Q^*}(\mathbb R^n)}=1,
 \qquad
 \int_{\mathbb R^n}|Xu|^pdx=C_0.
\]
Replacing $u$ by $|u|$, we may assume that $u$ is nonnegative. By the Lagrange multiplier rule, and by testing the resulting Euler--Lagrange equation with $u$ to identify the multiplier, we obtain
\[
 \int_{\mathbb R^n}|Xu|^{p-2}Xu\cdot X\varphi\,dx
 =C_0\int_{\mathbb R^n}u^{p_Q^*-1}\varphi\,dx
 \qquad
 \forall~\varphi\in C_0^\infty(\mathbb R^n).
\]
Set $u_0:=C_0^{\frac{1}{p_Q^*-p}}u$. 
Then $u_0\in \mathcal{M}_{X,0}^{p,Q}(\mathbb{R}^n)$ is a nontrivial nonnegative weak solution of \eqref{1-1-16}. Moreover,
\[
 \frac{u_0}{\|u_0\|_{L^{p_Q^*}(\mathbb R^n)}}=u,
\]
so the $L^{p_Q^*}$-normalization of $u_0$ is an extremal for \eqref{1-13}.

We next verify the least-energy property. Define the associated energy functional by
\[
 \mathcal E(v):=\frac1p\int_{\mathbb R^n}|Xv|^p\,dx
 -\frac1{p_Q^*}\int_{\mathbb R^n}|v|^{p_Q^*}\,dx,
 \qquad v\in\mathcal M_{X,0}^{p,Q}(\mathbb R^n).
\]
Let $v\in\mathcal{M}_{X,0}^{p,Q}(\mathbb R^n)$ be any nontrivial nonnegative weak solution of \eqref{1-1-16}. Testing the equation with $v$, by a standard density argument, gives
\[
 \int_{\mathbb R^n}|Xv|^p\,dx
 =\int_{\mathbb R^n}v^{p_Q^*}\,dx.
\]
Combining this identity with the optimal Sobolev inequality \eqref{1-13}, we obtain
\[
 C_0\left(\int_{\mathbb R^n}v^{p_Q^*}\,dx\right)^{\frac{p}{p_Q^*}}
 \leq \int_{\mathbb R^n}|Xv|^p\,dx
 =\int_{\mathbb R^n}v^{p_Q^*}\,dx,
\]
and hence
\[
 \int_{\mathbb R^n}|Xv|^p\,dx
 =\int_{\mathbb R^n}v^{p_Q^*}\,dx
 \geq C_0^{\frac{Q}{p}}.
\]
Since $\frac1p-\frac1{p_Q^*}=\frac1Q$, it follows that
\[
 \mathcal E(v)=\frac1Q\int_{\mathbb R^n}|Xv|^p\,dx
 \geq \frac1Q C_0^{\frac{Q}{p}}.
\]
On the other hand, the definition of $u_0$ and the identity
$\frac{p_Q^*}{p_Q^*-p}=\frac{Q}{p}$ yield
\[
 \int_{\mathbb R^n}|Xu_0|^p\,dx
 =\int_{\mathbb R^n}u_0^{p_Q^*}\,dx
 =C_0^{\frac{Q}{p}},
\]
and
\[
 \mathcal E(u_0)=\frac1Q C_0^{\frac{Q}{p}}
 =\inf\Bigl\{\mathcal E(v)\mid v\in\mathcal M_{X,0}^{p,Q}(\mathbb R^n)\setminus\{0\},\ 
 v\geq0,\ v\text{ is a weak solution of }\eqref{1-1-16}\Bigr\}.
\]
Thus, $u_0$ is a least-energy weak solution, that is, a ground state.

We next record the integrability and decay estimates. In the Carnot group setting, 
the proof of global integrability and optimal two-sided asymptotics in 
\cite{Loiudice2018} relies only on the global critical Sobolev inequality, 
Caccioppoli estimates for weak solutions, local boundedness and Harnack 
inequalities for nonlinear subelliptic equations, the doubling--Poincar\'e 
structure of subunit balls, and the homogeneity relations
\[
 d(\delta_t (x))=td(x),
 \qquad
 |B(0,r)|=r^Q|B(0,1)|.
\]
All these ingredients are available in the present setting, in view of 
\eqref{1-12}, Propositions \ref{prop2-5}-\ref{prop2-9}, and the local theory of 
Capogna--Danielli--Garofalo \cite{Capogna1993,Capogna1996-ajm}. By an argument analogous to that in \cite[Theorem 4.1]{Gandal-Loiudice-Tyagi2026},  the arguments of \cite{Loiudice2018} carry over to the present framework
and yield
\begin{equation*}
 u_0\in L^s(\mathbb R^n)
 \qquad \forall~s\in
 \left(\frac{Q(p-1)}{Q-p},\infty\right].
\end{equation*}
Additionally,  there exist constants $R_0>0$ and $c_1,c_2>0$ such that
\begin{equation*}
 c_1d(x)^{\frac{p-Q}{p-1}}
 \le u_0(x)\le
 c_2d(x)^{\frac{p-Q}{p-1}}
 \qquad \mbox{for}\qquad d(x)\ge R_0.
\end{equation*}

For any bounded open domain $U\subset \mathbb{R}^n$,  $u_0$ is a weak solution in 
$\mathcal{W}_{X}^{1,p}(U)$ satisfying
\[ \int_{U}|Xu_0|^{p-2}Xu_0\cdot X\varphi dx=\int_{U}V(x) u_0^{p-1}\varphi dx \qquad \forall \varphi\in \mathcal{W}_{X,0}^{1,p}(U), \]
where $V:=u_0^{p_{Q}^{*}-p}\in L^{\infty}(U)$. The local H\"older regularity theorem for nonlinear subelliptic equations with bounded lower-order coefficient, see \cite[Theorem 3.35]{Capogna1993}, implies that $u_0\in C_{\rm loc}^{\alpha}(U)$ for some $\alpha\in(0,1)$. Since $U\subset\subset\mathbb R^n$ is arbitrary, $u_0\in C(\mathbb R^n)$. 
Moreover, the Harnack inequality for nonnegative weak solutions, again from \cite[Theorem 3.1]{Capogna1993}, guarantees that  $u_0(x)>0$ for all $x\in \mathbb{R}^n$.

Finally, let $p=2$. Then
\[
 -\triangle_Xu_0=u_0^{2_{Q}^{*}-1}
 \qquad \mbox{in }\mathbb R^n,
\]
such that $u_0>0$ and $u_0\in C_{\rm loc}^{\alpha}(\mathbb R^n)$. The Rothschild--Stein hypoelliptic regularity \cite{stein1976}, followed by the standard bootstrap argument, gives $u_0\in C_{\rm loc}^{\infty}(\mathbb R^n)$. Therefore $u_0\in C^\infty(\mathbb R^n)$. This completes the proof.
\end{proof}

\subsection{Proof of Theorem \ref{thm5}}
Finally, we present the proof of Theorem \ref{thm5}.

\begin{proof}[Proof of Theorem \ref{thm5}]
Since $\mathcal{M}_{X,0}^{p,Q}(\Omega)\subset \mathcal{M}_{X,0}^{p,Q}(\mathbb{R}^n)$, it follows from \eqref{1-13} and \eqref{1-15} that $C_{0}\leq \widetilde{C_{0}}(\Omega)$. Let us choose $x_0\in \Omega\cap H$ and $r_0>0$ small enough such that the subunit ball $B_0=B(x_0,r_0)\subset \Omega$. The optimal Sobolev constant on $B_0$, corresponding to the exponent $\frac{Qp}{Q-p}$, is given by
\begin{equation}\label{4-25}
\widetilde{C_{0}}(B_0):=\inf_{u\in \mathcal{M}_{X,0}^{p,Q}(B_0),~\|u\|_{L^{p_{Q}^{*}}(B_0)}=1}\int_{B_0}|Xu|^{p}dx. 
\end{equation}
Therefore, we have $\widetilde{C_{0}}(\Omega)\leq \widetilde{C_{0}}(B_0)$. 

Next, we show that $\widetilde{C_{0}}(B_0)\leq C_{0}$, which implies $C_{0}=\widetilde{C_{0}}(\Omega)$. By \eqref{1-13}, for any $\varepsilon>0$, there exists a function $u_{\varepsilon}\in C_{0}^{\infty}(\mathbb{R}^n)$ such that 
\[ \int_{\mathbb{R}^n}|u_{\varepsilon}|^{p_{Q}^{*}}dx=1,\qquad \mbox{and}\qquad \int_{\mathbb{R}^n}|Xu_{\varepsilon}|^{p}dx<C_{0}+\varepsilon. \]
Since ${\rm supp}~u_{\varepsilon}$ is compact in $\mathbb{R}^n$, there exists $R_{0}>0$ such that ${\rm supp}~u_{\varepsilon}\subset B_0$. According to Proposition \ref{prop4-1}, there exists a smooth map $T:H\times \mathbb{R}^n\to \mathbb{R}^n$ such that 
$T(w,\cdot)\in \mathcal{G}$ and $T(w,0)=w$ for all $w\in H$. Let $T_{x_0}:=T(x_0,\cdot)\in \mathcal{G}$ and define
\[ v_{\varepsilon}(x):=\left(\frac{R_0}{r_{0}}\right)^{\frac{Q}{p_{Q}^{*}}}u_{\varepsilon}(T_{x_0}
(\delta_{\frac{R_{0}}{r_0}}(T_{x_0}^{-1}(x)))). \]
By means of Propositions \ref{prop2-4} and \ref{prop2-5}, we have  ${\rm supp}~v_{\varepsilon}\subset B_0$. Moreover, 
\[ \int_{B_0}|v_{\varepsilon}|^{p_{Q}^{*}}dx=\int_{\mathbb{R}^n}|u_{\varepsilon}|^{p_{Q}^{*}}dx=1,~~~~ \int_{B_0}|Xv_{\varepsilon}|^{p}dx=\int_{\mathbb{R}^n}|Xu_{\varepsilon}|^{p}dx<C_{0}+\varepsilon.\]
Letting $\varepsilon\to 0$, we obtain $\widetilde{C_{0}}(B_0)\le C_0$. Combining this with $C_0\le \widetilde{C_{0}}(\Omega)\le \widetilde{C_{0}}(B_0)$, we conclude that $\widetilde{C_{0}}(\Omega)=C_0$.
\end{proof}

\section*{Acknowledgements}
Hua Chen is supported by National Natural Science Foundation of China (Grant Nos. 12571249, 12221001 and 12131017) and National Key R\&D Program of China (no. 2022YFA1005602). Hong-Ge Chen is supported by National Natural Science Foundation of China (Grant No. 12201607) and Postdoctor Project of Hubei Province (Grant No. 2024HBBHXF095). Jin-Ning Li is supported by China National Postdoctoral Program for Innovative Talents (Grant No.  BX20230270) and China Postdoctoral Science Foundation (Grant No. 2024M762460). The authors would like to thank Fei Liu for helpful discussions
concerning Proposition \ref{prop4-1}.

\section*{Declarations}
On behalf of all authors, the corresponding author states that there is no conflict of interest. Also, no data was used for the research described in the article.

\end{document}